\documentclass[a4paper, 12pt, reqno]{amsart} 

\usepackage{amsmath,amsfonts,amssymb,amsthm}
\usepackage{enumerate,enumitem}

\usepackage{geometry}
\usepackage[utf8]{inputenc}
\usepackage[alphabetic,nobysame,abbrev]{amsrefs}
\newcommand{\citeref}[2]{\cite{#1}*{#2}}

\usepackage[hidelinks]{hyperref} 

\newif\ifarxiv
\arxivtrue 

\newcommand{\eq}[1]{\begin{equation}\label{eq:#1}}
\newcommand{\eqe}{\end{equation}}

\newcommand\eps{\varepsilon}

\newcommand\E{{\mathbb E}}
\newcommand\Ec[2]{{\E\left[#1\,\middle|\,#2\right]}}

\newcommand\Var{\operatorname{Var}}

\newcommand\Cov{\operatorname{Cov}}

\renewcommand\Pr{{\mathbb P}}
\newcommand\prob[1]{\Pr\left(#1\right)}
\newcommand\Pc[2]{{\Pr\left(#1\,\middle|\,#2\right)}}

\newcommand\pto{\overset{\mathrm{p}}{\to}}

\newcommand\floor[1]{\lfloor #1 \rfloor}
\newcommand\ceil[1]{\lceil #1 \rceil}

\newcommand\bigcpar[1]{\bigl\{#1\bigr\}}
\newcommand\Bigcpar[1]{\Bigl\{#1\Bigr\}}
\newcommand\biggcpar[1]{\biggl\{#1\biggr\}}

\newcommand\Bigsqpar[1]{\Bigl[#1\Bigr]}
\newcommand\bigpar[1]{\bigl(#1\bigr)}
\newcommand\Bigpar[1]{\Bigl(#1\Bigr)}
\newcommand\biggpar[1]{\biggl(#1\biggr)}

\newcommand\Bin{\operatorname{Bin}}

\newcommand{\RR}{\mathbb{R}}

\newcommand{\G}{{\mathbb G}}
\newcommand{\Gnp}{{\mathbb G}_{n,p}} 
\newcommand\bG{{\mathbb G}}

\newcommand\cB{\mathcal{B}}

\newcommand\cE{\mathcal{E}}

\newcommand\cG{{\mathcal G}}

\newcommand\cT{{\mathcal T}}
\newcommand\cU{\mathcal{U}}

\newcommand{\e}{{\mathrm{e}}}

\newcommand{\xx}{{\mathbf x}}

\newcommand{\aut}{\operatorname{aut}}

\newcommand{\oset}[2]{(#1)_{#2}} 
\newcommand{\randset}[2]{[#1]_{#2}} 
\newcommand{\neighs}{\randset{n}{p}} 
\newcommand{\neighsstar}{\randset{n^*}{p}} 

\newcommand{\refT}[1]{Theorem~\ref{#1}}
\newcommand{\refC}[1]{Corollary~\ref{#1}}
\newcommand{\refL}[1]{Lemma~\ref{#1}}
\newcommand{\refR}[1]{Remark~\ref{#1}}
\newcommand{\refS}[1]{Section~\ref{#1}}

\newcommand{\refP}[1]{Proposition~\ref{#1}}

\newcommand{\refApp}[1]{Appendix~\ref{#1}}

\newcommand{\justify}[1]{\fbox{\tiny{#1}}\quad}

\newtheorem{theorem}{Theorem}
\newtheorem{fact}[theorem]{Fact}

\newtheorem{proposition}[theorem]{Proposition}
\newtheorem{corollary}[theorem]{Corollary}
\newtheorem{lemma}[theorem]{Lemma}

\newtheorem{remark}{Remark}

\newtheorem{problem}{Problem}

\usepackage{xcolor} 

\let\OLDthebibliography\thebibliography
\renewcommand\thebibliography[1]{
  \OLDthebibliography{#1}
  \setlength{\parskip}{0pt}
  \setlength{\itemsep}{0pt plus 0.3ex}
}

\title[Maximum tree extension counts in random graphs]{Extreme local statistics in random graphs:\\ maximum tree extension counts}
\author{Pedro Ara\'ujo}
\address{Departamento de Matemática,
 Universidade Federal de Pernambuco,
Avenida Jornalista Aníbal Fernandes - Cidade Universitária, Recife, Brasil.
E-mail: {\tt pedrocampos.araujo\allowbreak@ufpe.br}}
\author{Simon Griffiths}
\address{Departamento de Matem\'atica, PUC-Rio, Rua Marqu\^{e}s de S\~{a}o Vicente 225, G\'avea, 22451-900 Rio de Janeiro, Brasil. E-mail: {\tt simon@puc-rio.br.} }
\author{Matas \v{S}ileikis}
\address{Institute of Computer Science of the Czech Academy of Sciences, Pod Vod\'{a}renskou v\v{e}\v{z}\'{\i} 2, 182 00 Prague, Czechia; Vilnius University, Faculty of Mathematics and
Informatics, Institute of Computer Science, Didlaukio g. 47, 08303 Vilnius, Lithuania. E-mail: {\tt matas@cs.cas.cz}.} 
\author{Lutz Warnke}
\address{Department of Mathematics, University of California San Diego, La Jolla CA~92093, USA. E-mail: {\tt lwarnke@ucsd.edu}.}

\thanks{PA and M\v{S} were supported by the Czech Science Foundation, project No. 20-27757Y, with institutional support RVO:67985807. Work was done while PA was affiliated with Institute of Computer Science of the Czech Academy of Sciences. SG was supported by FAPERJ (Proc.~201.194/2022) and by CNPq (Proc.~307521/2019-2).  LW was supported by NSF~CAREER grant~DMS-2225631 and a Sloan Research Fellowship.}
\date{October~27, 2023; revised January~16, 2026}

\begin{document}

\begin{abstract}
We consider maximum rooted tree extension counts in random graphs, i.e., we consider~$M_n = \max_v X_v$ where~$X_v$ counts the number of copies of a given tree in~$\Gnp$ rooted at vertex~$v$. 
We determine the asymptotics of~$M_n$  when the random graph is not too sparse, specifically when the edge probability~$p=p(n)$ satisfies~${p(1-p)n \gg \log n}$.  
The problem is more difficult in the sparser regime ${1 \ll pn\ll \log{n}}$, where we determine the asymptotics of~$M_n$ for specific classes of trees. 
Interestingly, here our large deviation type optimization arguments reveal that the behavior of~$M_n$ changes as we vary~${p=p(n)}$, due to different mechanisms that can make the maximum~large. 
\end{abstract}
\maketitle

  \section{Introduction}
In extreme value theory, the distribution of the maximum~$\max_{i \in [n]} X_i$ of~$n$ many i.i.d.\ random variables is a classical topic~\cites{LLR1983, EKM1997, BGTS2004}. 
In this paper we study a variant of this problem from random graph theory, where the relevant random \mbox{variables~${X_1, \ldots, X_n}$} are not independent (and may also depend on~$n$). 
Interestingly, for the binomial random graph~$\Gnp$ we (i)~discover that the qualitative behavior of the relevant maximum changes as we vary the edge probability~$p=p(n)$, 
and (ii)~identify the different underlying mechanisms that can make the corresponding maximum large.

The oldest extreme value problem 
in random graph theory concerns the maximum vertex~degree~${\Delta = \Delta(\Gnp)}$ over all~$n$ vertices of the binomial random graph $\Gnp$. 
Indeed, already in the 1960s it was folklore that for most edge probabilities~$p=p(n)$ of interest the maximum degree~$\Delta$ is concentrated around the mean degree $pn$, 
or more formally~that
\begin{equation*}
 \frac{\Delta}{pn} \,  \pto\,  1 \,,
\end{equation*} 
where~$\pto$ denotes convergence in probability. 
In the 1970s and 1980s the \mbox{asymptotic} behavior of the centered random variable~$\Delta-pn$ was established for many edge probabilities~$p=p(n)$ of interest, 
yielding in particular that, for\footnote{See Subsection~\ref{ss:asympt} for asymptotic notation}~$p(1-p) n \gg \log n$,
\begin{equation}\label{eq:deg:center:pto}
\frac{\Delta-pn}{\sqrt{2p(1-p)n\log n}} \,  \pto\,  1 \,,
\end{equation} 
where henceforth $\log$ stands for the natural logarithm. 
See~\cites{Ivchenko73,B80,Bollobas01} for more details about the maximum~degree~$\Delta$.

In this paper we investigate perhaps the simplest generalization of the maximum degree that is not well understood for the binomial random graph~$\Gnp$. Namely, given a rooted tree~$T$, for each vertex~$v \in [n]$ we define a local statistic~$X_v = X_{T,v}$ (deferring the precise definition for a while) that counts copies of~$T$ rooted at~$v$, 
write~$\mu_T:=\E X_{v}$ for the mean (which by symmetry does not depend on the choice of~$v$), 
and then study the 
maximum rooted tree extension~count
\begin{equation*}
M_n = M_{T,n} \, :=\, \max_{v \in [n]} X_{v} 
\end{equation*}
taken over all~$n$ vertices of the binomial random graph~$\Gnp$. To avoid trivial boundary cases in results throughout the paper we will implicitly assume that $T$ has at least one edge.
This includes the maximum degree problem as the simplest case where~$T$ is an edge, rooted at one end vertex. 
Conceptually, the main difficulty is that~$M_n$ is an extreme order statistic of~$n$ random variables~$X_v$ that are not independent (and their distribution depends on~$n$). 
As we shall discuss in \mbox{Sections~\ref{sec:res1} and \ref{sec:res2}}, the main results of this paper answer the following fundamental extreme value theory questions concerning~$M_n$, which are also interesting random graph theory questions in their own~right:
\begin{enumerate}[leftmargin=2.5em]
\item[(i)] For what range of edge probabilities~$p=p(n)$ does $M_n/\mu_T\pto 1$ hold, i.e., is the maximum~$M_n$ concentrated around the mean extension count~$\mu_T$?
\item[(ii)] When~$M_n/\mu_T\pto 1$ holds, what is the correct asymptotic expression for the centered random variable $M_n-\mu_T$, i.e., how much does the maximum~$M_n$ typically deviate from the mean extension count~$\mu_T$?
\item[(iii)] When~$M_n/\mu_T\pto 1$ fails, what is the asymptotic behavior of~$M_n$, i.e., can we identify a sequence~$\alpha_n$ such that~$M_n/\alpha_n\pto 1$ holds for certain trees? 
\end{enumerate}

We now provide some definitions omitted so far. 
Let~$\Gnp$ denote the random graph with vertex set $[n] := \{1, \dots, n\}$, 
in which each of the~$\binom{n}{2}$ possible edges is included independently with probability~$p=p(n)$;  see~\cites{G1959,ER60,Bollobas01,JLRbook}. 
Given a rooted graph~$H$ with root~$\rho$, we consider the maximum number of copies of~$H$ `rooted' at a vertex~${v\in [n]}$, counted with multiplicity\footnote{Using this convention we end up counting each copy of $H$ exactly~$\aut(H)$ times, where $\aut(H)$ is the number of automorphisms of rooted graph~$H$ which fix the~root.} to simplify the formulas. 
For each vertex~$v\in [n]$ we denote by~${X_v=X_{H,v}}$ the number of $H$-\emph{extensions of~$v$} in~$\Gnp$, i.e., $X_v$ denotes the number of injective functions~${\phi:V(H)\to [n]}$ such that~${\phi(\rho)=v}$ and~${\phi(u)\phi(w)\in E(\Gnp)}$ for every edge~${uw\in E(H)}$. 
Note that if~$H$ has~$v_H$ vertices and~$e_H$ edges, then\footnote{We henceforth denote by~$(n)_k$ the falling factorial $(n)_k=n(n-1)\dots (n-k+1)$, as usual.} 
\[
{\mu_H = \E X_v =(n-1)_{v_H-1}\, p^{e_H}} \sim n^{v_{H}-1}p^{e_H}.
\]
(Note that~$e_T=v_T-1$ for any tree~$T$.)

Maximum tree extension counts are a natural generalization of the maximum degree, 
but there is an important
further motivation 
from random graph theory. 
Indeed, a rooted graph~$H$ is called \textit{balanced} if, among all subgraphs~${F\subseteq H}$ containing the root vertex~$\rho$ and at least one edge, the maximizers of~${e_F/(v_F-1)}$ include~$H$ itself. 
Similarly, $H$ is called \textit{strictly balanced} if~${F=H}$ is the only maximizer (see \citeref{JLRbook}{Section 3.4}). 
As usual for such definitions, the first step is to study the maximum extension count~${M_n=M_{H,n}}$ for strictly balanced~$H$, which was successfully done in~\cites{S90,SW19}; see~\refS{eq:background} for more details.
To gain further insight into extension counts, 
the natural next step is to consider balanced~$H$, for which even the answer to question~(i) is only known for sufficiently large~$p=p(n)$; see~\cites{S90,SW19}.  
In this paper we address this fundamental  knowledge gap by studying questions~(i)~to~(iii) for the class of rooted trees, 
which are balanced\footnote{A rooted tree~$T$ is balanced but not strictly balanced, because ${e_F/(v_F-1)} = 1$ for all connected subgraphs~$F \subseteq T$ (since these are again trees).} but not strictly balanced and thus sit at the boundary between the~known and the~unknown.

\subsection{How and when does \texorpdfstring{$M_n$}{M\_n} concentrate around the mean?}\label{sec:res1}
In this section we answer the closely related questions~(i) and~(ii) concerning the case when the maximum tree extension count~$M_n$ is concentrated at the mean. 
In particular, we establish that~${M_n/\mu_T\pto 1}$ holds when~${p(1-p) n \gg \log n}$, 
in which case we also obtain the correct asymptotic expression for~${M_n-\mu_T}$ (see \refT{thm:max_ext}) 
and show that the maximum degree essentially determines the behavior of~$M_n$ (see Proposition~\ref{prop:degtoext}).
Furthermore, we demonstrate that the aforementioned range of edge probabilities~${p=p(n)}$ is best possible, since~${M_n/\mu_T\pto 1}$ fails when~$pn = \Theta(\log n)$ holds (see \refC{cor:breaks}). 

Turning to the details, set~$q:=1-p$ for brevity. 
As discussed, in the `dense' case $pqn\gg \log{n}$ the maximum degree $\Delta=\Delta(\Gnp)$ is known to satisfy the asymptotic behavior~\eqref{eq:deg:center:pto}; see also Proposition~\ref{prop:max_deg_dense}. 
Since the degree variance is asymptotically~$pqn$, in concrete words the  limit~\eqref{eq:deg:center:pto} says that `the maximum degree typically deviates from the mean degree by~$\sqrt{2 \log n}$ standard deviations'.
Our first main result shows that this strong concentration assertion in fact holds for all maximum rooted tree extension~counts. 
\begin{theorem}[Maximum for general trees, dense case] 
\label{thm:max_ext} 
Fix a rooted tree~$T$, with root degree~$a$. 
If~$pqn \gg \log n$, then 
    \begin{equation}
      \label{eq:LLNvariance}
      \frac{M_n\, -\, \mu_T}{\sigma_T \sqrt{2\log n} }\, \pto\, 1\,,
    \end{equation}
where the variance~$\sigma_T^2:=\Var X_v$ satisfies~$\sigma_T^2 \sim a^2\mu_T^2 q/(pn)$. 
\end{theorem}
En route to Theorem~\ref{thm:max_ext} we establish Proposition~\ref{prop:degtoext} below, which is stronger in two ways:
(a)~it allows us to understand the number of $T$-extensions of $v$ for \emph{every} vertex~$v$ as a simple function of the degree of~$v$, 
and (b)~it gives a smaller error. 
This in particular demonstrates that vertices of near-maximum degree determine the asymptotic behavior~\eqref{eq:LLNvariance} of~$M_n$. 
We remark that in the regime $pq n \ge C_0\log n$ the technical condition $C_0 > 1$ implies that the minimum degree is of order~$pn$, which somewhat simplifies our arguments. 
\begin{proposition}\label{prop:degtoext}
Fix a rooted tree~$T$, with root of degree~$a$. 
For any constant~$C_0>1$ there exists~$C>0$ such that the following holds. 
If~$pq n \ge C_0\log n$, then whp\footnote{Henceforth \emph{whp} stands for \emph{with high probability}, i.e., with probability tending to $1$ as~$n \to \infty$.} 
\begin{equation}
\label{eq:Yt_conc}
\max_{v \in [n]}\big| X_{T,v}\, -\, d(v)^{a}(pn)^{e_T - a}\big|\, \le\,   C (pn)^{e_T-1}\sqrt{\log{n}}, 
\end{equation}
where~$d(v)$ denotes the degree of vertex $v$ in the random graph~$\Gnp$. 
\end{proposition}

The following simple corollary of Proposition~\ref{prop:degtoext} (proved in Section~\ref{sec:proof_dense})  
shows that the maximum~$M_n$ 
is no longer concentrated around~$\mu_T$ when~$pn=\Theta(\log{n})$. 
\begin{corollary}
\label{cor:breaks}
Fix a rooted tree~$T$. 
There is a strictly decreasing function $f : (1, \infty) \to (1, \infty)$ such that the following holds. 
If~$pn \sim C \log n$ for some constant~$C > 1$, then
\begin{equation}\label{eq:cor:breaks}
    \frac{M_n}{\mu_T} \pto f(C) > 1 \,.
\end{equation}
\end{corollary}

\subsection{Can \texorpdfstring{$M_n$}{M\_n} behave differently for smaller edge probabilities?} 
\label{sec:res2}
Question~(iii) for maximum tree extension counts asks us to identify a sequence~$\alpha_n$ such that $M_n/\alpha_n\pto 1$, 
and in view of the results from \refS{sec:res1} we henceforth restrict our attention to the sparser regime $pn\ll \log{n}$. 
Here we expect that~$\alpha_n$ will be significantly larger than the mean~$\mu_T$, but it seems difficult to determine~$\alpha_n$ for general trees~$T$. 
The results in this section demonstrate that indeed new complexities emerge in the answer to question~(iii): 
for an interesting class of trees we establish that the form of~$\alpha_n$ changes for different ranges of~$p=p(n)$, 
and show that the behavior of~$M_n$ is not always associated with the maximum degree, i.e., our large deviation type optimization proofs reveal that there are several different mechanisms that can make~$M_n$ large
(see~Theorems~\ref{thm:symm_trees_abr}, \ref{thm:sparseP2}, and~Corollary~\ref{cor:Taone}).

To illustrate different types of behaviors that the maximum extension count~$M_n$ can have when~$pn\ll \log{n}$, 
we shall restrict our attention to \textit{spherically symmetric} trees~$T_{a,b}$ of height two, where the root has~$a$ neighbors (children), and each of these vertices has~$b$ children.  
To motivate the statement of \refT{thm:symm_trees_abr} below for~$T=T_{a,b}$, it is instructive to consider different possible `strategies' for finding a vertex $v$ for which the extension count~$X_v=X_{T,v}$ is particularly large.  
One natural strategy would be to select a vertex~$v$ of maximum degree~$\Delta$, 
which for~$1 \ll pn \ll \log n$ satisfies (see Proposition~\ref{prop:max_deg_asymp}):
\begin{equation}
  \label{eq:LLN_maxdeg_sparse}
  \frac{\Delta}{D} \pto 1 \qquad \text{with} \qquad D = D(n,p) :=  \frac{\log n}{\log\frac{\log n}{pn}}\,.
\end{equation}
It is plausible that most neighbors of such a vertex~$v$ should have degree $(1+o(1))pn$, 
so this strategy ought to produce a vertex~$v$ such that, whp,
\begin{equation}
  \label{eq:strategy1}
X_v\, \ge\, (1+o(1))D^a (pn)^{ab} \,. 
\end{equation}

Another possible strategy would be to select a vertex~$v$ which is adjacent to one or more vertices of large degree, i.e., of order~$D$. 
Since for any~$x \in (0,1)$ the `probability cost' associated with having degree around $xD$ turns out to be $n^{-x+o(1)}$ (see \refL{lem:n_alpha}), 
it is usually possible to find a vertex~$v$ which has, amongst its neighbors, a set of~$a$ neighbors which all have degrees of the form~$(1/a +o(1))D$. 
This strategy thus ought to produce a vertex~$v$ such that, whp,
\[
X_v\, \ge \, (1+o(1))a!\, \left(\frac{D}{a}\right)^{ab} . 
\]
In fact, the latter strategy is more versatile, as one may consider other sequences which sum to one in place of the sequence $(1/a,\dots ,1/a)$.
Intuitively, 
\begin{equation}
    \label{eq:Lambda}
\Lambda := \bigcup_{k \ge 1} \Bigcpar{(x_1,\dots ,x_k)\in [0,\infty)^k\, :\, \sum_{1 \le i \le k}x_i\, \le\, 1\, }  
\end{equation}
represents the space of possible implementations of this strategy. 
To evaluate the number of extensions each achieves, for nonnegative integers $a, b$ we define the function $f_{a,b}:\Lambda\to \RR$ by setting
\begin{equation}
\label{eq:deg_func}
f_{a,b}(x_1,\dots ,x_k)\, :=\, \sum_{\text{distinct } i_1, \ldots, i_a\in [k]}\quad  \prod_{j\in [a]}\, x_{i_j}^b \, . 
\end{equation}
Note that~$f_{a,b}$ is trivially zero whenever~$k < a$. 
Putting things together, for any vector~${\xx\in \Lambda}$ the discussed strategy ought to produce a vertex~$v$ such that, whp,
\begin{equation}
\label{eq:strategy2}
X_v\, \ge \, (1+o(1))f_{a,b}(\xx)\, D^{ab} \,. 
\end{equation}
Naturally, to get the best lower bound one should then optimize~$f_{a,b}$ over~$\Lambda$.

Our second main result determines the asymptotic behavior of~$M_n$ for spherically symmetric trees~$T=T_{a,b}$ with~$b\ge 2$: it shows that one of the two strategies discussed above is optimal over practically the whole sparse range~$1\ll pn\ll \log{n}$. 
\begin{theorem}[Maximum for spherically symmetric trees~$T_{a,b}$, simplified]
  \label{thm:symm_trees_abr} 
Fix~$T=T_{a,b}$, with $a\ge 1$ and~$b\ge 2$. 
If~$(\log{n}/\log{\log{n}})^{1-1/b}\ll pn \ll \log{n}$, then
\begin{equation}
\label{eq:sph_symm_denser}
\frac{M_n}{D^{a}(pn)^{ab}}\, \pto\, 1 \,. 
\end{equation}
If~$1\ll pn\ll (\log{n}/\log{\log{n}})^{1-1/b}$, then 
\begin{equation}
\label{eq:sph_symm}
\frac{M_n}{D^{ab}}\, \pto\, \sup_{\xx\in \Lambda} \, f_{a,b}(\xx) \,.
\end{equation}
\end{theorem}
\begin{remark} 
Theorem~\ref{thm:symm_trees} in Section~\ref{sec:symm_trees} shows that the correct asymptotics in the `missing range' $pn \asymp (\log{n}/\log{\log{n}})^{1-1/b}$ are obtained by a combination of the above strategies. 
\end{remark}
		
The most important spherically symmetric tree not covered by Theorem~\ref{thm:symm_trees_abr} is the two-edge path~${T = T_{1,1}=P_2}$ rooted at an end vertex. 
In this case \refT{thm:sparseP2} below shows that the first strategy `wins' whenever~$pn\gg \log{\log{n}}$, yielding~$M_n/Dpn\pto 1$. 
Here the second strategy only gives $\Theta(D)$ extensions (since~$a=b=1$), meaning that it is always inferior to the first strategy. 
Note that in the two strategies discussed above, the vertex~$v$ which maximizes the number of extensions is either itself of (near) maximum degree~$D$ or has neighbors of degree~$\Theta(D)$.
Interestingly, in the very sparse regime $pn \ll \log \log n$ a new optimal strategy emerges for paths~$P_2$ of length two: 
the proof of \refT{thm:sparseP2} below shows that the maximum number of extensions can be attained by a vertex of degree~$o(D)$, 
or, more precisely, by a vertex~$v$ whose degree is asymptotically equal~to 
\[
\frac{1}{\log\frac{\log{\log{n}}}{pn}}\, \cdot \, D\, \approx\, \frac{\log n}{(\log \log n) \log \frac{\log \log n}{pn}} \,, 
\]
and whose neighbors have, on average, degree about~$\log \log n$.
\begin{theorem}[Maximum for paths of length two, sparse case]\label{thm:sparseP2}
Let~$T = T_{1,1}$ be a path of length two rooted at one end vertex.  
If~$1 \ll pn \ll \log{n}$, then
\begin{equation*}
  \frac{M_n}{\alpha_n}\, \pto\, 1 \,,
\end{equation*}
where the sequence~$\alpha_n$ satisfies
\begin{equation*}
  \alpha_n \sim \begin{cases}
    Dpn  \quad &\text{if } \log{\log{n}}\ll pn \ll \log{n}\,, \\
    \frac{\log n}{\log\left( 1 +\frac{\log \log n}{pn} \right)}  \quad &\text{if } 1\ll pn=O(\log{\log{n}}) \,.
  \end{cases}
\end{equation*}
\end{theorem}
\begin{remark}
\refT{thm:paths} in \refS{sec:paths} extends \refT{thm:sparseP2} to paths~$P_m$ of fixed length~${m \ge 1}$, 
in which case the optimal strategy changes more often; see~the discussion below~\eqref{eq:km_lambda_kmminus}. 
\end{remark}

Finally, the following corollary of Theorem~\ref{thm:sparseP2} (proved in Section~\ref{sec:paths}) 
covers the remaining spherically symmetric trees~$T_{a,b}$ with~$b=1$ that were excluded so far. 
\begin{corollary}
  \label{cor:Taone}
Fix~$T = T_{a,1}$, with~$a \ge 1$. 
If~$1 \ll pn \ll \log n$, then
\begin{equation}\label{eq:Taone}
    \frac{M_n}{\alpha_n^a} \pto 1 \,,
\end{equation}
where the sequence~$\alpha_n$ is as in \refT{thm:sparseP2}. 
\end{corollary}

\subsection{Background and discussion}\label{eq:background}
The study of subgraph counts in random graphs has become an extremely well-established area of research~\cites{B81,R1988,JOR,JW2016,SW18,HMS2022}. 
Extension counts are an important and natural variant, which frequently arises in many probabilistic proofs and applications, 
including zero-one~laws in random graphs~\cites{SS1988,LS1991}, 
games on random graphs~\cites{LP2010}, random graph processes~\cites{BK2010,BW2019}, 
and random analogues of classical extremal and Ramsey results~\cites{SS2018,BK2019}.

For strictly balanced rooted graphs~$H$, the concentration of the maximum extension count~$M_n=M_{H,n}$ around the mean count~$\mu_H$ is fairly well understood.\footnote{We remark that~\cite{S90} and~\cite{SW19} both treat more general extension counts, where the root may consist of a set of vertices (while we focus on the case where the root consists of a single vertex).} 
Indeed, for such graphs Spencer~\cite{S90} proved that~${M_n/\mu_H \pto 1}$ when~$\mu_H\gg \log{n}$, and asked whether this condition on~$\mu_H$ was necessary.
\v{S}ileikis and Warnke~\cite{SW19} answered this question, by showing 
that~${M_n/\mu_H \pto 1}$ breaks down when~$\mu_H=\Theta(\log{n})$. 

For rooted graphs~$H$ that are not strictly balanced our understanding remains unsa\-tisfactory, and we are not aware of a general formula for the order of magnitude of~$M_n$ for all edge probabilities~$p=p(n)$ of interest. 
In particular, as pointed out in~\cite{SW19}, the behavior of the maximum extension count~$M_n$ can be quite different if $H$ is not strictly balanced.
For example, for graphs that are not balanced it can even happen\footnote{Consider for example a rooted triangle with a pendant edge added at one of the non-root vertices, 
with~$p=n^{-\gamma}$ for some $\gamma\in (3/4,1)$.
By rooting at a vertex of which is in a triangle of~$\Gnp$, it then is straightforward to see that whp~$M_n \ge \Theta(pn) = \Theta(n^{1-\gamma})$, despite~$\mu_H = \Theta(n^3 p^4) \ll 1$.}
that~$M_n = n^{\Theta(1)}$ whp despite~$\mu_H \ll 1$. 
In view of this, it remains an interesting open problem to find, for all rooted graphs,  a suitable sequence~$\alpha_n$ such that $M_n/\alpha_n$ is tight. 

In Sections~\ref{sec:res1} and \ref{sec:res2} we have seen that in some cases it is possible to find a sequence~$\alpha_n$ such that $M_n/\alpha_n \pto~1$, or, even better, to find yet another sequence~$\beta_n$ such that 
\begin{equation}
\label{eq:second_term}
\frac{M_n\, -\, \alpha_n}{\beta_n}\, \pto\, 1 \,.
\end{equation}
It remains an intriguing open problem to understand for which rooted graphs and for which edge probabilities such detailed results hold. 
To stimulate more research into this circle of problems, we propose the following modest generalization of \refT{thm:max_ext}.
\begin{problem}
  \label{pr:std_dev}
Let $H$ be a graph with root vertex $\rho$ and define
\[
\Phi_H := \min \{\mu_G : G \subseteq H, \rho \in V(G), e_G \ge 1 \}.
\]
Determine for which connected rooted graphs~$H$ the conditions $\Phi_H \gg \log n$ and~$p \ll 1$
together imply that
\begin{equation}\label{eq:pr:std_dev}
    \frac{M_n - \mu_H}{ \sigma_H \sqrt{2  \log n}} \pto 1\,.
\end{equation}
\end{problem}
To motivate why~\eqref{eq:pr:std_dev} is plausible for many rooted graphs, first note that in \mbox{Problem~\ref{pr:std_dev}} the assumption~$\Phi_H \gg \log n$ and the standard variance estimate~$\sigma_H^2 \asymp \mu_H^2/\Phi_H$ (see~\citeref{SW19}{eq.(10)}) together imply that~$\sigma_H \sqrt{2  \log n} \ll \mu_H$ holds, 
which intuitively means that~\eqref{eq:pr:std_dev} concerns the so-called `moderate deviations regime' of extension counts (where one often hopes for the tail behavior to be asymptotically normal). 
In view of the known fact that~$X_v$ is asymptotically normal (see~\citeref{SW19}{Claim~17(ii)}), 
the hope is then (i)~that for $x = \Theta(\sigma_H \sqrt{2\log n})$ the lower and upper tail satisfy $-\log \prob{X_v \le \mu_H - x} \sim x^2/(2\sigma_H^2)$ and $-\log \prob{X_v \ge \mu_H + x} \sim x^2/(2\sigma_H^2)$, 
and (ii)~that the union bound over the $n$ choices of~$v$ is not wasteful, 
which together lead to~\eqref{eq:pr:std_dev}.
In concrete words, \refT{thm:max_ext} shows that this heuristic is correct for rooted trees, and we believe that it is also correct for many other rooted graphs. 
However, the general case is somewhat delicate, since this heuristic can fail: see~\citeref{SW19}{Proposition~2} for a counterexample (where the upper tail behaves differently). 
An interesting generalization of Problem~\ref{pr:std_dev} concerns graphs with~$r$ root vertices, defining~$M_n$ as the maximum of extension counts over all $r$-tuples 
(here the natural denominator in~\eqref{eq:pr:std_dev} would then  be~$\sigma_H \sqrt{2r \log n}$, since the tail probability would need to beat the number of $r$-tuples, which is $(n)_r \sim n^r$). Recently results on such graphs have been obtained by Rodionov and Zhukovskii~\cite{RZ2023} for stars rooted on all leaves, for a wide range of $p = p(n)$, and also by Vakhrushev and Zhukovskii~\cite{vakhrushev2023maximum} for a larger class of rooted graphs, focusing on constant $p$.

\begin{remark} We believe our methods can be extended to derive similar results for all trees of depth two (but did not check all details).  We did not include these extensions because the statements become more complex and the proofs more technical.  If the degrees of the neighbors are given by $\mathbf{b}=(b_{1},\dots, b_a)$ where $2\le b_1\le b_2\le \dots \le b_a$, then the natural analogue of the function $f_{a,b}$ would be
\[
f_{\mathbf{b}}(x_1,\dots ,x_k)\, :=\, \sum_{\text{distinct } i_1, \ldots, i_a\in [k]}\quad  \prod_{j\in [a]}\, x_{i_j}^{b_j} \, .
\]
Let $B:=\sum_i b_i$. By adapting our arguments, we believe that the following can be shown: 
if
\( (\log{n}/\log{\log{n}})^{1-1/b_a}\ll pn \ll \log{n},
\)
then
\[
\frac{M_n}{D^{a}(pn)^{B}}\, \pto\, 1 \, ,
\]
and if~$1\ll pn\ll (\log{n}/\log{\log{n}})^{1-1/b_1}$, then
\[
\frac{M_n}{D^{B}}\, \pto\, \sup_{\xx\in \Lambda} \, f_{\mathbf{b}}(\xx) \,.
\]
However, in the missing range $(\log{n}/\log{\log{n}})^{1-1/b_1}\le pn\le (\log{n}/\log{\log{n}})^{1-1/b_a}$ there may be a number of intermediate regimes in which different intermediate strategies are optimal.  
This illustrates that the problem is surprisingly subtle even for trees of depth two. 
For trees of depth three or more, we believe that further technical difficulties emerge.
\end{remark}

Limit ~\eqref{eq:second_term} is equivalent to $(M_n - \alpha_n - \beta_n)/\beta_n \pto 0$. Occasionally it is possible to further refine this by showing that for some sequence $\gamma_n \ll \beta_n$ the random variable ${(M_n-\alpha_n-\beta_n)/\gamma_n}$ converges to a nondegenerate distribution. 
For the maximum degree, such results were obtained by Ivchenko~\cite{Ivchenko73} and Bollob{\'a}s~\cite{B80} in the 1970s and 1980s; see also~\cites{RZ2023, vakhrushev2023maximum}. 
For rooted cliques, such a result was recently obtained by Isaev, Rodionov, Zhang and Zhukovskii~\cite{Isaev2021extremal},  
for a restricted range of edge probabilities~${p=p(n)}$ that assumes $\Phi_{H} \ge n^{\Omega(1)}$ (rather than the natural target range proposed by \mbox{Problem~\ref{pr:std_dev}}), using a conditional maximization argument that differs from our approach.  
As a simple corollary of Proposition~\ref{prop:degtoext} (proved in Section~\ref{sec:proof_dense}), 
for rooted trees we also obtain such a limiting distribution result 
for most edge probabilities~${p=p(n)}$ of interest. 
\begin{corollary}[Maximum for general trees: Gumbel distribution]
\label{cor:Gumbel}
Fix a rooted tree~$T$. If $pqn \gg (\log n)^2$, then 
there exist positive sequences~$a_n$ and~$b_n$ 
such that~${(M_n - a_n)/b_n}$ converges in distribution to the standard Gumbel~distribution: for any fixed~$x \in \RR$, we~have 
\begin{equation}\label{eq:cor:Gumbel}
\lim_{n \to \infty} \prob{\frac{M_n - a_n}{b_n} \le x}  = \exp\bigpar{-\e^{-x}}.
  \end{equation}
\end{corollary}
In the aforementioned results and Corollary~\ref{cor:Gumbel}, the limit is the standard Gumbel distribution. Recent work~\cite{vakhrushev2023maximum} demonstrates that other limiting distributions can arise.

Natural refinements of the maximum extension counts are order statistics, that is, the $m$-th largest extension count for $m = 1, 2, \dots$. In the case of vertex degrees, results for such statistics were obtained in \cite{B80}, \cite{Ivchenko73}, \cite{RZ2023}. We have omitted this direction, but remark that some results can be obtained, provided $p$ does not tend to $0$ or $1$ too fast, by combining Proposition~\ref{prop:degtoext} and results on the degree sequence from, say,  \citeref{Bollobas01}{Chapter~3}.  

\subsection{Organization of the paper}
%
In Section~\ref{sec:prelim} we present some preliminary random graph theory results. 
In Section~\ref{sec:proof_dense} we prove Theorem~\ref{thm:max_ext} for arbitrary trees~$T$ in the `dense' case~$pqn \gg \log n$.  
In Section~\ref{sec:reduction} we prove two lemmas which allow us to consider extension counts in random trees rather than random graphs, 
and these auxiliary results are subsequently used to prove our main results in the `sparse' case~$1 \ll pn \ll \log n$: 
in Section~\ref{sec:paths} we prove Theorem~\ref{thm:sparseP2} for paths~$P_m$ of any fixed length~$m \ge 1$ (see \refT{thm:paths}), 
and in Section~\ref{sec:symm_trees} we give the more involved proof of Theorem~\ref{thm:symm_trees_abr} for spherically symmetric trees~$T_{a,b}$ of height two in the entire range~$1 \ll pn \ll \log n$ (see \refT{thm:symm_trees}).  
Finally, in \refS{sec:minimum} we demonstrate that our methods also carry over to minimum rooted tree extension counts, i.e., to the random variable~${\min_{v \in [n]}X_v}$ (see \refT{thm:min_ext}).

\pagebreak[3]

\section{Preliminaries}
\label{sec:prelim}
In this preparatory section, we introduce basic concepts and results on the binomial distribution and the degrees in~$\Gnp$. We also provide an asymptotic formula for the variance of the number of $T$-extensions. Since all proofs are more or less routine (and some of them even follow from textbook results~\cite{Bollobas01}), we defer them to~\refApp{app:proofs} to avoid clutter. 
On a first reading the reader may wish to skip straight to~Section~\ref{sec:proof_dense}. 

\subsection{Asymptotic notation} 
\label{ss:asympt}
We use standard asymptotic notation with respect to two sequences $a_n, b_n$ (which are sometimes implicit): we write $a_n \sim b_n$ whenever $a_n/b_n \to 1$; moreover, assuming that eventually $b_n > 0$, we write $a_n = o(b_n)$ or $a_n \ll b_n$ or $b_n \gg a_n$ whenever $a_n /b_n \to 0$, write $a_n = O(b_n)$ whenever $\limsup_{n \to \infty} |a_n| / b_n < \infty$, write $a_n = \Omega(b_n)$ whenever $\liminf_{n \to \infty} a_n/b_n > 0$, and write $a_n \asymp b_n$ or $a_n = \Theta(b_n)$ whenever both $a_n = O(b_n)$ and $a_n = \Omega(b_n)$ hold.
\subsection{Asymptotics of binomial probabilities}\label{ss:bin_asymp}
We recall rather universal upper and lower bounds for binomial probabilities.
Let~$q := 1 - p$ and $\xi \sim \Bin(n,p)$.
To this end we recall the large deviation rate function for the Poisson distribution, defined~as
\begin{equation}
    \label{eq:entropy}
    \phi(x) := (1 + x) \log (1 + x) - x\,, \qquad x \in [-1,\infty)\,,
\end{equation}
with~$\phi(-1) = 1$ defined as the right limit (as usual). 
We will use the following well-known \emph{Chernoff~bound} (see~\citeref{JLRbook}{Theorem~2.1}) for the upper tail:
\begin{equation}
\label{eq:Chern_upper}
\prob{\xi \ge (1 + \eta)pn}  \le \exp \Bigpar{ - pn \phi(\eta) } 
\,,  \qquad \eta \ge 0.
\end{equation}
Noting that~$\phi(\eta) \ge \phi(\eta)-1 = (1 + \eta) \log\bigpar{(1 + \eta)/\e}$, inequality~\eqref{eq:Chern_upper} yields the following simple upper bound, which is often useful when~$x > \e pn$:
\begin{equation}
  \label{eq:Chern_upper_simpler}
  \prob{\xi \ge x} \le \exp \left( - x \log \frac{x}{\e pn}  \right)\,, \qquad x > 0\,.
\end{equation}

The following proposition gives a lower bound on the point probability. When  \eqref{eq:bin_lower2} is used as a lower bound for the upper tail, typically the $\log k$ term is negligible, in which case it matches the upper bound~\eqref{eq:Chern_upper}.
\begin{proposition}
  \label{prop:lower_poisson}
  Let $\xi \sim \Bin(n,p)$. Let $k$ be an integer and define a number $\eta$ so that $k = (1 + \eta)pn$. If
$1 \ll k \ll \sqrt{n}$ and $pn \ll \sqrt{n}$, then
\begin{equation}
\label{eq:bin_lower2}
\prob{\xi = k} \ge \exp\Bigpar{- pn \phi(\eta) + O(\log k)}\,.
\end{equation}
\end{proposition}
It is also convenient to have the following inequalities (which follow from the inequalities above) when one aims at probabilities of the form $n^{-c}$.
\begin{lemma}
  \label{lem:n_alpha}
  Let $D = D(n,p)$ be as defined in \eqref{eq:LLN_maxdeg_sparse}. Assume that $pn \ll \log n$.
  Assume that~$\alpha = \alpha(n)>0$ satisfies $|\log \alpha| = o\Bigpar{\log \frac{\log n}{pn} }$, which includes any constant $\alpha > 0$. Then the random variable $\xi \sim \Bin(n,p)$ satisfies 
\begin{equation}
  \label{eq:n_alpha_upp}
  \prob{\xi \ge \alpha D}   \le n^{-\alpha(1 + o(1))}\,.
\end{equation}
Furthermore, under the additional assumption~$pn \ge 1$, for any constants $\alpha, \eps > 0$ we~have 
\begin{equation}
  \label{eq:n_alpha_low}
  \prob{\alpha D \le \xi < (\alpha + \eps)D} \ge n^{-\alpha + o(1)}\,.
\end{equation}
\end{lemma}

\subsection{Extremal degrees of binomial random graphs}%
We now recall some well-known results about the maximum and minimum degrees in the binomial random graph~$\Gnp$, denoted by $\Delta=\Delta(\Gnp)$ and $\delta=\delta(\Gnp)$ respectively.
\refP{prop:max_deg_dense} gives detailed information about the extremal degrees in the `dense' case~${pqn \gg \log n}$. 
\begin{proposition}
  \label{prop:max_deg_dense}
If~$pqn \gg \log n$, then 
\begin{equation}
\label{eq:minmaxdegree}
\frac{\Delta -pn }{\sqrt{2pqn \log n}} \pto 1
\quad 
\text{ and }
\quad 
\frac{pn - \delta}{\sqrt{2pqn \log n}} \pto 1\,.
\end{equation}
\end{proposition}
\begin{proposition}
 \label{prop:degree_range}
Let $C_0 > 1$ be a constant. If $pn \ge C_0 \log n$, then there are constants $0 < C_1 < 1 < C_2$ depending on~$C_0$ such~that whp $C_1 pn \le \delta \le \Delta \le C_2 pn$.
\end{proposition}
In the `sparse' case~$1 \ll pn = O(\log n)$, \refP{prop:max_deg_asymp} states that the asymptotics of the maximum degree~$\Delta$ of~$\Gnp$ is more involved (and, when~$pn \ll \log n$, then the minimum degree of~$\Gnp$ is in fact typically zero, see \refS{sec:minimum}). 
Asymptotic properties of the inverse~$\phi^{-1}$ of the function~$\phi$ defined  in~\eqref{eq:entropy} are discussed in~\refR{rem:phi}~below. 
\begin{proposition}
  \label{prop:max_deg_asymp}
If~$1 \ll pn = O(\log n)$, then
\begin{equation}
\label{eq:LLNmaxdeg}
\frac{\Delta - pn}{\alpha_n} \pto 1 \,,
  \end{equation}
where~$\alpha_n := pn \phi^{-1}\left(\frac{\log n}{pn}\right)$.
If furthermore~$pn \ll \log n$, then 
\begin{equation}
\label{eq:alpha_n_sim}
  \alpha_n \sim 
  \frac{\log n}{\log \frac{\log n}{pn} } \gg pn \,.
\end{equation}
\end{proposition}
\begin{remark}
   \label{rem:phi} 
   The function~$\phi : [-1, \infty) \to [0, \infty)$ satisfies the following well-known asymptotics that can be proved by basic calculus: 
\begin{equation}
\label{eq:phi_asymp}
  \phi(x) \sim \begin{cases}
    x^2/2 \quad & \text{if $x \to 0$}\,, \\
    x \log x \quad & \text{if $x \to \infty$}\,.
  \end{cases}  
\end{equation}
When restricted to $[0, \infty)$, $\phi$ is an increasing bijection onto $[0, \infty)$ and therefore has an increasing inverse $\phi^{-1} : [0, \infty) \to [0, \infty)$. 
In view of \eqref{eq:phi_asymp}, the function $\phi^{-1}$ satisfies
  \begin{equation}
    \label{eq:phi_inv}
    \phi^{-1}(y) \sim \begin{cases}
      \sqrt{2y} \quad & \text{if $y \to 0$}\,, \\
      \frac{y}{\log y} \quad &\text{if $y \to \infty$}\,.
    \end{cases}
  \end{equation}
\end{remark}

\subsection{Expectation and variance of  tree extension counts}
Recall that~$X_v=X_{T,v}$ denotes the number of $T$-extensions of~$v$.
For trees~$T$ we now record the (routine) asymptotics of the expectation and variance of~$X_v$. 
\begin{proposition}
  \label{prop:variance}
Fix a rooted tree~$T$, with root of degree~$a$. 
Let~$\mu_T$ and~$\sigma_T^2$ denote the expectation and variance of~${X_v = X_{T,v}}$, for some vertex~$v \in [n]$. 
  If~$pn \gg 1$, then
  \begin{equation}
\label{eq:var_trees}
\mu_T = (pn)^{e_T}(1 + O(1/n))
\qquad \text{ and } \qquad 
\sigma_T^2 \sim a^2 \mu_T^2 q/(pn)\,.
  \end{equation}
\end{proposition}

\section{Trees: the dense case} 
\label{sec:proof_dense}
In this section we deal with extension counts of an arbitrary rooted tree~$T$ in the `dense' case ${pqn \gg \log n}$.
In particular, in~\refS{sec:proof_dense_deduction} we deduce Theorem~\ref{thm:max_ext} and Corollaries~\ref{cor:breaks},\ref{cor:Gumbel}
from the key result Proposition~\ref{prop:degtoext}, which is subsequently proved in~\refS{sec:proof_dense_proof}.

\subsection{Proof of the main result for trees: Theorem~\ref{thm:max_ext} and Corollaries~\ref{cor:breaks},\ref{cor:Gumbel}}\label{sec:proof_dense_deduction} 
In the following proofs of Theorem~\ref{thm:max_ext} as well as Corollaries~\ref{cor:breaks} and \ref{cor:Gumbel}, to avoid clutter we shall write $o_p(a_n)$ for a sequence of random variables~$X_n$ such that $X_n/a_n \pto 0$, as usual (see~\citeref{JLRbook}{p.~11}). 
\begin{proof}[Proof of Theorem~\ref{thm:max_ext}]
The claimed variance asymptotics of~$\sigma_T^2$ hold by Proposition~\ref{prop:variance}. 
Gearing up towards~\eqref{eq:LLNvariance}, set $\eta := \sqrt{2q \log n /(pn)}$. 
Proposition~\ref{prop:max_deg_dense} implies that 
\[
  \max_{v \in [n]} d(v) = pn \left( 1 + \eta \left( 1 + o_p(1)\right) \right).
\]
Considering the error term in \eqref{eq:Yt_conc}, note that condition $pqn \gg \log n$ implies 
\[
C(pn)^{e_T -1}\sqrt{\log n} \ll \eta(pn)^{e_T}
\]
and $\eta \to 0$. Using the conclusion~\eqref{eq:Yt_conc} of Proposition~\ref{prop:degtoext}  it then follows~that 
\begin{align*}
M_n = \max_{v \in [n]} X_v  &= (pn)^{e_T} \cdot \Bigpar{1 + \eta \bigpar{1 + o_p(1)}}^{a}+ o_p(\eta (pn)^{e_T}) \\
&= (pn)^{e_T} \left(1 + a \eta + o_p(\eta) \right)\,.
\end{align*}
This implies~\eqref{eq:LLNvariance} by noting (see \eqref{eq:var_trees}) that~$\mu_T  =(pn)^{e_T}( 1 + O(1/n)) = (pn)^{e_T}(1 + o(\eta))$~and 
\[
a \eta(pn)^{e_T} \sim \sqrt{\tfrac{a^2 q \cdot 2 \log n}{pn}} \mu_T \sim \sigma_T\sqrt{2 \log n}\,,
\]
completing the proof of Theorem~\ref{thm:max_ext}. 
\end{proof}

\begin{proof}[Proof of Corollary~\ref{cor:breaks}]
With foresight, we define $f(C) := (1 + \phi^{-1}(1/C))^a$, where~$\phi^{-1}$ is the inverse of the function~$\phi$. The desired properties of~$f$ follow from Remark~\ref{rem:phi}.
Proposition~\ref{prop:max_deg_asymp} then implies (using continuity of~$\phi^{-1}$) that
\[
\max_{v \in [n]} d(v) = pn \cdot \bigpar{1 + \phi^{-1}(1/C) +  o_p(1)}\,,
\]
whereas the error term in \eqref{eq:Yt_conc} is $o( (pn)^{e_T} )$.
Using the conclusion~\eqref{eq:Yt_conc} of Proposition~\ref{prop:degtoext}, 
in view of~$f(C) > 1$ it then follows~that 
\begin{align*}
M_n = \max_{v \in [n]} X_v  &= (pn)^{e_T} \cdot \Bigpar{ 1 + \phi^{-1}(1/C) + o_p(1)}^{a}+ o_p((pn)^{e_T}) \\
&= (pn)^{e_T} \cdot f(C) \cdot (1 + o_p(1)) \,,
\end{align*}
which together with~$(pn)^{e_T} \sim \mu_T$ establishes~\eqref{eq:cor:breaks}, 
completing the proof of Corollary~\ref{cor:breaks}. 
\end{proof}
\begin{proof}[Proof of \refC{cor:Gumbel}]
Let~$B_n := \sqrt{pqn/(2\log n)}$. 
  We rely on the fact that the maximum degree $\Delta=\Delta(\Gnp)$ satisfies, for any fixed~$x \in \RR$,
  \begin{equation}
\label{eq:maxdeg_Gumbel}
    \lim_{n \to \infty} \prob{\Delta \le A_n + B_n x} = \exp\bigpar{-\e^{-x}}\,,
  \end{equation}
where~$A_n$ is a suitable sequence satisfying~$A_n = \Omega(pn)$; under the stronger assumption $pqn \gg \log^3 n$ this follows from  \citeref{Bollobas01}{Corollary~3.4}, which gives an explicit formula for~$A_n$, while in the remaining cases \eqref{eq:maxdeg_Gumbel} follows from \citeref{Ivchenko73}{Theorems~5 and~8} 
with~$A_n$ of a more complicated form.  
Writing~$a$ for the root degree of~$T$, we~set
\[
a_n := A_n^a(pn)^{e_T - a}\qquad \text{and} \qquad b_n := aA_n^{a-1}B_n(pn)^{e_T -a}.   
\]
Denoting the error term in~\eqref{eq:Yt_conc} by $\theta_n:= C(pn)^{e_T-1}\sqrt{\log n}$, 
  Proposition~\ref{prop:degtoext} implies that
  \begin{multline*}
    \Pr(\hat M_n \le a_n + b_n x - \theta_n) - o(1) 
    \le \prob{M_n \le a_n + b_n x} 
    \le \Pr(\hat M_n \le a_n + b_n x + \theta_n) + o(1) 
  \end{multline*}
for the random variable 
\[
\hat M_n := \max_{v \in [n]} d(v)^a(pn)^{e_T - a} = \Delta^a (pn)^{e_T -a}. 
\]
Since by \eqref{eq:maxdeg_Gumbel} we have
  \begin{equation*}
    \prob{\hat M_n \le (A_n + B_n x)^a (pn)^{e_T - a}}
    = \prob{\Delta \le A_n + B_n x} 
    \to \exp\bigpar{-\e^{-x}}\,,
  \end{equation*}
  using continuity of~$\exp (-\e^{-x})$ it thus suffices to show that 
  \begin{equation*}
  (A_n + B_nx)^a(pn)^{e_T -a} = a_n + b_n (x + o(1))
     \qquad \text{and} \qquad 
     \theta_n = o(b_n) \,.
  \end{equation*}
The first estimate follows by expanding the left-hand side, using that $B_n \ll A_n$.
The second inequality holds because $A_n = \Omega(pn)$ and $pqn \gg (\log n)^2$ together imply
  \begin{equation*}
    \frac{\theta_n}{b_n} = \frac{C (pn)^{e_T - 1} \sqrt{\log n}} {a A_n^{a-1}B_n(pn)^{e_T -a}}
    = O(1) \cdot \sqrt{\frac{(\log n)^2}{pqn}} \ll 1 \,,
  \end{equation*}
which completes the proof of~\eqref{eq:cor:Gumbel}, as discussed. 
\end{proof}

\subsection{Proof of main technical result for trees: Proposition~\ref{prop:degtoext}}\label{sec:proof_dense_proof}
The following proof of Proposition~\ref{prop:degtoext} is based on induction, 
and the base case is the core of the matter: 
in order to obtain the desired `good' concentration estimate~\eqref{eq:Yt_conc}, here we shall first condition on the neighborhood of a potential root vertex, 
and then afterwards exploit the following variant of the bounded differences inequality due to Warnke~\cite{W16}, 
which conveniently allows us to take into account 
(a)~that the typical one-step changes can be much smaller than the worst-case ones and~(b) that the underlying independent random variables are binary. 
(To clarify: Lemma~\ref{thm:McDExtV} follows from~\citeref{W16}{\mbox{Theorem~1.3} and eq.~{(1.5)}}, by setting~$c_k=r$, $d_k=R$, $\gamma_k = r/R$ and~$p_k=p$.) 
\begin{lemma}\label{thm:McDExtV}
Let~$X=(\xi_1, \ldots, \xi_N) \in \{0,1\}^{N}$ be a random vector with independent entries such that~$\prob{\xi_k=1}=p$ for all~$1 \le k \le N$.
Let~$f:\{0,1\}^{N} \to \RR$ be a function.
Assume that there is a set $\Gamma \subseteq \{0,1\}^{N}$ and numbers~$R \ge r > 0$ such that
\begin{equation}\label{eq:fTL}
|f(x)-f(\tilde{x})| \; \le \; \begin{cases}
		r & \;\text{if $x \in \Gamma$,}\\ 
		R & \;\text{otherwise,} 
\end{cases}
\end{equation}
whenever $x,\tilde{x} \in \{0,1\}^{N}$ differ in exactly one coordinate. 
Then, for all $t \ge 0$, 
\begin{equation}\label{McDExtV:Pr}
\prob{|f(X)-\E f(X)| \ge t} \; \le \; 2 \cdot \exp\left(-\frac{t^2}{8Npr^2+4rt}\right) + \prob{X \notin \Gamma} \cdot 2N R/r.
\end{equation}
\end{lemma} 

The following lemma encapsulates the core matter of the base case of the inductive proof of Proposition~\ref{prop:degtoext}.
Intuitively, inequality~\eqref{eq:technical} implies that whp the following holds: for all vertices~$v \in [n]$ with `typical' degrees~$d(v)$, the rooted extension count~$X_{T,v}$ does not deviate too much from its `anticipated' value~$\mu_{d(v)}$, where
\begin{gather}
\label{def:mus}
  \mu_s := \Ec{X_{T, v}}{d(v) = s} = (s)_a\cdot p^{e_T-a}(n - a - 1)_{e_T-a}\,,\\
\label{def:eps}
\eps := \sqrt{\log n}/(pn) \,.
\end{gather}
\begin{lemma}
  \label{lem:crux}
Let $T$ be a rooted tree of height two.
    Let $C_0 > 1, 0 < C_1 < 1 < C_2$ be constants. Then there is a constant $C>0$ such that if $pn \ge C_0 \log n$, then
    \begin{equation}
     \label{eq:technical}    
\sum_{v \in [n]} \prob{|X_{T, v} - \mu_{d(v)}| \ge C \eps (pn)^{e_T}, \; d(v) \in [C_1 pn , C_2 pn]}  \; = \; o(1)\,.
    \end{equation}
\end{lemma}
\begin{proof}
In the remainder we will define further positive constants~$C_3, C_4$, and $C_5$, where each $C_i$ may depend on $C_{0}, \ldots, C_{i-1}$ and~$e_T$.
Since (i)~by symmetry the probabilities in \eqref{eq:technical} are equal for all vertices~$v \in [n]$ and (ii)~the inequality $\prob{\mathcal{A}, \mathcal{B}} \le \Pc{\mathcal{A}}{\mathcal{B}}$ holds for arbitrary events $\mathcal{A}, \mathcal{B}$, to prove~\eqref{eq:technical} it suffices to consider the vertex~$v=1$ and~prove
\begin{equation}
\label{eq:smaller_one_over_n}
\max_{s \in [C_1 pn, C_2 pn] } \Pc{|X_{T, 1} - \mu_{d(1)}| \ge C_5 \eps (pn)^{e_T} }{ d(1) = s}  \; \ll \; 1/n\,, 
\end{equation}
where we tacitly assume that $s$ is restricted to the finite set of values in $[C_1 pn, C_2 pn]$ that are attained by $d(1)$. Among such values fix~$s$ that attains the maximum in \eqref{eq:smaller_one_over_n} and consider the set $S := \{2, \ldots, s+1\}$ (which is nonempty because $s = \Omega(\log n)$). By symmetry 
\[
\Pc{|X_{T, 1} - \mu_{d(1)}| \ge C_5 \eps (pn)^{e_T}}{d(1) = s} = \Pc{|X_{T, 1} - \mu_{d(1)}| \ge C_5 \eps (pn)^{e_T}}{N(1) = S} \ .
\]
Henceforth we consider the conditional probability space with respect to the event $N(1) = S$, and use the shorthand $\Pr_s, \E_s$ for the corresponding conditional probabilities and expectations. 
In particular, we have~$\mu_s = \E_s{X_{T,1}}$ and
\begin{equation}\label{eq:Gv:ineq}
  \Pc{|X_{T, 1} - \mu_{d(1)}| \ge C_5 \eps (pn)^{e_T}}{ d(1)=s } 
  = \Pr_s(|X_{T, 1} - \E_s{X_{T,1}}| \ge C_5 \eps (pn)^{e_T})\,.
\end{equation}
Note that this conditional probability space corresponds to~$\binom{n-1}{2}$ independent binary random variables, each with success probability~$p$ (which encode the status of the pairs~$xy$ with~${x,y \neq v}$, i.e., whether they are an edge or not). 
Since $T$ has height two, $X_{T,1}$ is already determined~by
\begin{equation}
  \label{eq:N_upper}
N:=s (n- 1-s)+\tbinom{s}{2} \le s n \le C_2 pn^2  
\end{equation}
 of these random variables, 
namely the status of the pairs~$xy$ with at least one element in~$S$; we denote these independent random variables by~$X=(\xi_1, \ldots, \xi_N) \in \{0,1\}^N$. 
Hence there is a (deterministic) function~$f:\{0,1\}^N \to \RR$ such that~$X_{T,1}=f(X)$. 

In the following we prepare for applying Lemma~\ref{thm:McDExtV} to~$X_{T,1}=f(X)$. 
For the Lipschitz condition~\eqref{eq:fTL}, 
we need to control by how much~$X_{T,1}$ changes if we alter the status of one pair of vertices, i.e., whether it is an edge or not.
To this end, given a graph~$G$ with~$N(1)=S$, let~$X_{T,1,G,uv}$ denote the number of~$T$-extensions of vertex~$1$ in $G$ containing the edge~$uv$.
Note that any such $T$-extension must map some child $w$ of the root to a vertex~$z\in\{u,v\}$ and then some child of $w$ to $\{u,v\} \setminus \{z\}$. 
To bound~$X_{T,1,G,uv}$ from above, we can thus first choose which vertices we map to $u$ and $v$ (in at most $2e_T$ ways),  
and then iteratively map the remaining~$v_T-3=e_T-2$ vertices of~$T$ in a suitable order (to a vertex adjacent to the root vertex~$1$ or a vertex already mapped to), yielding 
\begin{equation}\label{eq:ZT:extension}
X_{T,1,G,uv} \; \le \; 2e_T \cdot \Bigsqpar{\max_{w \in S \cup \{1\}}d_G(w)}^{e_T-2}\,,
\end{equation}
where~$d_G(w)$ denotes the degree of~$w$ in~$G$. 
As the degree~$d(1)=s \le C_2 pn$ of vertex~$1$ is  fixed (in the conditional probability space), 
we define the `good' event
\[
\Gamma := \Bigcpar{\max_{w \in S} d(w) \le C_3 pn } \qquad \text{ with } \qquad  C_3:= \e^2e_T+1.
\]
Note that~$\Gamma$ is determined by~$X=(\xi_1, \ldots, \xi_N)$, and that vertex~$w \in S$ has degree~$d(w)={1+Y_w}$, where~$Y_w \sim \Bin(n-2,p)$. 
Using a standard union bound argument over all~$w \in S$ and the inequality~$pn \ge C_0 \log n \ge \log n$, 
by applying the Chernoff bound~\eqref{eq:Chern_upper_simpler} to~$Y_w$ with~$x:=\e^2e_T pn$ (since~$d(w)=1+Y_{w} \ge C_3pn$ implies~$Y_{w} \ge x$) it follows~that 
\begin{equation}\label{eq:Pr:Gammafails}
\Pr_s(X \not \in \Gamma) \; \le \; |S| \cdot \exp \left( - \e^2e_Tpn \ln\left(\frac{\e^2e_T pn}{\e p(n-2)}\right) \right) \le n \cdot \e^{-\e^2e_Tpn} \ll n^{-5 e_T} \,.
\end{equation}
Next we claim that, whenever~$x,\tilde{x} \in \{0,1\}^N$ differ by at most one coordinate, then 
\begin{equation}\label{eq:fTL:concrete}
|f(x)-f(\tilde{x})| \; \le \; \begin{cases}
		C_4 (pn)^{e_T-2} & \;\text{if $x \in \Gamma$,}\\ 
		n^{e_T}  & \;\text{otherwise,} 
\end{cases}
\end{equation}
where $C_4 := 2e_T(2\max\{C_3,C_2\})^{e_T-2}$. 
The second bound~$n^{e_T}$ in~\eqref{eq:fTL:concrete} follows crudely, since for every $x$ we have~$0 \le f(x) \le n^{v_T-1} = n^{e_T}$. 
For the first bound in~\eqref{eq:fTL:concrete}, we imagine that we take~$x \in \Gamma$ and then flip one coordinate to obtain~$\tilde{x} \in \{0,1\}^N$, 
which corresponds to removing or adding one edge, say~$uv$. 
After adding the edge~$uv$, the degrees in the resulting graph~$G$ satisfy, by exploiting the degree property~${x \in \Gamma}$ and~$d_G(1)=s \le C_2 pn$, 
\begin{equation*}
\max_{w \in S \cup \{1\}}d_G(w) 
\le \max \left\{ C_3pn+1 , \: C_2pn \right\}
\le {2 \max\{C_3,C_2\}pn} \,,
\end{equation*}
which together with~\eqref{eq:ZT:extension} readily establishes~\eqref{eq:fTL:concrete}.

We are now ready to apply Lemma~\ref{thm:McDExtV} to~$X_{T,1}=f(X)$ with parameters
\begin{equation*}
r := C_4 (pn)^{e_T-2}, \qquad 
R := n^{e_T} \quad \text{ and } \quad  
t := C_5 \eps (pn)^{e_T} 
\end{equation*} 
and appropriately chosen constant $C_5$. Using~$N \le C_2 pn^2$ (see \eqref{eq:N_upper}) together with~$\eps=\sqrt{\log n}/pn \ll 1$, we now see that if we pick~$C_5>0$ large enough (in view of definitions of $C_3, C_4$ it can be chosen as a function of $e_T, C_1$ and $C_2$) then the exponent in inequality~\eqref{McDExtV:Pr} is at least
\begin{equation}\label{eq:TBDI_exponent}
\begin{split}
  \frac{t^2}{8Npr^2 + 4rt} & \ge \min\biggcpar{\frac{t^2}{16Npr^2}, \: \frac{t}{8r}} \\
  & \ge \min\biggcpar{\frac{(C_5 \eps)^2 (pn)^2}{16C_2C_4^2 }, \: \frac{C_5\eps (pn)^{2}}{8C_4}} \ge 2 \log n\,.
\end{split}
\end{equation}
Applying the typical bounded differences inequality~\eqref{McDExtV:Pr}, 
by combining the exponent estimate~\eqref{eq:TBDI_exponent} with the error probability estimate~\eqref{eq:Pr:Gammafails} 
it follows~that 
\begin{equation*}
\begin{split}
 \Pr_s \left(|X_{T, 1} - \E_s{X_{T,1}}| \ge C_5 \eps (pn)^{e_T} \right) & \; \le \;  
 2 \cdot \e^{-2\log n}
 + o(n^{-5e_T}) \cdot O(n^{2+e_T}) \ll 1/n \,,
\end{split}
\end{equation*}
   which together with~\eqref{eq:Gv:ineq} completes the proof of inequality~\eqref{eq:smaller_one_over_n} and thus Lemma~\ref{lem:crux} with the constant~$C = C_5$. 
\end{proof}
Let~$Z_{T,v}$ denote the number of graph homomorphisms of~$T$ into $\Gnp$ mapping the root of~$T$ to~$v$.  
In the upcoming proof of  Proposition~\ref{prop:degtoext}, it will be convenient to consider~$Z_{T,v}$ instead of~$X_{T,v}$, since we then do not need to worry about whether certain vertices coincide. 
We trivially have~$Z_{T,v} \ge  X_{T,v}$ (since~$Z_{T,v}$ includes noninjective homomorphisms), 
and the following claim asserts that typically~$Z_{T,v} \approx  X_{T,v}$.
\begin{lemma}
\label{lem:noninj}
For every $C_0 > 1$ there is a constant $C'=C'(C_0,T) > 0$ such that for $pn \ge C_0 \log n$ we~have
\begin{equation}   
  \label{eq:Zt_negl}
	    \max_{v \in [n]} \:  |Z_{T,v} - X_{T,v}| \le C'(pn)^{e_T-1} \quad \text{whp}.
	  \end{equation}
\end{lemma}
\begin{proof}
  We shall bound noninjective homomorphisms in terms of the maximum degree $\Delta = \max_{v \in [n]} d(v)$, exploiting the loss of freedom due to repeated vertices. 
  In particular, every homomorphism counted by $Z_{T,v} - X_{T,v} \ge 0$ maps the~${v_T-1}=e_T$ non-root vertices to~$k \le {e_{T}-1}$ vertices, and so the image of this mapping contains a \mbox{$U$-extension} of~$v$, where~$U$ is some rooted tree with~$v_U = k+1$ vertices and~$e_{U} = k$ edges. 
  Taking all such rooted trees~$U$ into account (let $c_k$ be their number), by iteratively choosing the vertices of~$U$ in~$\Gnp$ (each adjacent to the root vertex~$v$ or an already chosen vertex) we infer via the whp maximum degree bound~$\Delta \le C_2pn$ from Proposition~\ref{prop:degree_range} (where $C_2 \ge 1$ depends on $C_0$) that,~whp, 
  \[
  \max_{v \in [n]} \:  |Z_{T,v} - X_{T,v}| \le \sum_{0 \le k \le e_T-1} c_k \Delta^{k} \le  C'(pn)^{e_T-1} \,,  
  \]
for a suitable constant~$C'=C'(C_2,T)>0$,   establishing the desired bound~\eqref{eq:Zt_negl}. 
\end{proof}
After these preparations, we are now ready to prove Proposition~\ref{prop:degtoext}.
\begin{proof}[Proof of Proposition~\ref{prop:degtoext}]
Recall that we are assuming $pqn \ge C_0 \log n$ for constant $C_0 > 1$. By Proposition~\ref{prop:degree_range}  there are two constants $0 < C_1 < 1 < C_2$ depending on~$C_0$ such~that 
\begin{equation}
\label{eq:probminmax}
\prob{\text{$C_1 pn \le d(v) \le C_2 pn$ for all $v \in [n]$}} \to 1 \,.
\end{equation}
Recall that $\eps = \sqrt{\log n}/(pn)$ as defined in~\eqref{def:eps}.
To complete the proof of Proposition~\ref{prop:degtoext}, we now claim that it suffices to show that, for some constant~$C_T>0$, we have
\begin{equation}
    \label{eq:Zt_conc}
	  \max_{v \in [n]} \left|\frac{Z_{T,v}}{d(v)^a(pn)^{e_T-a}} - 1\right|\leq C_T\eps  \quad \text{whp}.
  \end{equation}
Indeed, this readily implies the desired estimate~\eqref{eq:Yt_conc} with $C = C_2^a C_T + 1$, say, by applying the triangle inequality and the degree bound~$d(v) \le C_2pn$ from \eqref{eq:probminmax} as well as \eqref{eq:Zt_negl}.

It thus remains to prove inequality~\eqref{eq:Zt_conc}, for which we will use induction on the height~$h$ of the tree~$T$, 
the base case being~$h \in \{0, 1, 2\}$. The case~${h=0}$ is trivial and for~${h=1}$ the formula $Z_{T, v} = d(v)^{a} = d(v)^{a} (pn)^0$ holds deterministically. 
Deferring the proof of the remaining base case~$h=2$ (which is most of the work, and will be based on Lemma~\ref{lem:crux}), 
we now turn to the induction step for height~$h \ge 3$, where we assume that \eqref{eq:Zt_conc} holds for all trees $T$ of height at most~$h-1$. Now, for a tree $T$ of height $h\geq 3$, let $u_1,\dots, u_a$ be the children of the root of $T$, and, for $i\in[a]$, let~$T_i$ denote the rooted tree consisting of the root~$u_i$ and all its descendants in~$T$. 
For notational convenience,~set 
  \begin{equation*}
  a_i := \deg_{T_i}(u_i)
  \qquad \text{ and } \qquad 
     b := \sum_{i\in[a]}a_i \,.
  \end{equation*}
Denoting by~$N(v)$ the set of neighbors of~$v$ in $\Gnp$,
using induction we infer that, whp, 
	\begin{align}
	  \notag Z_{T, v} & =\sum_{v_1,\dots, v_a\in N(v)}\prod_{i\in [a]}Z_{T_i, v_i}\\
	  \notag &= \sum_{v_1,\dots, v_a\in N(v)}\prod_{i\in [a]}d(v_i)^{a_i}(pn)^{e_{T_i}-a_i}(1 \pm C_{T_i}\eps)\\
	\label{eq:induction}
	& =\biggpar{\sum_{v_1,\dots, v_a\in N(v)}\prod_{i\in [a]}d(v_i)^{a_i}} \cdot (pn)^{e_T-a-b}\prod_{i \in [a]}(1 \pm C_{T_i}\eps)\,.
	\end{align}
for all~$v \in [n]$.
Denote by $T'$ the subtree of $T$ induced by vertices of depth at most two, so that $e_{T'} = a + b$. Observe that the first factor in \eqref{eq:induction} is precisely~$Z_{T', v}$ and using the induction hypothesis again, this time for height $2$, 
it follows that,~whp, 
	\begin{equation*}
\begin{split}
  Z_{T, v} & =(1 \pm C_{T'}\eps)d(v)^a (pn)^b \cdot (pn)^{e_T-a-b}\prod_{i \in [a]}(1 \pm  C_{T_i}\eps) 
\end{split}
\end{equation*}
for all~$v \in [n]$, which implies that~$Z_{T, v} = (1 \pm C_T\eps)d(v)^a(pn)^{e_T-a}$
with a suitable constant~${C_T>0}$ for large enough~$n$ (since~$\eps \to 0$), establishing the induction step.

We turn to the base case of height~$h = 2$ for inequality~\eqref{eq:Zt_conc}, which is the core of the matter. Here we find it convenient to focus on injective homomorphism counts~$X_{T,v}$ (and then use~\eqref{eq:Zt_negl} from Lemma~\ref{lem:noninj} to transfer to~$Z_{T,v}$).
Fix a tree~$T$ of height~two.  
Recalling the two constants $C_1, C_2$ from \eqref{eq:probminmax}, let the constant $C=C(T, C_0, C_1, C_2)>0$ be as in Lemma~\ref{lem:crux}. Define the `good'~event
\begin{equation*}
  \cG_v := \bigl\{ |X_{T, v} - \mu_{d(v)}| < C \eps (pn)^{e_T} \bigr\}\,,
\end{equation*}
Recall that by definition~\eqref{def:mus} we have~$\mu_s=(s)_a\cdot p^{e_T-a}(n - a - 1)_{e_T-a}$. 
For all ${s \ge C_1pn}$, we uniformly have $(s)_a = s^a (1 - O(1/pn))$, which together with ${(n-a - 1)_{e_T - a}} = {n^{e_T - a} \left( 1 - O(1/n) \right)}$ and ${\eps \gg 1/pn}$ 
readily implies that
\begin{equation}
	\label{eq:ExpHom}
	\max_{s \in [C_1 pn, C_2 pn]} \left| \frac{\mu_s}{s^a(pn)^{e_T - a}} - 1 \right|  \; \ll \; \eps\,.
\end{equation}
Furthermore, the union bound,~\eqref{eq:probminmax} and Lemma~\ref{lem:crux} imply~that
\begin{align*}
&\Pr\biggpar{\bigcup_{v \in [n]}\bigl\{\neg\cG_v \text{ or } d(v)/(pn) \notin [C_1, C_2]\bigr\}} \\
& \quad \le \prob{d(v)/(pn) \notin [C_1, C_2]\text{ for some } v \in [n] } + \sum_{v \in [n]} \prob{\neg\cG_v, \; d(v)/(pn) \in [C_1, C_2]} \ll 1\,. 
\end{align*}
Hence whp the events~$\cG_v$ and~${d(v)/(pn) \in [C_1, C_2]}$ hold simultaneously for all~${v \in [n]}$. 
Together with inequalities~\eqref{eq:Zt_negl}  and~\eqref{eq:ExpHom},
using~$C_1 pn \le \min\{pn, d(v)\}$ and~${(pn)^{e_{T}-1}} \ll {\eps (pn)^{e_{T}}}$ it follows that, whp, for all vertices~$v \in [n]$ we have
\begin{equation*}
  Z_{T, v} = \mu_{d(v)} \pm C \eps (pn)^{e_T} + O( (pn)^{e_T - 1} )
= d(v)^a(pn)^{e_T - a}(1\pm 2C_2^a C \eps)\,,
\end{equation*}
completing for~$C_T := 2C_2^a C$ the proof of inequality~\eqref{eq:Zt_conc} in the base case $h=2$. 
\end{proof}

\section{Trees: reduction to random trees in the sparse case}
\label{sec:reduction}
In this section we establish two auxiliary lemmas that will conceptually simplify the proofs of Theorems~\ref{thm:paths} and~\ref{thm:symm_trees} for rooted paths and spherically symmetric trees in Sections~\ref{sec:paths} and~\ref{sec:symm_trees}: 
in the `sparse' case ${1 \ll pn \ll \log n}$, Lemmas~\ref{lem:sparse_lower_GW} and~\ref{lem:sparse_upper_GW} below allow us to focus on the tails of extension counts in random trees, 
which are much easier to handle than extension counts in random graphs where there are more dependencies. 
We note that the proofs of these `transfer results' need to take into account that maximum tree extension counts are an extreme value problem (maximizing over all possible root~vertices).

Recall that a Galton--Watson tree~$\cT_{n,p}$ 
with offspring distribution ${\xi \sim \Bin(n,p)}$ is defined recursively, starting with the root vertex and giving each vertex a random set of children, the number of which is an independent copy of~$\xi$. 

For a `large' rooted tree~$G$ and a `small' rooted tree $T$, let~$f_T(G)$ be the number of $T$-extensions of the root in $G$ (i.e., injective homomorphisms from $T$ to $G$ that map the root of $T$ to the root of $G$).
In the next lemma we obtain a lower bound on the maximum number of $T$-extensions in~$\Gnp$ via a lower bound on the tail of the extension count in $\cT_{n,p}$. (We have not attempted to optimize assumption~\eqref{eq:tree_sparse_lower_cond}, since it suffices for our purposes.)
\begin{lemma}[Reduction to random tree: lower bound]\label{lem:sparse_lower_GW}
  Let~$T$ be a rooted tree of height $h \geq 1$. Let $p=p(n)\in(0,1)$ satisfy $1 \ll pn \le \log n$. There exists a sequence $n^* = n^*(n)  \sim n$  such that if a sequence $k = k(n)$ satisfies
  \begin{equation}
\label{eq:tree_sparse_lower_cond}
\prob{f_T(\cT_{n^*,p}) \ge k} \gg (\log n)^{h+1}/n\,,
  \end{equation}
  then with high probability some vertex has at least $k$ many $T$-extensions in $\Gnp$.
\end{lemma}
The purpose of the second lemma is to obtain an upper bound on the maximum number of $T$-extensions in~$\Gnp$ via an upper bound on the tail of the extension count in~$\cT_{n,p}$. 
The assumption on~$k$ will turn out to be negligible in our applications,
because in the range~$1 \ll pn \ll \log n$ the maximum degree is of order~$D \gg pn$ (see~\eqref{eq:LLN_maxdeg_sparse}), and the maximum count of extensions will thus be of higher order than the mean $\mu_T \sim (pn)^{e_T}$.
\begin{lemma}[Reduction to random tree: upper bound]\label{lem:sparse_upper_GW}
Fix~$c \in (0,1]$. 
Let~$T$ be a rooted tree of height $h \geq 1$. Let $p=p(n)\in(0,1)$ satisfy $pn \gg 1$. 
If a sequence $k = k(n)$ satisfies~$k \ge ((1+c) pn)^{e_T}$ and 
\begin{equation}
\label{eq:tree_sparse_upper_cond}
\prob{f_T(\cT_{n,p}) \ge k} \ll 1/n\,,
  \end{equation}
  then with high probability every vertex has at most~$k$ many $T$-extensions in $\Gnp$.
\end{lemma}

\subsection{Proofs of Lemmas~\ref{lem:sparse_lower_GW} and~\ref{lem:sparse_upper_GW}}
The proofs of Lemmas~\ref{lem:sparse_lower_GW} and~\ref{lem:sparse_upper_GW} are both based on coupling arguments, using an exploration process to relate the neighborhood structure in the random graph~$\Gnp$ to the Galton--Watson tree~$\cT_{n,p}$. 
Given a vertex~${v \in [n]}$, we explore all vertices at distance at most $h$ from $v$ in the breadth-first manner, generating a subtree of~$\Gnp$ rooted at~$v$. 

Turning to the precise definition, at each step~$i \ge 0$, the set~$L_i$ will be a subset of the leaves of  the partially discovered tree, and~set $S_i$ (which we call the \emph{candidate set}) will consist of vertices not in the tree. 
We start~with
\[
L_0 := \{v\} \quad \text{ and } \quad S_0 := [n] \setminus \{ v \}.
\]
At each step~$i \ge 0$, we pick a vertex~$v_i$ in~$L_i$ of minimal distance from~$v$ and expose the edges between $v_i$ and $S_i$ (their number has distribution $\Bin(|S_i|,p)$). Then we remove~$v_i$ from~$L_i$,  and transfer the new neighbors of~$v_i$ from~$S_i$ to~$L_i$, i.e., set 
\[
L_{i+1} := (L_i \setminus \{v_i\}) \cup N(v_i, S_i) \quad \text{ and } \quad 
  S_{i+1} := S_{i} \setminus N(v_i, S_i) \,,
\]
where~$N(u, S)$ denotes the neighbors in~$\Gnp$ of a vertex~$u$ in a set~$S$.
Repeating this step until $L_i$ has no vertices of depth smaller than $h$, we obtain a tree of height at most~$h$ which we call the \emph{$h$-neighborhood subtree of~$v$ in~$\Gnp$}. Note that the  edge set of this tree equals~${\cup_i (\{v_i\} \times N(v_i, S_i))}$ by construction. 

\begin{proof}[Proof of Lemma~\ref{lem:sparse_lower_GW}]
  Let $\Delta_p$ denote the maximum degree of $\Gnp$. We claim that there is ~$D = \Theta(\log n)$ such that whp~$\Delta_p \le D$. Since $pn \le \log n$, by a well-known coupling the random graph $\Gnp$ can be treated as a subgraph of $\G_{n,(\log n)/n}$ and thus we can assume $\Delta_p \le \Delta_{(\log n)/n}$. By Proposition~\ref{prop:max_deg_asymp}, whp $\Delta_{(\log n) / n} \le (1.001 + \phi^{-1}(1)) \log n =: D$. Clearly $D = \Theta(\log n)$. Let us further omit the subscript and write $\Delta = \Delta_p$.
With foresight,~set 
\[
  n^* := \floor{n - n/D} \qquad \text{and} \qquad n' := \floor{n/(2D^{h+1})}.
\]
We define a sequence of vertex-disjoint rooted subtrees $\cT_1, \cT_2, \ldots$  (of random finite length) in $\Gnp$ as follows. 
Pick an arbitrary vertex~$u_1 \in [n]$, and let~$\cT_1$ be the $h$-neighborhood tree of~$u_1$ in~$\Gnp$, as defined above. 
Note that, conditioned on $\cT_1$, the `leftover' subgraph of $\Gnp$ induced by the vertex set $[n] \setminus V(\cT_1)$, 
has the same distribution as $\bG_{n-|V(\cT_1)|,p}$ 
(to see this, note that the exploration process only looked at edges with at least one endpoint in~$V(\cT_1)$, 
so all remaining edges are still present independently with probability~$p$).
Hence in this leftover graph we can pick an arbitrary vertex~$u_2$, and let~$\cT_2$ be its $h$-neighborhood tree in $\bG_{n-|V(\cT_1)|,p}$, as defined above. 
We repeat the procedure (remove the vertices of the last tree; pick an arbitrary vertex; explore its $h$-neighborhood subtree) until we run out of vertices, obtaining a sequence of disjoint rooted trees $\cT_1, \cT_2, \dots$, as~desired.

Let $\cT_{n^*, p, h}$ be the subtree of $\cT_{n^*,p}$ consisting of vertices of depth at most $h$. We now claim that we can couple this sequence $\cT_1, \cT_2, \dots$ of trees with independent copies~$\cU_1, \dots, \cU_{n'}$ of the tree~$\cT_{n^*, p, h}$, with the property that
  \[ 
\text{$\Delta \le D$ implies that for every~$j \in [n']$, the tree~$\cT_j$ exists and satisfies $\cU_j \subseteq \cT_j$.} 
  \]
To see that such a coupling exists, we need to check that as long as (i)~we have not completed constructing all trees $\cT_1, \dots, \cT_{n'}$ and (ii)~we have not revealed more than $D$ edges incident to the same vertex, we still have enough space to sample another set of $\Bin(n^*,p)$-distributed children. In other words, we need to check that for every $j \in [n']$ and every step $i$ in the construction of $\cT_j$ the candidate set $S_i$ contains at least $n^*$ vertices.  Note that whenever $\Delta \le D$, each tree $\cT_j$ spans at most~$\sum_{k = 0}^{h} D^k \le 2D^h$ vertices, and thus the trees $\cT_1, \dots, \cT_{n'}$ exist and together span at most~$2n' D^h \le n/D$ vertices. This means that as long as we have not discovered a vertex with more than~$D$ neighbors, we still have (for every $j$ and $i$) that~$|S_i| \ge n - n/D \ge n^*$. To define the coupling, consider a step $i$ of the construction of $\cT_j$. Recall that at this step we pick $v_i \in L_{i}$ and find its neighbors in $S_i$. If $v_i$ has not been included in $\cU_j$ so far, none of its children in $\cT_j$ will be included in $\cU_j$; otherwise (if $v_i$ has been included in $\cU_j$), we choose $n^*$ vertices in $S_i$ (say first according to some predetermined ordering) before exposing the edges between~$v_i$ and~$S_i$ and only include in $\cU_j$ those new vertices (and, of course, the edges leading to them) in $N(v_i,S_i)$ that are among the $n^*$ chosen ones. 
If we do discover a vertex of degree more than~$D$ in~$\Gnp$, then we complete the construction of the trees~$\cU_1, \dots, \cU_{n'}$ without embedding the remaining~$\cU_j$ into the trees $\cT_j$, just to make sure that the $\cU_1, \dots, \cU_{n'}$ are well~defined.  

Writing $X_v$ for the number of $T$-extensions of the vertex~$v$ in~$\Gnp$, by the properties of the coupling constructed above it follows that 
\begin{equation}\label{eq:tree-coupling}
\begin{split}
\prob{\max_{v \in [n]} X_v < k} & \le \prob{\max_{i \in [n']} f_T(\cT_i) < k, \; \Delta \le D} + \Pr(\Delta > D)\\
 &\le \prob{\max_{i \in [n']} f_T(\cU_i) < k} + o(1)\,,
\end{split}
\end{equation}
where for the last inequality we recalled that whp~$\Delta \le D$.
Let~$\pi_n := \prob{f_T(\cT_{n^*,p}) \ge k}$. 
Since $D =\Theta(\log n)$, from the definition of~$n'$ and assumption~\eqref{eq:tree_sparse_lower_cond} it follows that~$\pi_n n'\gg 1$.
Since the trees $\cU_1, \dots, \cU_{n'}$ are independent copies of~$\cT_{n^*, p, h}$, and $f_T(\cT_{n^*,p,h}) = f_T(\cT_{n^*, p})$, it follows~that
\[
\prob{\max_{i \in [n']} f_T(\cU_i) < k} = \left( 1 - \pi_n \right)^{n'} \le e^{-\pi_n n'} \to 0\,,
\]
which together with~\eqref{eq:tree-coupling} and the definition of~$X_v$ completes the proof of Lemma~\ref{lem:sparse_lower_GW}.
\end{proof}

\begin{proof}[Proof of Lemma~\ref{lem:sparse_upper_GW}]
Given a vertex~$v \in [n]$, let~$\cT_v$ denote the \mbox{$h$-neighborhood} tree of~$v$ in~$\Gnp$, as defined before the proof of Lemma~\ref{lem:sparse_lower_GW}. 
By construction of~$\cT_v$ (and the trivial bound~$|S_i| \le n$) there is a coupling of~$\cT_{n,p}$ and~$\cT_v$ so that $\cT_v$ is a subtree of~$\cT_{n,p}$ (with the same root), which in particular ensures that~$f_T(\cT_v) \le f_T(\cT_{n,p})$.
Assumption \eqref{eq:tree_sparse_upper_cond} and a union bound over the $n$ vertices~$v \in [n]$ implies that
\begin{equation}\label{eq:Tv:lower}
  \max_{v \in [n]} f_T(\cT_v) \le k \quad \text{whp}.
\end{equation}

Recall that $X_v$ is the number of $T$-extensions of $v$ in $\Gnp$. Intuitively, we should typically have $X_v \approx f_T(\cT_v)$ since the neighborhood around~$v$ looks like~$\cT_v$ plus a few \emph{surplus} edges, i.e., edges in~$\Gnp \setminus \cT_v$ spanned by~$V(\cT_v)$. 
To exploit this intuition, let~$\cB$ denote the `bad' event that~$\Gnp$ contains a cycle~$C$ of length at most~$2h+1$ and a vertex $u$ of degree at least~$(1+c)pn$ which is within graph distance at most~$2h$ from $C$ (i.e., there is a path with at most $2h$ edges from $u$ to some vertex on the cycle).
We claim~that  
\begin{equation}\label{eq:Tv:lower:2}
\neg\cB \qquad \text{ implies } \qquad 
  X_v \le \max\bigcpar{f_T(\cT_v), \: k} \quad \text{for all~$v \in [n]$}\,.
\end{equation}
To prove~\eqref{eq:Tv:lower:2}, we fix~$v \in [n]$. 
When~$\cT_v$ has no surplus edges, then~$X_v = f_T(\cT_v)$, establishing~\eqref{eq:Tv:lower:2}.
We henceforth consider the case when~$\cT_v$ has at least one surplus edge, which means that vertex~$v$ is within distance~$h$ of a cycle of length at most~$2h + 1$. 
Note that if another vertex $u$ is within distance $h$ of $v$, then $u$ is within distance $2h$ of the cycle.
Therefore event~$\neg\cB$ implies that all vertices within distance~$h$ of~$v$ have degree less than~$(1+c)pn$. 
Since~$T$ has height~$h$,
it easily follows that 
~$X_v \le ((1+c)pn)^{e_T} \le k$, establishing~\eqref{eq:Tv:lower:2}. 

To complete the proof of Lemma~\ref{lem:sparse_upper_GW}, 
in view of~\eqref{eq:Tv:lower} and~\eqref{eq:Tv:lower:2} it suffices to show that
\begin{equation}\label{eq:Tv:Bad}
\Pr(\cB)=o(1)\,.
\end{equation}
Note that if the event~$\cB$ occurs, then~$\Gnp$ contains a cycle~$C$ of length~$3 \le \ell \le 2h + 1$ and a path~$P$ of length~$0 \le i \le 2h$ connecting~$C$ to a vertex~$v$ of degree at least~$(1+c)pn$, where~$C$ and~$P$ share exactly one vertex (so~$v$ is contained in~$C$ when~$i=0$).
Note that the vertex~$v$ is uniquely determined by the choice of~$C$ and~$P$, and that~$v$ has at most~$\ell+i \le 4h + 1$ neighbors in the subgraph~$C \cup P$. 
Using a standard union bound argument that takes all such cycles~$C$ and connecting paths~$P$ into account, it thus follows~that 
\begin{equation*}
  \Pr(\cB) \le \sum_{3 \le \ell \le 2h + 1} \sum_{0 \le i \le 2h} n^{\ell} \cdot \ell n^{i} \cdot p^{\ell + i} \cdot \underbrace{\Pr\bigpar{\Bin\bigpar{n-(\ell+i),p}\ge (1+c)pn-(4h + 1)}}_{=: \Pi_{\ell,i}} \,.
\end{equation*}
Recall that~$c \in (0,1]$ is fixed. 
Using stochastic domination together with~$pn \to \infty$, 
by the Chernoff bound~\eqref{eq:Chern_upper} it follows (for all large enough~$n$) that 
\begin{equation*}
  \Pi_{\ell,i} \le \Pr\bigpar{\Bin(n,p)\ge (1+c/2)pn} \le \e^{-\phi(c/2)pn}\,.
\end{equation*}  
Since~$c,h>0$ are constants, using again~$pn \to \infty$ it readily follows that 
\begin{equation*}
  \Pr(\cB) \: \le \: (2h + 1)^3 \cdot (pn)^{4h + 1} \cdot \e^{-\phi(c/2)pn} = o(1)\,,
\end{equation*}
completing the proof of~\eqref{eq:Tv:Bad} and thus of Lemma~\ref{lem:sparse_upper_GW}, as discussed. 
\end{proof}

\section{Paths: the sparse case}
\label{sec:paths}
In this section we deal with extension counts of rooted paths in the `sparse' case ${1 \ll pn \ll \log n}$.  In particular, we state and prove Theorem~\ref{thm:paths} and Proposition~\ref{prop:km_asymp} below, which generalize Theorem~\ref{thm:sparseP2} to any $m$-edge path~$P_m$ rooted at an endpoint. 
As we shall see, here the typical behavior 
is determined by the parameter $k_m = k_m(n,\lambda)$, which for positive integers $n$ and~$\lambda \in (0, \log n)$ is recursively defined~as
    \begin{equation}\label{eq:km:def}
      k_m(n, \lambda) := 
      \begin{cases}
	\frac{\log n}{\log \frac{\log n}{\lambda}} \quad & \text{if } m=1\,,\\
	\frac{\log n}{\log \left( 1 + \frac{\log n }{\lambda k_{m-1}(n,\lambda)} \right)}\quad & \text{if } m \ge 2\,.
      \end{cases}
    \end{equation}
We note that the `correct' definition of $k_m$ is, in some sense, already a significant part of the proof of Theorem~\ref{thm:paths}, 
since it will enable us to inductively determine certain tail bounds using large deviation-type optimization arguments. 

%
  \begin{theorem}[Maximum for paths, sparse case]
    \label{thm:paths}
Fix~${T = P_m}$, with~$m \ge 1$. 
Set~$\lambda := pn$. 
If~$1 \ll \lambda \ll \log n$, then  
\begin{equation*}
\frac{M_n}{k_m(n, \lambda)} \pto 1 \,.
\end{equation*}
  \end{theorem}
We will also show that the~$k_m$ satisfies the following asymptotics.
For any integer~$i \ge 1$, we denote by~ $\log^{(i)}$ the natural logarithm iterated $i$ times (in particular~$\log^{(1)}=\log$). 
\begin{proposition}
  \label{prop:km_asymp}
  For $m \ge 1$, the function~$k_m = k_m(n, \lambda)$ defined in \eqref{eq:km:def} is increasing in~$\lambda$, and $k_m$ satisfies
\begin{equation}\label{eq:km:asymp}
  k_m(n,\lambda) \sim
      \begin{cases}
	\frac{\log n}{\log \frac{\log^{(m)} n }{\lambda}} \quad &\text{if } 1 \le \lambda \ll \log^{(m)} n\,,\\[4mm]
	\frac{\lambda^{m-i}\log n}{\log \frac{\log^{(i)} n }{\lambda}} \quad &\text{if } \log^{(i+1)} n \ll \lambda \ll \log^{(i)} n \text{ with } 1 \le i \le m-1\,,\\[4mm]
	\frac{\lambda^{m-i}\log n}{\log (1 + C)}  \quad & \text{if } \lambda \sim (\log^{(i)} n)/C \text{ for constant~$C>0$, with } 2 \le i \le m\,.
      \end{cases}
    \end{equation}
\end{proposition}
Applying the inequality~$\log (1 + x) \le x$ to the denominator in~\eqref{eq:km:def}, it follows that 
    \begin{equation}
      \label{eq:km_lambda_kmminus}
      k_m \ge \lambda k_{m-1} \qquad \text{for } m \ge 2\,,
    \end{equation}
    which has a natural interpretation: the right-hand side of~\eqref{eq:km_lambda_kmminus} is the expected number of $P_m$-extensions of a vertex that has~$k_{m-1}$ many $P_{m-1}$-extensions; hence if some vertex has about $k_{m-1}$ many $P_{m-1}$-extensions, the maximum number of $P_m$-extensions is likely to be at least $\lambda k_{m-1} (1 + o(1) )$. The asymptotics~\eqref{eq:km:asymp} imply that $k_m \sim \lambda k_{m-1}$ if and only if $\lambda \gg \log^{(m)} n$, and thus the message of Theorem~\ref{thm:paths} is that in this regime the optimal strategy to obtain many $P_m$-extensions is `inherited' from $P_{m-1}$, while in the complementing case $\lambda =O(\log^{(m)} n)$ we have a better strategy, genuinely using the length~$m$ of the~path. 
\smallskip

We now give the short deduction of Theorem~\ref{thm:sparseP2} from \refT{thm:paths} and \refP{prop:km_asymp} and the proof of Corollary~\ref{cor:Taone}.
\smallskip
\begin{proof}[Proof of Theorem~\ref{thm:sparseP2}]
  Recall $D$ is defined in \eqref{eq:LLN_maxdeg_sparse} and that $\lambda = pn$. Combine Theorem~\ref{thm:paths} and Proposition~\ref{prop:km_asymp} for $m = 2$. The case $\log \log n \ll pn \ll \log n$ corresponds to $i = 1$ in \eqref{eq:km:asymp} while the case $1 \ll pn = O(\log \log n)$ is obtained by noting that the first case and the third case in \eqref{eq:km:asymp} can be cast in a common expression.
\end{proof}
\smallskip
\begin{proof}[Proof of Corollary~\ref{cor:Taone}]
The upper bound follows trivially from Theorem~\ref{thm:sparseP2}, using the simple inequality $X_{T_{a,1},v} \le (X_{T_{1,1},v})^a$. 

For the lower bound we consider the number $Y_v := {(X_{T_{1,1},v})^a - X_{T_{a,1}, v}}$ of $a$-tuples of $T_{1,1}$-extensions that are not vertex disjoint outside of $v$. 
Using Theorem~\ref{thm:sparseP2} it suffices to show that $Y_v = o(\alpha_n^a)$ whp. 
Note that, having chosen an $(a-1)$-tuple of $T_{1,1}$-extensions, the number of ways to choose one more~$T_{1,1}$-extension that overlaps with one of them, is clearly at most a constant multiple of the maximum degree $\Delta$. Therefore, by using \refT{thm:sparseP2} and \eqref{eq:LLN_maxdeg_sparse}, we infer that, whp,
  \begin{equation*}
    Y_v = O( X_{T_{1, 1}, v}^{a-1} \Delta) = O( \alpha_n^{a-1} D) 
  \end{equation*}
  for any~$v \in [n]$. 
  Simple calculus shows that $D \ll \alpha_n$, completing the proof of~\eqref{eq:Taone}.
\end{proof}

The proof of Theorem~\ref{thm:paths} relies on a reduction of the $P_m$-extension counts in~$\Gnp$ to the size of level~$m$ in the Galton--Watson tree~$\cT_{n,p}$ with binomial offspring distribution. 
The core Lemma~\ref{lem:paths} in Section~\ref{ss:randomtree_level} determines at which value the upper tail of the level~$m$ size in~$\cT_{n,p}$ changes from being much smaller than~$1/n$ to much larger than~$1/n$. 
In Section~\ref{sec:proof:paths} we then use the auxiliary results from Section~\ref{sec:reduction}
to deduce Theorem~\ref{thm:paths} from Lemma~\ref{lem:paths} in a straightforward~way (and also prove~Proposition~\ref{prop:km_asymp}). 

\subsection{Random tree: tails of level sizes}
\label{ss:randomtree_level}
Recall that $\cT_{n,p}$ denotes the Galton--Watson tree with binomial offspring distribution~$\Bin(n,p)$. Given an integer~$m \ge 1$, let $Z_m$ denote the number of vertices of $\cT_{n,p}$ of depth $m$ (also called the size of level~$m$).

\begin{lemma}
  \label{lem:paths}
Fix~$m \ge 1$. Let~$\lambda := pn$, and assume that $1 \le \lambda \ll \log n$. Then, for any constant~$\alpha > 0$,
  \begin{equation}\label{eq:n_alpha}
    \prob{Z_m \ge \alpha k_m(n,\lambda)} = n^{-\alpha + o(1)}\,.
  \end{equation}
\end{lemma}
\begin{proof}
Note that by Proposition~\ref{prop:km_asymp}, 
using monotonicity and~\eqref{eq:km:asymp} it follows that 
    \begin{equation}
        \label{eq:km_infty}
    k_m(n, \lambda) \ge k_m(n, 1) \sim \frac{\log n}{\log^{(m+1)} n} \gg 1 \quad \text{ for any fixed integer~$m \ge 1$}\,.
    \end{equation}
A certain optimization problem will play a key role in our upcoming large deviation type arguments.  
To prepare for this, we now prove that, for any~$\rho > 0$ and~$\beta \in (0,1]$, 
\begin{equation}
  \label{eq:sup}
  \beta \left( 1 + \frac{1}{\rho} \phi \left( \frac{\rho}{\beta \log (1 + \rho)} -1 \right)  \right) \ge 1 \quad \text { with equality for }
 \beta = \beta_\rho := \frac{\rho}{\rho + \phi(\rho)}\,,
\end{equation}
where~$\phi(x)$ is as defined in~\eqref{eq:entropy}. 
To prove this, we treat $\rho$ as fixed and use change of variable~$x := \frac{\rho}{\beta \log (1 + \rho)} - 1$. Note that $x \ge \frac{\rho}{\log (1 + \rho)} - 1 > 0$. Since $\beta = \frac{\rho}{(1 + x)\log(1 + \rho)}$, we can rewrite the inequality in~\eqref{eq:sup} as 
\begin{equation}
\label{eq:phiC}
\frac{\rho + \phi(x)}{(1 + x) \log (1 + \rho)} \ge 1 \,.
\end{equation}
Denoting the left-hand side of \eqref{eq:phiC} by $f_\rho(x)$, straightforward calculus gives 
\[
  f_\rho'(x) = \frac{(1 + x)\log(1+x) - \phi(x) - \rho}{(1 + x)^2\log (1 + \rho)} = \frac{x-\rho}{(1 + x)^2 \log (1 + \rho)}\,,
\]
which implies that on $(0,\infty)$ the function $f_\rho$ is minimized at~$x = \rho$, with value $f_\rho(\rho) = \frac{\phi(\rho) + \rho}{(1 + \rho)\log (1 + \rho)} = 1$.
This implies~\eqref{eq:phiC} and thus~\eqref{eq:sup}, since~${x = \rho}$ corresponds to~${\beta = \beta_\rho}$.

We are now ready to prove estimate~\eqref{eq:n_alpha} using induction on~$m$. In what follows we use the shorthand $k_i = k_i(n,\lambda)$ to avoid clutter. 
In the base case~${m = 1}$ we simply have~${Z_1 \sim \Bin(n,p)}$ and~${k_1 = D(n,p)}$, so that~\eqref{eq:n_alpha} with~$m=1$ follows from~\eqref{eq:n_alpha_upp} and~\eqref{eq:n_alpha_low}.
For the induction step we henceforth assume~$m \ge 2$. 
Fix any constant~$\eps \in (0,1)$ such that $\alpha /\eps$ is an integer. Note that if we condition on~${Z_{m-1} \in [x, y)}$ for some numbers~$x, y$, then the random variable~$Z_m$ is stochastically dominated 
by~$\Bin(\floor{y}n,p)$. Therefore
  \begin{align*}
    \prob{Z_m \ge \alpha k_m} &\le \prob{Z_{m-1} \ge \alpha k_{m-1}} \\
&+ \sum_{1 \le i \le \alpha/\eps}  \prob{Z_{m-1} \in [(i-1)\eps k_{m-1}, i\eps k_{m-1})} \prob{\Bin\left( \floor{i\eps k_{m-1}} n, p \right) \ge \alpha k_m} \,.
  \end{align*}
  Note that by \eqref{eq:km_lambda_kmminus} we have $\alpha k_m \ge \alpha k_{m-1} \lambda \ge i\eps k_{m-1}pn$, so the Chernoff bound~\eqref{eq:Chern_upper} applies to the (upper) binomial tail. By invoking the induction hypothesis for~$Z_{m-1}$ (together with the trivial inequality $\prob{Z_{m-1} \ge 0} \le 1 = n^{o(1)}$ in the case~$i = 1$), it follows~that
\begin{equation}\label{eq:Zm_bound}
\begin{split}
    &\prob{Z_m \ge \alpha k_m} \\
     &\le n^{-\alpha + o(1)} + \sum_{1 \le i \le \alpha/\eps} n^{-(i-1)\eps + o(1)} \cdot \exp \left( -i\eps k_{m-1} \lambda \phi \left( \frac{\alpha k_m}{i \eps k_{m-1} \lambda} - 1 \right)  \right)\,,
\end{split}
\end{equation}
  where we omitted the rounding to integers because the function $\mu \mapsto \mu \phi(x/\mu -1 ) = x \log (x/\mu) - x + \mu$ is decreasing on $(0, x]$ (as can be seen by calculating the derivative). 
Recall the definition of $k_m$ from \eqref{eq:km:def}. The logarithm of the $i$th term in \eqref{eq:Zm_bound} is
  \begin{align*}
    -i\eps \log n &\left( 1 + \frac{k_{m-1} \lambda}{\log n} \phi \left( \frac{\alpha \log n}{i\eps \lambda k_{m-1} \log \left( 1 + \frac{\log n}{\lambda k_{m-1}} \right)} - 1 \right) \right)  + (\eps + o(1))\log n\\
    \justify{$\rho_n := \frac{\log n}{\lambda k_{m-1}}, \beta := i\eps /\alpha$}    &= - (\alpha \log n) \cdot \beta \left( 1 + \frac{1}{\rho_n}\phi \left( \frac{\rho_n}{\beta \log (1 + \rho_n)} - 1 \right) \right) + (\eps + o(1))\log n\\
    \justify{\eqref{eq:sup}}    &\le - (\alpha \log n) + (\eps + o(1))\log n \,.
  \end{align*}
  Since $\eps, \alpha$ are constants, using~\eqref{eq:Zm_bound} we obtain~that 
  \begin{equation*}
    \prob{Z_m \ge \alpha k_m}  \le n^{-\alpha + o(1)} + \sum_{1 \le i \le \alpha/\eps} n^{-\alpha + \eps + o(1)} \le (1+\alpha/\eps) \cdot n^{-\alpha + \eps + o(1)} \le n^{-\alpha + \eps + o(1)} \,,
  \end{equation*}
which establishes the upper bound of~\eqref{eq:n_alpha} since~$\eps > 0$ was~arbitrary.

To complete the proof of the induction step~$m \ge 2$, it remains to establish the lower bound of~\eqref{eq:n_alpha}. Define $\alpha_n$ so that $\alpha_n k_m = \ceil{\alpha k_m}$. Since $k_m \to \infty$ by \eqref{eq:km_infty}, we have 
\begin{equation}
  \label{eq:alpha_conv}
  \alpha_n \to \alpha\,.
\end{equation}
With an eye on the form of~\eqref{eq:sup} and~\eqref{eq:Zm_bound}, with foresight we~set
\begin{equation}\label{def:DnGamman}
\rho_n := \frac{\log n}{\lambda k_{m-1}}  
\quad \text{ and } \quad 
\gamma_{n} := \alpha_n \beta_{\rho_n}, 
\end{equation}
where~$\beta_{\rho_n} = \rho_n/(\rho_n + \phi(\rho_n))$ is as in~\eqref{eq:sup}. 

Since $\gamma_n$ is (in general) not a constant, some care is needed when applying the induction hypothesis. We claim that we still~have
\begin{equation}
  \label{eq:ind_hyp_gener}
  \prob{Z_{m-1} \ge \gamma_n k_{m-1}} \ge n^{-\gamma_n + o(1)}.
\end{equation}
To see this, note that~${\limsup_{n \to \infty} \gamma_n \le \alpha}$ (because $\beta_{\rho_n} < 1$ and $\alpha_n \to \alpha$). By the subsubsequence principle it suffices to prove~\eqref{eq:ind_hyp_gener} under the assumption that~${\gamma_n \to \gamma \in [0, \alpha]}$. 
For any constant~$\eps > 0$, the induction hypothesis then implies that, for sufficiently large~$n$,
\begin{equation*}
  \prob{Z_{m-1} \ge \gamma_n k_{m-1}} \ge \prob{Z_{m-1} \ge (\gamma + \eps) k_{m-1}} \ge n^{-(\gamma + \eps) + o(1)} = n^{-(\gamma_n + \eps) + o(1) }.
\end{equation*}
Since~$\eps > 0$ was arbitrary, this implies the claimed lower bound~\eqref{eq:ind_hyp_gener}.

Conditioning on $Z_{m-1} \ge \gamma_n k_{m-1}$ we have that $Z_m$ stochastically dominates a random variable $X \sim \Bin(\ceil{\gamma_n k_{m-1}}n, p)$. We claim that
\begin{equation}
\label{eq:Zm_lower}
  \prob{X \ge \alpha_n k_m} \ge n^{o(1)}\exp \left( -\gamma_n k_{m-1}\lambda \phi \left( \frac{\alpha_n k_m}{\gamma_n k_{m-1} \lambda} - 1 \right) \right).
\end{equation}
The rounding of $\gamma_n k_{m-1}$ causes a technical issue which we overcome by considering two cases: (i) $\E X = \ceil{\gamma_n k_{m-1}}\lambda \le \alpha_n k_m$ and (ii) $\E X = \ceil{\gamma_n k_{m-1}}\lambda > \alpha_n k_m$.
Recalling that $\lambda \ll \log n$, an easy argument by induction on $i$ shows that 
\begin{equation}
  \label{eq:km_upper}
  k_i \ll (\log n)^i, \quad i = 1, 2, \ldots .
\end{equation}
In the case~(i), recalling that $\alpha_n k_m$ is an integer, Proposition~\ref{prop:lower_poisson} implies
\begin{align*}
  \log \prob{X \ge \alpha_n k_m} &\ge - \ceil{\gamma_n k_{m-1}}\lambda \phi \left( \frac{\alpha_n k_m}{\ceil{\gamma_n k_{m-1}} \lambda} - 1 \right) + O\left( \log \alpha_n k_m \right) \\
  \justify{$\mu \mapsto \mu \phi(x/\mu - 1)$ decreases on $(0,x]$; \eqref{eq:km_upper}, \eqref{eq:alpha_conv}} &\ge -\gamma_n k_{m-1}\lambda \phi \left( \frac{\alpha_n k_m}{\gamma_n k_{m-1} \lambda} - 1 \right) + O \left( \log \log n \right),
\end{align*}
which implies \eqref{eq:Zm_lower} in the case~(i).
In the somewhat degenerate case~(ii),  
we recall that $\alpha_n k_m$ is an integer, and therefore is at most $\floor{\E X}$, which, as is well known, is at most the median of $X$. Hence
\begin{equation*}
  \prob{X \ge \alpha_n k_m} \ge \prob{X \ge \floor{\E X}} \ge 1/2 = n^{o(1)},
\end{equation*}
which is at least the right-hand side of \eqref{eq:Zm_lower}, because $\phi$ takes only nonnegative values.

Finally we are ready to infer the lower bound of \eqref{eq:n_alpha} as follows:
\begin{equation*}
\begin{split}
  &\log \prob{Z_m\geq\alpha k_m} = \log \prob{Z_m \ge \alpha_n k_m} \\
  &\geq \log \big( \prob{Z_{m-1} \ge \gamma_n k_{m-1}} \cdot \prob{X \ge \alpha_n k_m} \big)\\
    \justify{\eqref{eq:ind_hyp_gener}, \eqref{eq:Zm_lower}} & \ge -\gamma_n \log n +o( \log n) -\gamma_n k_{m-1}\lambda\phi\left(\frac{\alpha_n k_m}{\gamma_n k_{m-1}\lambda}-1 \right) \\
    \justify{\eqref{eq:km:def}} &= -(\gamma_n \log n)\left( 1 + \frac{\lambda k_{m-1}}{\log n} \phi \left( \frac{\alpha_n \log n}{\gamma_n \lambda k_{m-1} \log \left( 1 + \frac{\log n}{\lambda k_{m-1}} \right)} - 1 \right) \right) + o(\log n) \\
    \justify{\eqref{def:DnGamman}} &= -(\alpha_n \log n)\beta_{\rho_n} \left( 1 + \frac{1}{\rho_n} \phi \left( \frac{\rho_n}{\beta_{\rho_n} \log \left( 1 + \rho_n \right)} - 1\right) \right) + o(\log n) \\
    \justify{\eqref{eq:sup}} &= -\alpha_n \log n + o(\log n) \\
    \justify{\eqref{eq:alpha_conv}}   &= -(\alpha + o(1)) \log n , 
\end{split}
\end{equation*} 
which completes the induction step and thus the proof of \refL{lem:paths}. 
\end{proof}

\subsection{Proof of main result for paths}\label{sec:proof:paths}
In this section we first deduce Theorem~\ref{thm:paths} from Lemma~\ref{lem:paths} and Proposition~\ref{prop:km_asymp},
and then give the deferred proof of Proposition~\ref{prop:km_asymp}.
\begin{proof}[Proof of Theorem~\ref{thm:paths}]
It suffices to show that, for any constant~$\eps \in (0,1)$, we whp~have 
\begin{equation}
\label{eq:path_LDP}
    (1 - \eps) k_m(n, \lambda) \le M_n \le (1 + \eps) k_m(n, \lambda) .
\end{equation}

We start with the lower bound in~\eqref{eq:path_LDP}.
Let $n^* = n^*(n)$ be the sequence from Lemma~\ref{lem:sparse_lower_GW}. 
Since~$T = P_m$ is the $m$-edge path, the random variable~$f_T(\cT_{n^*,p})$ is exactly the size of the $m$-th level of~$\cT_{n^*, p}$.
Lemma~\ref{lem:paths} with $n=n^*$ implies, for any constant~$\eps \in (0,1)$, that
  \[
    \prob{f_T(\cT_{n^*,p}) \ge (1 - \eps/2)k_m(n^*, pn^*)} \ge (n^*)^{-(1-\eps/2)+ o(1)} \gg (\log n)^{m+1}/n.
  \]
Using the estimate~\eqref{eq:km:asymp}, it is straightforward to check that~$1 \ll pn = \lambda \ll \log n$ and~$n^* \sim n$ imply $k_m(n^*,pn^*) \sim k_m(n, \lambda)$. 
Applying Lemma~\ref{lem:sparse_lower_GW} it thus follows that,~whp, 
  \[
M_n \ge (1 - \eps/2) k_m(n^*, pn^*) \ge (1 - \eps) k_m(n, \lambda)\,,
  \]
which establishes the lower bound in~\eqref{eq:path_LDP}. 

We now turn to the upper bound in~\eqref{eq:path_LDP}. 
Since the random variable~$f_T(\cT_{n,p})$ is the size of the $m$-th level of~$\cT_{n,p}$, Lemma~\ref{lem:paths} implies for any constant~$\eps \in (0,1)$ that 
  \[
    \prob{f_T(\cT_{n,p}) \ge (1 + \eps)k_m(n, \lambda)} \le n^{-(1+\eps) + o(1)} \ll 1/n.
  \]
  The assumption~$\lambda \ll \log n$ implies~$k_1/\lambda = \frac{\log n}{\lambda}/\left(  \log \frac{\log n}{\lambda} \right) \to \infty$, therefore
  using inequality~\eqref{eq:km_lambda_kmminus} we readily infer that~$k_m \ge \lambda^{m-1}k_1 \gg \lambda^m =(pn)^{e_T}$. 
Applying Lemma~\ref{lem:sparse_upper_GW} it thus follows that,~whp, $M_n \le (1+\eps) k_m(n, \lambda)$, which establishes the upper bound in~\eqref{eq:path_LDP},
completing the proof of~Theorem~\ref{thm:paths}. 
\end{proof}

\begin{proof}[Proof of Proposition~\ref{prop:km_asymp}]
Monotonicity of~$k_m$ follows by induction: this is easy for $m=1$, and monotonicity of $k_m$ then follows from monotonicity of~$k_{m-1}$ and $\log$. 

  To prove \eqref{eq:km:asymp} we also use induction on $m$.
  The base case~${m=1}$ is immediate from the definition of $k_1$. In the induction step~${m \ge 2}$ we employ a case distinction.
If ${1 \le \lambda = O(\log^{(m)} n)}$ holds, then by invoking the induction hypothesis (using~$1 \le \lambda = O(\log^{(m)} n) \ll \log^{(m-1)} n$)  it follows~that 
\[
\frac{\log n}{\lambda k_{m-1}(n,\lambda)} \sim \frac{\log \frac{\log^{(m-1)} n}{\lambda}}{\lambda} \sim \frac{\log^{(m)} n }{\lambda} \,,
\]
which in view of~\eqref{eq:km:def} yields the claimed asymptotics~\eqref{eq:km:asymp} in the cases $1 \le \lambda \ll \log^{(m)}n$ and $\lambda \sim (\log^{(m)} n)/C$ for fixed positive $C$. 
If~${\lambda \gg \log^{(m)} n}$ holds, then by invoking monotonicity and the induction hypothesis (using~$1 \ll \log^{(m)} n \ll \log^{(m-1)} n$) it follows that 
\[
 \lambda k_{m-1}(n,\lambda) \gg (\log^{(m)} n) \cdot k_{m-1}\bigl(n,\log^{(m)} n\bigr) \sim \frac{(\log^{(m)} n) \log n}{\log \frac{\log^{(m-1)} n}{\log^{(m)} n}} \sim \log n.
\] 
Hence, combining the definition~\eqref{eq:km:def} and the asymptotics~$\log (1 + x) \sim x$ as~$x \to 0$ with the induction hypothesis, it follows that
\begin{equation*}
\begin{split}
k_m(n,\lambda) & \sim \lambda k_{m-1} (n,\lambda) \\
& \sim 
      \begin{cases}
	\frac{\lambda^{m-i}\log n}{\log \frac{\log^{(i)} n }{\lambda}} \quad &\text{if } \log^{(i+1)} n \ll \lambda \ll \log^{(i)} n \text{ with } 1 \le i \le m-1\,,\\
	\frac{\lambda^{m-i}\log n}{\log (1 + C)}  \quad &\text{if } \lambda \sim (\log^{(i)} n)/C \text{ with } 2 \le i \le m-1, 
      \end{cases}
\end{split}
\end{equation*}
which yields the claimed asymptotics~\eqref{eq:km:asymp} in the remaining cases.
\end{proof}

\section{Spherically symmetric trees: the sparse case}
\label{sec:symm_trees}
In this section we deal with extension counts of spherically symmetric trees~$T_{a,b}$ (defined in Section~\ref{sec:res2}) in the `sparse' case ${1 \ll pn \ll \log n}$. 
In particular, we state and prove Theorem~\ref{thm:symm_trees} below, which generalizes Theorem~\ref{thm:symm_trees_abr} by including the intermediate regime $pn \asymp {\left( \log n / \log \log n\right)^{1 - 1/b}}$ not covered by~\eqref{eq:sph_symm_denser} and \eqref{eq:sph_symm}, 
in which case the optimal strategy  interpolates between the two distinct strategies leading to~\eqref{eq:strategy1} and \eqref{eq:strategy2}.  
As the reader can guess, the proof of Theorem~\ref{thm:symm_trees} is complicated by the fact that we need to rule out the existence of other `better' strategies (for creating a vertex with a large extension count). 

We start by recalling that the maximum degree~$\Delta$ of~$\Gnp$ is concentrated around~$D = D(n,p)$ defined in~\eqref{eq:LLN_maxdeg_sparse}.
It is routine to check that the condition $pn \asymp \left( \log n / \log \log n\right)^{1- 1/b}$ is equi\-valent to $D^{b-1} \asymp (pn)^b$, so in \refT{thm:symm_trees} below we do not lose generality by assuming that~$D^{b-1}/(pn)^b$ converges to some limit~${L  \in [0,\infty]}$.

In Theorem~\ref{thm:symm_trees} we refer to the set $\Lambda$ defined in~\eqref{eq:Lambda}, and functions $f_{a,b}$ defined in~\eqref{eq:deg_func}. Note that $f_{0,b} \equiv 1$, so that the term for $m=0$ in \eqref{eq:FL_def} equals $x_0^a$.
  \begin{theorem}[Maximum for trees~$T_{a,b}$, sparse case]
    \label{thm:symm_trees}
Fix~$T=T_{a,b}$, with~${a\ge 1}$ and~${b\ge 2}$. 
Suppose that~$1 \ll pn\ll \log{n}$ and~$D^{b-1}/(pn)^b \to L \in [0, \infty]$. 
Then
\begin{equation}\label{eq:thm:symm_trees:conv}
    \frac{M_n}{\alpha_n} \pto 1 ,
  \end{equation}
where, for a suitable constant $C_{a,b} \in [1,\infty)$, we have 
\begin{equation}\label{eq:thm:symm_trees:conv:asymp}
    \alpha_n := \begin{cases}
    [D (pn)^{b}]^a, \quad &\text{if } L \in [0,1]\,, \\
    [D(pn)^b]^a \sup_{(x_0, \dots, x_k) \in \Lambda} F_L(x_0, \dots, x_k) \quad &\text{if }L \in (1, C_{a,b})\,,\\
    D^{ab} \sup_{(x_1, \dots, x_k) \in \Lambda} f_{a,b}(x_1, \dots, x_k), \quad &\text{if }L \in [C_{a,b}, \infty] ,
    \end{cases}
  \end{equation}
where 
\begin{equation}
\label{eq:FL_def}
  F_L(x_0, \dots, x_k) := \sum_{m = 0}^a \binom{a}{m} L^m x_0^{a-m} f_{m,b}(x_1, \dots, x_k).
\end{equation}
  \end{theorem}

  \begin{remark} 
	  \label{rem:optimization}In the latter two cases the (asymptotic) extreme value $\alpha_n$ depends on the optimal value of an optimization problem.  This problem, especially in the final case, appears quite elementary, as it simply depends on optimizing multivariate polynomials over simple convex sets.  However, we do not know how to find the optimum in general, and it would be interesting to know these optimal values.  It is quite possible that the optimum is always achieved by constant vectors $(1/k,\dots ,1/k)$, for some $k$.  The optimization problem is discussed further in the appendix to the arXiv version of this paper. Added in proof: in the latest arXiv version we include an AI-aided proof that indeed a constant vector maximizer exists.
\end{remark}
  \begin{proof}[Proof of \refT{thm:symm_trees_abr}]
    Recalling the definition \eqref{eq:LLN_maxdeg_sparse} of $D$, we claim that whenever $pn \gg \left( \log n/\log \log n \right)^{1-1/b}$ then $D^{b-1}/(pn)^b \to 0$. This is straightforward under a stronger assumption $pn \gg (\log n)^{1 - 1/b}$; otherwise we can assume $pn = O((\log n)^{1 - 1/b})$ which implies that $D = O( \frac{\log n}{\log \log n} ) \ll (pn)^{b/(b-1)}$ from which the claim follows by taking power $b-1$. Furthermore $pn \ll \left( {\log n}/{\log \log n} \right)^{1-1/b}$ implies $D^{b-1}/(pn)^b \to \infty$.
    In view of this, now~\eqref{eq:thm:symm_trees:conv} and \eqref{eq:thm:symm_trees:conv:asymp} of \refT{thm:symm_trees} 
    imply \refT{thm:symm_trees_abr}.
  \end{proof}
  Let us provide heuristics for the denominator $\alpha_n$ in Theorem~\ref{thm:symm_trees} in the case $ L \in (0, \infty)$. We will see that a combined strategy gives a lower bound that interpolates between the bounds \eqref{eq:strategy1} and \eqref{eq:strategy2} discussed in the introduction.
      Introducing an extra variable $x_0$, consider a vector $(x_0, x_1, \dots, x_k)$ of positive numbers with $x_0 + \dots + x_k = 1$. It can be shown that with probability $n^{-x_0 + o(1)}$ a vertex has $x_0D$ neighbors of degree $(1 + o(1))pn$. Moreover, with probability $n^{-(x_1 + \ldots + x_k) + o(1)}$ a vertex has $k$ neighbors with their degrees at least $x_1D, \ldots, x_k D$. The probability that a vertex has both types of neighbors is therefore $n^{-(x_0 + \ldots + x_k) + o(1)} \approx n^{-1}$ and thus we expect that such a vertex exists.
      Classifying the extensions according to the number~$m$ of children of the root that are mapped to vertices of degrees $x_1 D, \dots, x_k D$ (and the remaining $a-m$ to the $x_0 D$ vertices of degree~$(1 + o(1))pn$), we thus ought to be able to find a vertex $v$ such that, whp,
    \begin{align*}
      \notag X_v &\ge (1 + o(1))\sum_{m=0}^a \binom{a}{m} (x_0D)^{a-m} (pn)^{(a-m)b} f_{m,b}(x_1, \dots, x_k) D^{mb} \\
       &\sim [D(pn)^b]^a \sum_{m = 0}^a \binom{a}{m} \left[ \frac{D^{b-1}}{(pn)^b} \right]^m x_0^{a-m} f_{m,b}(x_1, \dots, x_k) \\
\notag      &\sim [D(pn)^b]^a F_L(x_0, \ldots, x_k).
    \end{align*}
    Taking the supremum over all vectors $(x_0, x_1, \dots)$ with sum one, we thus ought to be able to find a vertex~$v$ such that, whp, 
    \begin{equation}
      \label{eq:heur_lower}
      X_v \ge (1 + o(1))[D(pn)^b]^a \sup_{(x_0, \dots, x_k) \in \Lambda} F_L(x_0, \dots, x_k).
    \end{equation}

    The following proposition claims that the supremum in \eqref{eq:heur_lower} can be simplified if the limit~$L$ is sufficiently small or large. It would be interesting to know if equality in~\eqref{eq:FL_lower} actually holds for all~$L$, since then Theorem~\ref{thm:symm_trees} would take a simpler form.
\begin{proposition}
  \label{prop:optimizer}
  Let $a \ge 1$ and $b \ge 2$ be integers.
  For every $L \in [0,\infty)$ we have
  \begin{equation}
    \label{eq:FL_lower}
    \sup_{(x_0, \dots, x_k) \in \Lambda} F_L(x_0, \dots, x_k) \: \ge \:  \max \Bigcpar{ 1, \ L^a \sup_{(x_1, \dots, x_k) \in \Lambda} f_{a,b}(x_1, \dots, x_k)} .
  \end{equation}
  Moreover, there is a constant $C_{a,b} \ge 1$ such that 
  \begin{equation}
  \label{eq:FL_cases}
  \sup_{(x_0, \dots, x_k) \in \Lambda} F_L(x_0, \dots, x_k)= 
    \begin{cases}
    1  \quad & \text{if } L \in [0,1]\,, \\
    L^a \sup_{(x_1, \dots, x_k) \in \Lambda} f_{a,b}(x_1, \dots, x_k) \quad &\text{if } L \in [C_{a,b}, \infty) .
    \end{cases}
  \end{equation}
\end{proposition}
\begin{proof}
Inequality \eqref{eq:FL_lower} follows by setting the first argument of~$F_L$ to~${x_0 = 1}$, or setting~${x_0 = 0}$ and optimizing over~$(x_1, \dots) \in \Lambda$. 

To prove \eqref{eq:FL_cases}, writing 
\[
\lambda_m := \sup_{(x_1, \dots, x_k) \in \Lambda} f_{m,b}(x_1, \dots, x_k),\]
we claim that, for every $(x_0, \dots, x_k) \in \Lambda$, we have $f_{m,b}(x_1, \dots, x_k) \le (1 - x_0)^m \lambda_m$. This claim is trivial when $x_0 = 1$ and when $x_0 < 1$ it follows by noting~that
\begin{equation*}
  f_{m,b}(x_1, \dots, x_k) = (1-x_0)^{mb}f_{m,b}\left(\frac{x_1}{1-x_0}, \dots, \frac{x_k}{1-x_0}\right)
  \le (1-x_0)^{mb} \lambda_m \le (1-x_0)^m \lambda_m.
\end{equation*}
Hence, fixing $x_0$ and bounding each term of $F_L$ separately, we get that 
\begin{align*}
  \sup_{(x_0, \dots, x_k) \in \Lambda} F_L(x_0, \dots, x_k) &\le \max_{x_0 \in [0,1]} \sum_{m=0}^a \binom{a}{m} x_0^{a-m} (1-x_0)^m  L^m \lambda_m \\
  &\le \max 
	\bigcpar{  L^m \lambda_m \: : \: m = 0, \dots, a},
\end{align*}
where the last inequality follows because the sum is a convex combination of numbers~$L^m \lambda_m$.
It is easy to show that $\lambda_0 = \lambda_1 = 1$ and that $\lambda_1 \ge \lambda_2 \ge \dots \ge \lambda_a > 0$.
Hence for $L \le 1$, the number $L^m \lambda_m$ is maximized by $m = 0$, together with lower bound~\eqref{eq:FL_lower} giving $\sup_{(x_0, \dots, x_k) \in \Lambda} F_L(x_0, \dots, x_k) = 1$, while for sufficiently large $L$ we have that maximum is $L^a \lambda_a$, giving the second case in \eqref{eq:FL_cases}.
\end{proof}

%

Our proof of \refT{thm:symm_trees} exploits the auxiliary results from Section~\ref{sec:reduction}, which allow us to focus on the number of ${T_{a,b}}$-extensions of the root in~$\cT_{n,p}$, a Galton--Watson tree with offspring distribution~$\Bin(n,p)$.   
 These extensions may be classified depending on the degrees of the children of the root. In Subsection~\ref{ss:high} we prove bounds related to high degree vertices, and in Subsection~\ref{ss:intermediate} we prove bounds related to intermediate degree vertices. In Subsection~\ref{ss:symm_upper} we use these bounds to prove the upper bound of Theorem~\ref{thm:symm_trees} and in Subsection~\ref{ss:symm_lower} we give a self-contained proof of the lower bound of Theorem~\ref{thm:symm_trees}.

To discuss $\cT_{n,p}$ throughout the whole section, we introduce the following notation. Let~$\neighs \subseteq [n]$ denote the set of indices of the children of the root (i.e., every element of $[n]$ is included in $\neighs$ independently with probability $p$), and let ${\xi_1, \dots, \xi_n}$ be independent random variables with distribution $\Bin(n,p)$ that are also independent of $[n]_p$, so that the degrees of the children of the root are~${(\xi_i : i \in \randset{n}{p})}$.

\subsection{High degree neighbors}
\label{ss:high}
It may well seem natural to consider vertices of large degree being of degree at least~$\eta D$ for some constant~$\eta\in (0,1)$, but we find it more convenient to take
\begin{equation}
    \label{eq:def_eta}
\eta\, :=\, \frac{1}{(\log{\log{n}})^5}\, .
\end{equation}
\refL{lem:large} is the main result of this subsection: it provides convenient bounds on
\begin{equation*}
  H = H(n,p)\, :=\, \sum \{ \xi_i : i \in \neighs, \xi_i \ge \eta D \} \, ,
\end{equation*}
which, in concrete words, is the sum of the degrees of high degree neighbors of the root.
\begin{lemma}
  \label{lem:large} 
Fix~$\theta\in(0,1)$ and $\eps\in (0,1]$.
If~$pn\le (\log{n})^{\theta}$, then
\begin{equation}\label{eq:lem:large:1} 
\prob{H\, \ge\, (1+\eps)D}\, \ll\, 1/n\, ,
\end{equation}
and, for any constant $x\in (0,1 + \eps)$, we also have 
\begin{equation}\label{eq:lem:large:2} 
  \prob{|\neighs|\, \ge\, xD, \;  H\, \ge\, (1-x+\eps)D}\, \ll\, 1/n\, .
\end{equation}
\end{lemma}
In the upcoming proof of~\refL{lem:large}, 
the discussed choice~\eqref{eq:def_eta} of~$\eta$ enables us to use the following simple fact about the tail probabilities of $\xi\sim \Bin(n,p)$. 
\begin{fact} Fix $\theta \in (0, 1)$. If $pn \le (\log n)^\theta$, then for any $x \in [\eta/2, 2]$ we have
  \begin{equation}
\label{eq:xiD}
    \prob{\xi\ge xD}\, \le  n^{-x(1+o(1))} \,.  
  \end{equation}
\end{fact}
\begin{proof}
  Apply inequality \eqref{eq:n_alpha_upp} of Lemma~\ref{lem:n_alpha} with $\alpha = x$.
\end{proof}
\medskip
\begin{proof}[Proof of Lemma~\ref{lem:large}] We shall tacitly assume several times that $n$ is larger than a suitable constant, which possibly depends on $\theta$ and $\eps$. 
We start with inequality~\eqref{eq:lem:large:1}. 
If the event $H\, \ge\, (1 +\eps)D$ occurs, then there exist positive numbers $y_1,\dots ,y_k$ such that
\begin{enumerate}
\item[(i)] $\sum_{j=1}^{k}y_j \ge 1 +\eps$,
\item[(ii)] $k \le 4\eta^{-1}$,
\item[(iii)] $y_j\ge \eta$ for all $j \in [k]$, and
\item[(iv)] there exist distinct 
$i_1,\dots ,i_k \in \neighs$ such that $\xi_{i_j} \ge y_{j}D$ for all $j \in [k]$.
\end{enumerate}
(Note that (ii) holds because choosing the $y_j$s minimally with respect to (i), we can assume $\sum_j y_j \le 4$, say.) 
With foresight, set $\gamma:=\eps\eta/16=\eps/[16(\log{\log{n}})^{5}]$. 
By rounding down if necessary, there exists a vector $\xx = (x_1,\dots ,x_k)$ of positive numbers such that
\begin{enumerate}
\item[(i)] $\sum_{j=1}^{k}x_j\ge 1  +3\eps/4$, 
\item[(ii)] $k \le 4\eta^{-1}$,
\item[(iii)] $x_j\ge \eta/2$ and $x_j$ is a multiple of $\gamma$ for all $j \in [k]$, and
\item[(iv)] there exist distinct 
$i_1,\dots ,i_k\in \neighs$ such that $\xi_{i_j}\ge x_{j}D$ for all~$j \in [k]$.
\end{enumerate}
Denote the event in (iv) by $\cE_{\xx}$.
As the number of choices of $\xx$ satisfying (i)---(iii) minimally is at most 
\[
((1 + 3\eps/4)\gamma^{-1})^{4\eta^{-1}}\le e^{(\log \log n)^6} = n^{o(1)} \,,
\]
to establish~\eqref{eq:lem:large:1} it thus suffices to prove that $\prob{\cE_{\xx}}\le n^{-1-\eps/2}$ for all sufficiently large $n$ and any $\xx$ satisfying (i)---(iii). 
Using the union bound, inequality \eqref{eq:xiD}, and independence of $\xi_i$s, it follows that 
\begin{align*}
\prob{\cE_{\xx}}\, &\le\, n^k p^k \cdot \prod_{j=1}^{k}n^{-x_j(1+o(1))}   \le\, (\log{n})^k \cdot n^{-(1+o(1))\sum_{j}x_j} \\
& \le\, n^{o(1)} \cdot  n^{-(1+o(1))(1+3\eps/4)} \le\, n^{-1-\eps/2}
\end{align*}
for all sufficiently large~$n$, completing the proof of inequality~\eqref{eq:lem:large:1}. 

The proof of inequality~\eqref{eq:lem:large:2} is essentially identical. Indeed, this time we argue that $H \ge (1 - x + \eps)D$ implies that there is a vector $\xx$ satisfying the four properties with, say, $1 - x + 3\eps/4$ instead of $1 + 3\eps / 4$ in the assumption (i). If we fix a particular $k$-tuple $i_1, \dots, i_k$ and condition on the event in (iv), the probability that the root has at least $(x+o(1))D$ other neighbors is at most $(n-k)^{-x + o(1)} = n^{-x+o(1)}$ by \eqref{eq:xiD}. By inserting this extra factor $n^{-x+o(1)}$ into our estimate of $\prob{\cE_{\xx}, |[n]_p| \ge xD}$, the exponents $-x$ and $x$ cancel. Finishing the proof as we did for \eqref{eq:lem:large:1} then yields a bound $n^{-1 - 3\eps/4 + o(1)} = o(n^{-1})$, as~required. 
\end{proof}

%
%

\subsection{Upper bound on extension counts using intermediate degrees}
\label{ss:intermediate}
We next consider the contribution of neighbors of the root with intermediate degree, between $(1+\eps)pn$ and $\eta D$, where $\eps > 0$ is fixed and $\eta$ is as in~\eqref{eq:def_eta}. 
\refL{lem:crude} is the main result of this subsection: it will later be used to show that contribution from $T_{a,b}$-extensions that use neighbors of intermediate degree is negligible.
\begin{lemma}
  \label{lem:crude}  
  Let~$a \ge 1$ and~$b\ge 2$ be integers. Fix~${\theta<1}$. Suppose that $1\le pn\le (\log{n})^{\theta}$.   
  For any constants $\delta,\eps>0$, with probability $1 - o(1/n)$ the root of $\cT_{n,p}$ has at most 
\begin{equation}
\label{eq:lem:crude2}
  \delta \left( \max\{D(pn)^{b},D^{b}\} \right)^a
\end{equation}
many $T_{a,b}$-extensions in which at least one of the neighbors of the root has intermediate degree, i.e., in the interval $[(1+\eps)pn, \eta D)$.
\end{lemma}

The upcoming proof of \refL{lem:crude} is based on the following two facts (Facts~\ref{fact:int} and~\ref{fact:midd}). 
\begin{fact}
  \label{fact:int} 
  Suppose that $1\le pn\le \log{n}$. For every $\eps>0$ there exists~$C>0$ such that 
  \begin{equation*}
    \prob{ |\{ i \in \neighs : \xi_i \ge (1 + \eps)pn\}| \ge C (pn)^{-1} \log n} = O(n^{-2})\,.
  \end{equation*}
\end{fact}
\begin{proof} 
Writing~$I_{\eps} := \{ i \in [n] : \xi_i\ge (1+\eps)pn \}$, our goal is to prove that 
\[
  \prob{|\neighs\cap I_{\eps}|\ge C(pn)^{-1}\log{n}}\, \le \, n^{-2}
\]
for some constant~$C>0$ and sufficiently large $n$.
By the Chernoff bound~\eqref{eq:Chern_upper} there is a constant $c = c(\eps) >0$ such that $\prob{i\in I_{\eps}}\le \e^{-cpn}$.  
Write $s := \ceil{C(pn)^{-1}\log{n}}$. Note that
\begin{equation}
\label{eq:supremum}
    \sup_{pn \ge 1} (pn)^2\e^{-cpn/2} < \infty
\end{equation}
Using the independence of $\xi_i$s and elements of $\neighs$, it follows~that 
\begin{align*} 
  \prob{|\neighs\cap I_{\eps}|\ge C(pn)^{-1}\log{n}}\, &\le \, \binom{n}{s}p^s \bigl(\e^{-cpn}\bigr)^s\\
&\le\, \left(\frac{pn \e^{1-cpn}}{s}\right)^s\\
\justify{\eqref{eq:supremum} and $n$ large enough} &\le\, \e^{-cpns/2}\\
\justify{$C \ge 4/c$} &\le\, n^{-2} \,,
\end{align*}
which completes the proof of Fact~\ref{fact:int}, as discussed. 
\end{proof}
\begin{fact}
   \label{fact:midd}  
If $1\le pn\le \log{n}$, then 
\begin{equation*}
  \max_{d \ge 64pn} \prob{\cE_d} = O(n^{-2})\,,
\end{equation*}
where~$\cE_d := \{ |\{ i \in \neighs : \xi_i \ge d\}| \ge d^{-1}\log n \}$ for any integer~$d \ge 1$\,. 
\end{fact}

\begin{proof}
Let $d = d(n) \ge 64pn$ be the integer which maximizes the probability $\prob{\cE_d}$.
Writing~$I_d := \{ i \in [n] : \xi_i\ge d\}$, 
 our goal is to prove that 
\[
  \prob{|\neighs\cap I_d|\ge d^{-1} \log n}\, \le n^{-2} \, ,
\]
for sufficiently large $n$.
As $d\ge 64pn$, it follows from the Chernoff bound~\eqref{eq:Chern_upper_simpler} that
\[
  \prob{i \in I_d} = \prob{\xi_i \ge d} \, \le\, \exp(-d\log(d/\e pn))\, \le\, \e^{-3d}\, .
\]
Write~$s := \ceil{d^{-1}\log{n}}$. Using independence of $\xi_i$s and elements of $\neighs$, it follows~that
\begin{align*}
  \prob{|\neighs\cap I_d|\ge d^{-1}\log{n}}\, &\le\, \binom{n}{s}p^s\bigl(\e^{-3d}\bigr)^s\\
&\le\, \left(\frac{pn \e^{1-3d}}{s}\right)^s\\
\justify{$pn \le d$, $\sup_{d \ge 1} d^2\e^{-d} < \infty$, large $n$} &\le\, \exp(-2ds)\\
&\le\, n^{-2}\,,
\end{align*}
which completes the proof of Fact~\ref{fact:midd}, as discussed. 
\end{proof}

\begin{proof}[Proof of Lemma~\ref{lem:crude}]
 Suppose that the number of $T_{1,b}$-extensions of the root is $t_{b}$, and the number of $T_{1,b}$-extensions of the root in which the neighbor of the root has intermediate degree is $t'_b$.  
We claim that for some constant $C_b$ with probability $1 - o(1/n)$
\begin{equation}
\label{eq:onecrude}
t_b \le  C_b\max\{D(pn)^b,D^{b}\}\,,
\end{equation}
and with probability $1 - o(1/n)$
\begin{equation}
\label{eq:interm}    
t'_b\, \le \, \delta D^{b} \,.
\end{equation}
Since the number of $T_{a,b}$-extensions that we consider is at most $a t'_b (t_b)^{a-1}$, 
bounds \eqref{eq:onecrude} and \eqref{eq:interm} will imply the bound \eqref{eq:lem:crude2} up to adjusting the value of $\delta$.

\smallskip
We start with the proof of \eqref{eq:onecrude}. 
For this note that the number of $T_{1,b}$-extensions of the root is at most $\sum_{i\in \neighs}\xi_i^b$.  We bound contributions to this sum from degrees of different sizes.  All the bounds stated hold with probability $1-o(n^{-1})$, and so \eqref{eq:onecrude} will follow by a union bound.

We first consider small degrees, at most $64pn$. By \eqref{eq:xiD} we may assume the root has degree at most $2D$, and so the contribution of such neighbors is at most $2(64)^bD(pn)^b$.

For degrees in the interval $[d,2d)$, with $64pn\le d<\eta D$, we use Fact~\ref{fact:midd}.  This bounds their contribution by $(d^{-1}\log{n})(2d)^b\, =\, 2^b(\log{n})d^{b-1}$.  We may then sum this over intervals of the form $[2^i pn, 2^{i+1}pn)$ that cover the interval $[64pn, \eta D]$. Since the number of intervals is at most $\log_2 (D/pn) = O(\log \log n)$ and each inequality holds with probability ${1 - O(n^{-2})}$, we have that with probability~${1 - o(n^{-1})}$ the total contribution is at~most
\begin{equation}
\label{eq:interm_larger} 
2^b (\log{n})\sum_{i=6}^{ \floor{\log_{2}(\eta D/pn) } }(2^{i}pn)^{b-1}\, \le\, 2^{b+1}(\log{n})(\eta D)^{b-1}  \le\, \frac{D^b}{(\log\log{n})^2}\, ,
\end{equation}
provided $n$ is sufficiently large (this is one point where we use that~$b \ge 2$ holds). 

For degrees at least $\eta D$, the required bound follows from Lemma~\ref{lem:large} with $\eps = 1$, say.  Indeed, if the sum of the degrees of high degree neighbors is at most~$2D$, then the sum of~$b$th powers of these degrees is at most~$2^b D^b$.  This completes the proof of \eqref{eq:onecrude}.

\smallskip
Finally we prove \eqref{eq:interm}. 
By Fact~\ref{fact:int}, except with probability $O(n^{-2})$, the number of $T_{1,b}$-extensions using a neighbor of degree $d\in [(1+\eps)pn, 64pn]$ is at most 
\begin{equation*}
\frac{C\log{n}}{pn}(64pn)^b\, \le\, 64^{b} C\log{n}(pn)^{b-1}\, \le\, \frac{\delta}{2} D^{b}\, ,
\end{equation*}
for all sufficiently large $n$ (again using that~$b \ge 2$ holds). 
On the other hand, the contribution to $t'_b$ of vertices with degree in the interval $[64pn,\eta D)$ was shown in \eqref{eq:interm_larger} to be $o(D^b)$. This completes the proof of estimate~\eqref{eq:interm}, and thus of Lemma~\ref{lem:crude}.
\end{proof}

\subsection{Proof of the upper bound of Theorem~\ref{thm:symm_trees}}
\label{ss:symm_upper}
In order to prove the upper bound of Theorem~\ref{thm:symm_trees}, in view of Lemma~\ref{lem:sparse_upper_GW}, it essentially (modulo checking that~$\alpha_n \gg \mu_T$) suffices to prove that, for any constant $\eps > 0$, in $\cT_{n,p}$ the number of $T_{a,b}$-extensions of the root exceeds $(1 + \eps)\alpha_n$  with probability~$o(1/n)$. 

Our argument is based on the following two  propositions, whose proofs are deferred.
Proposition~\ref{prop:largepn} deals with the cases where $pn\ge (\log{n})^{1-1/2b}$, which are contained within the case~$D^{b-1}/(pn)^{b}\to 0$, i.e., when $\alpha_n = [D(pn)^b]^a$.  
Proposition~\ref{prop:upperbound} deals with the remaining sparser and more technical cases where~$pn\le (\log{n})^{1-1/2b}$. 
\begin{proposition}
  \label{prop:largepn} 
  Let $a\ge 1$ and $b\ge 2$. Suppose that $(\log{n})^{1-1/2b}\le pn \ll \log{n}$.  For every constant $\delta>0$, with probability $1 - o(1/n)$ the number of $T_{a,b}$-extensions of the root of $\cT_{n,p}$ is at most $(1+\delta)[D(pn)^{b}]^a$.
\end{proposition}
\begin{proposition}\label{prop:upperbound} Let $a\ge 1$ and $b\ge 2$.  Suppose that $1\le pn\le (\log{n})^{1-1/2b}$. 
For every constant $\delta>0$, 
with probability $1 - o(1/n)$
the number of $T_{a,b}$-extensions of the root of $\cT_{n,p}$ is at most  
\begin{equation}
\label{eq:upperbound}
(1 + a^{ab}\delta)^2 \sup_{(x_0,\dots ,x_k)\in (1+\delta)\Lambda} \sum_{m=0}^{a} \binom{a}{m} (x_0 D)^{a-m}(pn)^{(a-m)b}\,  f_{m,b}(x_1, \dots, x_k) D^{mb}\,.
\end{equation}
\end{proposition}

\begin{proof}[Proof of the upper bound of Theorem~\ref{thm:symm_trees}]
  In view of Lemma~\ref{lem:sparse_upper_GW}, Proposition~\ref{prop:largepn} with~$\delta=\eps$ covers the case $(\log{n})^{1-1/2b}\le pn\ll \log{n}$, since, as we noted before Proposition~\ref{prop:largepn}, $\alpha_n = [D(pn)^b]^a$ holds, 
  which due to~$D \gg pn$ (which follows from~$pn \ll \log n$) satisfies~$\alpha_n \gg (pn)^{a+ab} =\Theta(\mu_T)$. 

  Now, for $1 \ll pn\le (\log{n})^{1-1/2b}$, we use Proposition~\ref{prop:upperbound}. We recall from the statement of Theorem~\ref{thm:symm_trees} that we use $L$ for the limit of $D^{b-1}/(pn)^b$ as $n$ tends to infinity (possibly~${L = \infty}$). We defer the choice of~$\delta = \delta(\eps,a,b) > 0$. By Lemma~\ref{lem:sparse_upper_GW} and Proposition~\ref{prop:upperbound} in $\Gnp$ the maximum number of $T_{a,b}$-extensions satisfies, whp, the upper bound 
\begin{align*}
M_n & \le\, (1+a^{ab}\delta)^2\sup_{(x_0,\dots ,x_k)\in (1+\delta)\Lambda} \sum_{m=0}^{a} \binom{a}{m} (x_0 D)^{a-m}(pn)^{(a-m)b}\,  f_{m,b}(x_1, \dots, x_k) D^{mb} \\ 
& \le\, (1+a^{ab}\delta)^{ab+2}\underbrace{\sup_{(x_0,\dots ,x_k)\in \Lambda} \sum_{m=0}^{a} \binom{a}{m} (x_0 D)^{a-m}(pn)^{(a-m)b}\,  f_{m,b}(x_1, \dots, x_k) D^{mb}}_{=: S_n} \,,
\end{align*}
provided that~$S_n \gg  \mu_T = \Theta((pn)^{(b+1)a})$ holds.
We choose $\delta = \delta(\eps, a, b)>0$ small enough such that the factor in front of~$S_n$ is strictly less than~$1 + \eps$. Hence it remains to check that $S_n \sim \alpha_n$ and~$S_n \gg (pn)^{a+ab}$, with $\alpha_n$ as in~\eqref{eq:thm:symm_trees:conv:asymp}. 
For this observe~that
\begin{equation}\label{eq:sn:bound}
\frac{S_n}{[D(pn)^b]^a} = \sup_{(x_0,\dots ,x_k)\in \Lambda}\sum_{m=0}^{a} \binom{a}{m} \underbrace{\biggpar{\frac{D^{b-1}}{(pn)^b}}^m}_{\to L^m} (x_0)^{a-m}\,  f_{m,b}(x_1, \dots, x_k)  \,,
\end{equation}
where, in the case~$L \in [0,\infty)$, the limit of the sum on the right-hand side of~\eqref{eq:sn:bound} equals~$F_L(x_0,\dots ,x_k)$ for~$F_L$ as in~\eqref{eq:FL_def}. 
In the case~$L \in [0,\infty)$ we thus~infer~that
\begin{equation}\label{eq:sn:bound:1}
  S_n \sim [D(pn)^b]^a \sup_{(x_0, \dots, x_k) \in \Lambda} F_L(x_0, \dots, x_k) \qquad \text{ if~$L \in [0,\infty)$,}
\end{equation}
where due to~$D \gg pn$ it is easy to see that~$S_n \asymp [D(pn)^b]^a \gg (pn)^{(b+1)a}$. 
Similarly, in the case~$L=\infty$, it is straightforward to see that the main contribution to the right-hand side of~\eqref{eq:sn:bound} comes from the case~$m=a$, so we~infer~that
\begin{equation}\label{eq:sn:bound:2}
  S_n  \sim D^{ab}\sup_{(x_1, \dots, x_k) \in \Lambda} f_{a,b}(x_1, \dots, x_k) \qquad \text{ if~$L = \infty$,}
\end{equation}
where due to~$D^{b-1}/(pn)^b \to L= \infty$ and~$D \gg pn$ it is easy to see that~$S_n \asymp D^{ab} = (D^b)^a \gg (D (pn)^b)^a \gg (pn)^{(b+1)a} \asymp \mu_T$.
Combining the estimates~\eqref{eq:sn:bound:1} and \eqref{eq:sn:bound:2} with Proposition~\ref{prop:optimizer}, 
it follows that whp $M_n \le (1 + \eps) S_n  \sim (1 + \eps) \alpha_n$,
which completes the proof of the upper bound of Theorem~\ref{thm:symm_trees}.
\end{proof}

We now give the deferred proofs of Propositions~\ref{prop:largepn} and~\ref{prop:upperbound}.
\begin{proof}[Proof of Proposition~\ref{prop:largepn}] Up to adjusting the value of~$\delta$, it clearly suffices to prove the case $a=1$, as the number of $T_{a,b}$-extensions is at most the $a$th power of the number of $T_{1,b}$-extensions.    

  Inequality~\eqref{eq:n_alpha_upp} of Lemma~\ref{lem:n_alpha} implies that with probability $1 - o(1/n)$  the root has at most $(1+\delta^2)D$ neighbors and therefore the number of $T_{1,b}$-extensions of the root which use a neighbor of degree at most $(1+\delta^2)pn$ is at most
\begin{equation}
\label{eq:smalldegs}
(1+\delta^2)D \, \left(1+\delta^2\right)^b(pn)^b\, \le\, \left(1+\delta/2\right)D(pn)^b \,,
\end{equation}
where the inequality holds because we may assume that $\delta$ is smaller than some constant (dependent only on~$b$).

Let $\cE'$ be the event that the root has more than $C\log{n}/pn$ neighbors with degree at least $(1+\delta^2)pn$, where $C=C(\delta^2)$ is the constant given by Fact~\ref{fact:int}. Fact~\ref{fact:int}  implies that $\prob{\cE'} \ll 1/n$.  

Let $I=\{2^{i}pn : i\ge 6, 2^ipn\le 2D\}$, and for each $d\in I$ let $\cE_d$ be the event that the root has more than $\log{n}/d$ neighbors with degrees at least $d$. Note that $\cup_{d \in I} [d, 2d]$ covers all degrees in the interval $[64pn, 2D]$.  
Since $|I| = O(\log n)$, by Fact~\ref{fact:midd} we have $\prob{\bigcup_{d\in I}\cE_d} \ll 1/n$.  

Finally, let $\cE''$ be the event that the root has a neighbor of degree at least $2D$. By inequality~\eqref{eq:n_alpha_upp} the expected number of such neighbors is at most $pn \cdot n^{-2 + o(1)} \ll 1/n$, hence $\prob{\cE''}\ll 1/n$.

We conclude that on the complement of $\cE' \cup \bigcup_d \cE_d \cup \cE''$ (i.e., with probability ${1 - o(1/n)}$) the number of $T_{1,b}$-extensions of the root which use a neighbor with degree at least ${(1+\delta^2)pn}$ is at most
\begin{align*}
\frac{C\log{n}}{pn} (64pn)^b\, +\, \sum_{d\in I} \frac{\log{n}}{d}(2d)^b\, & \le\, 2^{6b}C (pn)^{b-1}\log{n}\, +\, 2^{b+1}(2D)^{b-1}\log{n}\\
\justify{for sufficiently large $n$, using $(\log{n})^{1-1/2b}\le pn \ll \log n$} &\le\,\frac{\delta}{2} D(pn)^{b}\, ,
\end{align*}
which together with \eqref{eq:smalldegs} completes the proof of Proposition~\ref{prop:largepn}.
\end{proof}

\begin{proof}[Proof of Proposition~\ref{prop:upperbound}] 
  We may assume $\delta$ is at most some small constant that depends on $a$ and $b$ only. Set $\eta=(\log{\log{n}})^{-5}$ as in \eqref{eq:def_eta}.  We recall that a vertex degree is \emph{intermediate} if it lies in the interval $[(1+\delta^2)pn, \eta D]$ and \emph{large} if it is at least $\eta D$. In the rest of the proof by \emph{extension} we mean a $T_{a,b}$-extension of the root of $\cT_{n,p}$. Let us call an extension \emph{typical} if no child of the root of intermediate degree is used. Lemma~\ref{lem:crude} with $\eps = \delta^2$ implies that, with probability $1 - o(1/n)$, the total number of extensions which do use a child of the root of intermediate degree is at most 
\[
  \delta \max\{D^a(pn)^{ab},D^{ab}\}.  
\]
Since $D^a(pn)^{ab}$ is at most the supremum in \eqref{eq:upperbound} (consider $x_0=1$, $x_1 = x_2 = ... = 0$), and $D^{ab}$ is at most $a^{ab}$ times the supremum (consider $x_0 = 0, x_1=\dots =x_a=1/a$), with probability $1 - o(1/n)$ the number of nontypical extensions is at most $a^{ab}\delta$ times the supremum in \eqref{eq:upperbound}.

What remains to prove is that, with probability $1 - o(1/n)$, there are nonnegative numbers $x_0,x_1,\dots ,x_k$ such that $\sum_i x_i\le 1+\delta$ and such that the number of typical extensions is at most
\begin{equation}
\label{eq:typical}
  (1 + \delta) \sum_{m=0}^{a} \binom{a}{m} (x_0 D)^{a-m}(pn)^{(a-m)b}\,  f_{m,b}(x_1, \dots, x_k) D^{mb}\,  .
\end{equation}
Recall from Subsection \ref{ss:high} that $\neighs$ is the set of indices of the children of the root, and $H\, =\, \sum \{ \xi_i : i \in \neighs, \xi_i \ge \eta D \}$ is the sum of large degrees among the children of the root.
Having in mind a cover of the interval $[0, (1 + \delta/4)D]$ by a finite number of intervals of length $\delta D/2$, we let $\cE$ be the event that  either $|\neighs| \ge (1 + \delta/4)D$ or there exists a multiple $x_0$ of $\delta/2$ at most $1 + 3\delta/4$ such that
\begin{equation}
\label{eq:event}
|\neighs|\, \in\, [(x_0 - \delta/2) D, x_0 D] \qquad \text{and} \qquad H\, \ge\, (1-x_0+\delta)D \,.
\end{equation} 
Inequality \eqref{eq:xiD} implies that $\prob{|\neighs| \ge (1 + \delta/4)D} = o(1/n)$. Moreover, for a particular~$x_0$, event \eqref{eq:event} has probability at most $o(1/n)$ by Lemma~\ref{lem:large} (with choices $\eps = \delta/2$ and $x = x_0 - \delta/2$, noting that without loss of generality $\eps \le 1$).  Since there is a constant number of choices of $x_0$, an easy union bound gives us that $\prob{\cE}=o(1/n)$.

On the complement of the event $\cE$, there exists $x_0 \in (0, 1 + \delta]$ such that 
\begin{equation*}
  |\neighs|\, \in\, [(x_0 - \delta/2) D, x_0 D] \qquad \text{and} \qquad H\, <\, (1-x_0+\delta)D \,. 
\end{equation*}
In particular, there are nonnegative numbers $x_0, x_1,\dots ,x_k$ with $\sum_{i \ge 0}x_i\le 1+\delta$ such that $|\neighs| \le x_0 D$ and the degrees of children of the root of large degree are given by $x_1D,\dots ,x_kD$. 
A typical extension may only use neighbors of small degree (at most ${(1+\delta^2)pn}$) or large degree.  For each~$m\in \{0,\dots ,a\}$, let~$Y_m$ be the number of typical $T_{a,b}$-extensions in which~$m$ neighbors have large degree and~$a-m$ have small degree.  
Considering the information we have on neighbors of the root, we have that
\[
Y_m\, \le\, \binom{a}{m} (x_0 D)^{a-m}(1+ \delta^2)^{(a-m)b}(pn)^{(a-m)b} f_{m,b}(x_1, \dots, x_k)D^{mb}\, .
\]
Recalling that $\delta$ is at most a small constant, we have that $(1+\delta^2)^{ab}\le 1+\delta$.  
The required bound~\eqref{eq:typical} now follows as the number of typical $T_{a,b}$-extensions is precisely $\sum_{m=0}^{a}Y_m$.\end{proof}

  \subsection{Proof of the lower bound of Theorem~\ref{thm:symm_trees}}
  \label{ss:symm_lower}
In order to prove the lower bound of Theorem~\ref{thm:symm_trees}, in view of Lemma~\ref{lem:sparse_lower_GW}, it suffices to prove, for any~$n^* \sim n$ and any constant~${\eps > 0}$, that in $\cT_{n^*,p}$ the number of $T_{a,b}$-extensions of the root is at least $(1 - \eps)\alpha_n$  with probability~$\omega((\log n)^3/n)$.

Our argument is based on Lemmas~\ref{lem:symm_trees_GW_large} and~\ref{lem:symm_trees_GW_small} below, which effectively give lower bounds on the probabilities of the discussed strategies. 
Recall that $\xi_1, \dots, \xi_n$ are independent variables with distribution $\Bin(n,p)$ and $\neighs \subseteq [n]$ is a binomial subset, so that the degrees of the children of the root in $\cT_{n,p}$ are ${(\xi_i : i \in \neighs)}$.
  \begin{lemma}
    \label{lem:symm_trees_GW_large} 
    Fix real numbers $x_1, \dots, x_k > 0$ such that $x := x_1 + \dots + x_k < 1$. Let $1 \le pn\ll \log{n}$. For any constant $\eps > 0$, with probability at least $n^{-x + o(1)}$ there exist distinct $v_1,\dots, v_k \in \neighs$ with 
    \begin{equation}
      \label{eq:odeg_low}
      \xi_{v_1} \ge (x_1 - \eps)D, \; \dots, \;  \xi_{v_k} \ge (x_k - \eps)D \,.  
    \end{equation}
  \end{lemma}
  \begin{proof}
    We can assume $\eps < \min \{x_1, \dots, x_k\}$, since decreasing $\eps$ just makes the claim stronger. Cover the interval $[\eps,x]$ with intervals $[j\eps, (j+1)\eps), j \in [J]$, where $J:=\lfloor \eps^{-1}x\rfloor$. For $j \in [J]$ let $c_j = |\{ i \in [k] : x_i \in [j\eps, (j+1)\eps) \}|$. This implies that
  \begin{equation}
  \label{eq:sum_smaller_GW}
  \sum_{j=1}^{J}c_j \, j\eps\, \le \sum_{i=1}^{k}x_i\, =\, x\, .
  \end{equation}
  For $j \in [J]$, let
  \[
    V_j\, :=\, \{v \in [n]\, :\, \xi_v \in [j\eps D,(j+1)\eps D)\}\, 
  \]
  and let $\cE_j$ be the event that $|\neighs \cap V_j| = c_j$. Let $\cE := \bigcap_{j=1}^{J}\cE_j$. Note that if $\cE$ holds, then there exist $v_1, \dots, v_k \in \neighs$ satisfying \eqref{eq:odeg_low}. It remains to show that $\prob{\cE} \ge n^{-x + o(1)}$.

  We claim that there are integers $w_j = n^{1 - j\eps + o(1)}$, $j \in [J]$ such that
  \begin{equation}
    \label{eq:Wj_lower}
    \prob{|V_1| \ge w_1, \dots, |V_J| \ge w_J} \to 1\,.
  \end{equation}
  Note that $j\eps \le x < 1$ and therefore $w_j = n^{1 - j\eps + o(1)}  \to \infty$.

  To prove \eqref{eq:Wj_lower}, note that $J$ is a constant and therefore it is enough to prove that $|V_j| \ge w_j$ for a single index $j$. Note that $|V_j| \sim \Bin(n, \pi)$, where $\pi \ge n^{- j\eps + o(1)}$, by inequality \eqref{eq:n_alpha_low}. Since $\E |V_j| \ge n^{1 - j\eps + o(1)} \to \infty$, by Chebyshev's inequality, say, we have whp $|V_j| \ge (\E |V_j|)/2 \ge n^{1 - j\eps + o(1)}$, which proves \eqref{eq:Wj_lower}.
  We further condition on an arbitrary realization of the sets $V_j, j \in [J]$ that satisfies $|V_j| \ge w_j$ for every $j \in [J]$.  
  For such $V_j$ we have
  \begin{align*}
    \Pc{\cE_j}{|V_1| \ge w_1, \dots, |V_J| \ge w_J}\, &= \, \binom{|V_j|}{c_j}\, p^{c_j} (1-p)^{|V_j| - c_j} \\ 
    \justify{$|V_j|\in [w_j,n]$}&\ge \, \binom{w_j}{c_j}\, p^{c_j} (1-p)^{n} \\ 
    \justify{$w_j = n^{1 - j\eps + o(1)} \to \infty$, $pn \ll \log n$} &=\, {(pn)^{c_j}} n^{-c_j j\eps} n^{o(1)}   \\ 
    \justify{$pn \ge 1$} &\ge \, n^{-c_j j\eps + o(1)}\, .
  \end{align*}
  As events $\cE_1, \dots, \cE_J$ are conditionally independent (given realizations of $V_1, \dots, V_J$), using inequality \eqref{eq:sum_smaller_GW} we conclude 
  \begin{equation*}
    \Pc{\cE}{|V_1| \ge w_1, \dots, |V_J| \ge w_J} \ge \prod_{j \in [J]} n^{-c_j j\eps + o(1)} \ge n^{-\sum_j c_j j\eps + o(1)}\ge n^{-x + o(1)}\,,
  \end{equation*}
  which with \eqref{eq:Wj_lower} completes the proof of \refL{lem:symm_trees_GW_large}. 
  \end{proof}
  \begin{lemma}
    \label{lem:symm_trees_GW_small} 
    Fix $x_0 \in (0,1]$ and $\delta > 0$. Let $1 \ll pn\ll \log{n}$. With probability at least~$n^{-x_0 + o(1)}$ we have
    \[
      | \{ v \in \neighs : \xi_v \ge (1 - \delta) pn \} | \ge x_0(1 - \delta)D(n,p)\, .
    \]
  \end{lemma}
  \begin{proof}
    Writing $X = |\left\{ v \in \neighs : \xi_v \ge (1 - \delta)pn \right\}|$, we have $X \sim \Bin(n,p')$, where, in view of $pn \to \infty$, by Chebyshev's inequality,
  \begin{equation*}
    p' = p \cdot \prob{ \xi_1 \ge (1 - \delta)pn } \sim p\,.
  \end{equation*}
  Recall the definition of $D(n,p)$ from \eqref{eq:LLN_maxdeg_sparse}. From $p' \sim p$ and $pn \ll \log n$ it easily follows that for sufficiently large $n$ we have
  \[
    x_0 D(n,p') \ge x_0 (1 - \delta) D(n,p).
  \]
  Combining this with inequality \eqref{eq:n_alpha_low} we obtain that for sufficiently large $n$ 
  \begin{equation*}
    \prob{X \ge x_0(1 - \delta) D(n,p)} \ge \prob{X \ge x_0 D(n,p') } \ge n^{-x_0 + o(1)} \,,
  \end{equation*}
completing the proof of \refL{lem:symm_trees_GW_small}.
  \end{proof}

We are ready to prove the lower bound of Theorem~\ref{thm:symm_trees}. 
  \begin{proof}[Proof of the lower bound of Theorem~\ref{thm:symm_trees}]
Let $n^* = n^*(n) \sim n$ be given by Lemma~\ref{lem:sparse_lower_GW}. 
 Write $D = D(n,p)$ and $D^* = D(n^*,p)$. Note that $n^* \sim n$ and $1 \ll pn \ll \log n$ imply
\begin{equation}
  \label{eq:star_asymp}
  pn^* \sim pn \to \infty \quad \text{and} \quad D^* \sim D \to \infty \,.
\end{equation}
Note that 
\[
  f_T(\cT_{n^*,p}) = \sum_{\text{distinct } v_1, \ldots, v_a \in \neighsstar} \prod_{j \in [a]} (\xi_{v_j})_b \,,
\]
where $\neighsstar$ is a random subset of $[n^*]$ with every element included independently with probability $p$ and $\xi_1, \ldots, \xi_{n^*}$ are independent with distribution $\Bin(n^*,p)$.
Recall that ${D^{b-1}/(pn)^b} \to L \in [0, \infty]$.
Our goal is to show that, 
for every fixed $\eps > 0$, we have 
\begin{align}
\label{eq:prob_zero}
\prob{f_T(\cT_{n^*,p}) \ge (1 - \eps) [D(pn)^b]^a } &\gg \frac{(\log n)^{3}}{n} \quad \text{if } L= 0 \,,\\ 
\label{eq:prob_infty}
\prob{f_T(\cT_{n^*,p}) \ge (1 - \eps) D^{ab} \sup f_{a,b}(\Lambda)} &\gg \frac{(\log n)^{3}}{n} \quad \text{if } L=\infty \,, \\ %
\label{eq:prob_finite}
\prob{f_T(\cT_{n^*,p}) \ge (1 - \eps) [D(pn)^b]^a \sup F_L(\Lambda)} &\gg \frac{(\log n)^{3}}{n} \quad \text{if } L \in (0, \infty) \,,
\end{align}
where the function~$F_L$ is defined as in~\eqref{eq:FL_def} and $f_{a,b}(\Lambda)$ and $F_L(\Lambda)$ are image sets of the respective functions.  
Recall that~\eqref{eq:FL_cases} from Proposition~\ref{prop:optimizer} gives
\begin{equation*}
  [D(pn)^b]^a \sup F_L(\Lambda) \sim \begin{cases}
    [D(pn)^b]^a, \quad &L \le 1 \\
    D^{ab} \sup_{\xx \in \Lambda} f_{a,b}(\xx), \quad &L \ge C_{a,b}\,.
  \end{cases}
\end{equation*}
Hence, once estimates~\eqref{eq:prob_zero}--\eqref{eq:prob_finite} are established, the auxiliary `transfer result' Lemma~\ref{lem:sparse_lower_GW} will imply the lower bound in Theorem~\ref{thm:symm_trees}. 

\smallskip 

To prove \eqref{eq:prob_zero}, we defer the choice of the sufficiently small constant~$\delta = \delta(\eps,a,b) > 0$.
Invoking Lemma~\ref{lem:symm_trees_GW_small} with $x_0 = 1 - \delta$, it follows that, with probability at least
\[
  (n^*)^{\delta -1 + o(1)} = n^{\delta - 1 + o(1)} \gg \frac{(\log n)^3}{n}\,,
\]
we have $|\{ i \in \neighsstar : \xi_i \ge (1 - \delta) p n^* \}| \ge (1 - \delta)^2 D^*$ and therefore 
\begin{align*}
f_T(\cT_{n^*,p}) &\ge ( (1 - \delta)^2 D^* )_a (((1-\delta)pn^*)_b)^a \\
\justify{\eqref{eq:star_asymp}}  & \sim  ((1 - \delta)^2 D[( 1 - \delta) pn]^b)^{a}  \\
 &= (1 - \delta)^{a(b+2)} [D(pn)^b]^a \,,
\end{align*}
which, by choosing~$\delta$ small enough such that $(1 - \delta)^{a(b + 2)} > 1 - \eps$, establishes~\eqref{eq:prob_zero}.

\smallskip 

To prove \eqref{eq:prob_infty}, we again defer the choice of the constant~$\delta = \delta(\eps,a,b)>0$,  and use that by continuity of $f_{a,b}$ there is an integer $k \ge 1$ and a vector $(x_1, \dots, x_k) \in (0, 1)^k$ such that $\sum_{i \ge 1} x_i < 1$ and $f_{a,b}(x_1, \dots, x_k) \ge (1 - \delta)\sup f_{a,b}(\Lambda)$. Let $\delta' = \delta \min \{x_1, \dots, x_k \}$. Lemma~\ref{lem:symm_trees_GW_large} implies that, with probability at least
\[
  (n^*)^{-\sum_{i \ge 1} x_i + o(1)} = n^{- \sum_{i \ge 1} x_i + o(1)} \gg \frac{(\log n)^3}{n} \,,
\]
there exist distinct $v_1, \dots, v_k \in \neighsstar$ with $\xi_{v_i} \ge (x_i - \delta') D^* \ge (1 - \delta)x_i D^*$ and therefore 
\begin{align*}
f_T(\cT_{n^*,p}) &\ge  \sum_{\text{distinct } i_1, \ldots, i_a \in [k]} \prod_{j \in [a]} \left( (1 - \delta) x_{i_j}D^* \right)_b \\
\justify{$\delta < 1$, \eqref{eq:star_asymp}} & \sim (1 - \delta)^{ab}D^{ab} f_{a,b}(x_1, \dots, x_k)  \\
  &\ge (1 - \delta)^{ab + 1} D^{ab} \sup f_{a,b}(\Lambda) \,,
\end{align*}
which, by choosing~$\delta$ small enough such that~$(1 - \delta)^{ab + 1}> 1 - \eps$, establishes~\eqref{eq:prob_infty}.

\smallskip 

Finally we prove the `hybrid' case~\eqref{eq:prob_finite}, 
where we again defer the choice of~$\delta = \delta(\eps,a,b) > 0$. By continuity of $F_L$, there is an integer $k \ge 0$ and $\xx = (x_0, x_1, \dots, x_k) \in (0,1)^{k+1}$ such that $\sum_{i \ge 0} x_i < 1$ and $F_L(\xx) \ge (1 - \delta)\sup F_L(\Lambda)$. Consider two events: (i) there exist distinct $v_1, \dots, v_k \in \neighsstar$ with $\xi_{v_i} \ge (1 - \delta)x_i D^*$ for $i \in [k]$ and (ii) there are at least $x_0(1 - \delta)D^*$ elements $v \in \neighsstar$ such that ~${\xi_v \ge (1 - \delta) p n^*}$. By Harris' inequality, the probability that both events hold is at least the product of their probabilities. Since $\sum_{i \ge 0} x_i < 1$, invoking \refL{lem:symm_trees_GW_large} (with $\eps = \delta \min \left\{ x_1, \dots, x_k \right\}$) and~\refL{lem:symm_trees_GW_small} then implies that both events hold together with probability at least 
\[
  (n^*)^{-\sum_{i \ge 0} x_i + o(1)} = n^{-\sum_{i \ge 0} x_i + o(1)} \gg \frac{(\log n)^{3}}{n} \,
\]
and therefore random variable $f_T(\cT_{n^*,p})$ is at least
\begin{align*}
  \sum_{m = 0}^a \binom{a}{m} (x_0&(1 - \delta)D^* - k)_{a-m} (((1-\delta)pn^*)_b)^{a-m} \sum_{\text{distinct } i_1, \ldots, i_m \in [k]} \prod_{j \in [m]} \left( (1 - \delta)x_{i_j}D^* \right)_b \\
   \justify{$\delta < 1$, \eqref{eq:star_asymp}} & \sim \sum_{m = 0}^a \binom{a}{m} (x_0(1-\delta)D(( 1 - \delta) pn)^b)^{a-m} \cdot (1 - \delta)^{mb}f_{m,b}(x_1, \dots, x_k) D^{mb} \\
   &\ge  (1 - \delta)^{a(b+1)} (D(pn)^b)^a \sum_{m = 0}^a \binom{a}{m} \left( \frac{D^{b-1}}{(pn)^b} \right)^m x_0^{a-m} f_{m,b} (x_1, \ldots, x_k) \\
   \justify{${D^{b-1}}/{(pn)^b} \to L$} &\sim (1 - \delta)^{a(b + 1)} (D(pn)^b)^a F_L(x_0, \ldots, x_k) \\
   &\ge (1 - \delta)^{a(b + 1) + 1} (D(pn)^b)^a \sup F_L(\Lambda)\,.
\end{align*}
By choosing $\delta$ small enough such that $(1 - \delta)^{a(b + 1) +1} > 1 - \eps$, we conclude that~\eqref{eq:prob_finite} holds, which completes the proof of the lower bound of Theorem~\ref{thm:symm_trees}.
\end{proof}

\section{Minimum tree extension counts}\label{sec:minimum}
Throughout this paper we have studied the maximum rooted tree extension count, but it is also natural to ask about the minimum count instead.
Below we record that our methods provide the following answer to this extreme value theory question, where the `sparse' case~$pn \ll \log n$ is degenerate due to isolated~vertices in~$\Gnp$. Since this is not our main focus, we did not investigate the finer details of the phase transition of~$\min_{v \in [n]}X_v$, i.e., how it decreases from order~$\mu_T =\Theta((pn)^{e_T})$ down to~zero, 
which we leave as an interesting open problem. 
%
\begin{theorem}[Minimum for general trees]\label{thm:min_ext} 
Fix a rooted tree~$T$, with root degree~$a \ge 1$. 
If~$pqn \gg \log n$, then 
\begin{equation}
  \label{eq:LLNvariance:min}
  \frac{\mu_T  \, -\, \min_{v \in [n]}X_v}{\sigma_T \sqrt{2\log n} }\, \pto\, 1\,,
\end{equation}
where the variance~$\sigma_T^2$ is as in \refT{thm:max_ext}.
If~$pn \ll \log n$, then 
\begin{equation}
\label{eq:LLNvariance:min:sparse}
  \min_{v \in [n]}X_v = 0 \qquad \text{whp}.
\end{equation}
\end{theorem}
\begin{proof}
In the `dense' case~$pqn \gg \log n$, the minimum and maximum degree of~$\Gnp$ both deviate by the same amount from asymptotics~\eqref{eq:minmaxdegree}, modulo the sign.
Hence~\eqref{eq:LLNvariance:min} follows from~\eqref{eq:Yt_conc} in the same way as~\eqref{eq:LLNvariance} of \refT{thm:max_ext} follows from~\eqref{eq:Yt_conc} of Proposition~\ref{prop:degtoext} (see the proof in \refS{sec:proof_dense_deduction}). 

In the `sparse' case~$pn \ll \log n$ it is well-known that the minimum degree of~$\Gnp$ whp equals zero (see, e.g., ~\citeref{Bollobas01}{Exercise~3.2}), which immediately implies~\eqref{eq:LLNvariance:min:sparse}.  
\end{proof}
 
\footnotesize
\bibliographystyle{plain}
\bibliography{extreme_extensions}
\normalsize

\appendix

\section{Deferred routine proofs}
  \label{app:proofs}
In this appendix we give the deferred proofs of Propositions \ref{prop:lower_poisson} and \ref{prop:max_deg_dense} to  \ref{prop:variance} as well as Lemma~\ref{lem:n_alpha} (see Section~\ref{sec:prelim}): these are conceptually routine, but we include them for~completeness.

\begin{proof}[Proof of Proposition~\ref{prop:lower_poisson}]
Recalling~$k=(1 + \eta)pn$,
by using an approximation by a point probability of the Poisson distribution with mean $pn$ (see \citeref{Bollobas01}{eq. (1.14)}) and Stirling's formula (see \citeref{Bollobas01}{eq. (1.4)}) we obtain
\begin{align*}
	\prob{\xi \ge k} &\ge \prob{\xi = k} 
	\ge \frac{(pn)^k}{k!\e^{pn}} \cdot \e^{-((pn)^2 + k^2)/n} 
	= \frac{1}{\sqrt{2 \pi k}}\left(\frac{\e pn}{k}\right)^k \cdot \e^{-pn + o(1)} \\
	&= \exp \left( - k \log (k/(pn))  + k - pn  - \log \sqrt{2 \pi k} + o(1)\right) \\
	&= \exp \left( - pn \phi(\eta)  + O(\log k)\right) \,,
\end{align*}
which establishes inequality~\eqref{eq:bin_lower2}.
\end{proof}

\begin{proof}[Proof of \refL{lem:n_alpha}]
We start with inequality~\eqref{eq:n_alpha_upp}. 
By the simplified Chernoff bound~\eqref{eq:Chern_upper_simpler} and the assumption $\frac{\log n}{pn} \to \infty$, we obtain
\begin{equation*}
\begin{split}
  -\log \prob{\xi \ge \alpha D} 
  & \ge \frac{\alpha \log n}{\log \frac{\log n}{pn}}\log \frac{\alpha \log n}{\e pn \log \frac{\log n}{pn}}  \\
  & = \frac{\alpha \log n}{\log \frac{\log n}{pn}} \cdot \left( \log \frac{\log n}{pn} - \log \log \frac{\log n}{pn} + \log {\alpha} - 1\right) \\
  & \sim \alpha \log n \,,
\end{split}
\end{equation*}
which establishes inequality~\eqref{eq:n_alpha_upp}.

Next we prove inequality \eqref{eq:n_alpha_low}. 
Note that condition~$1 \le pn \ll \log n$ implies
\begin{equation*}
\frac{\alpha D}{pn}=\frac{\alpha \log n}{pn \log\left(\frac{\log n}{pn}\right)} \gg 1 \quad \text{and} \quad 1 \ll D \ll \log n \,.
\end{equation*}
By Proposition~\ref{prop:lower_poisson} and the asymptotics~\eqref{eq:phi_asymp}, it thus follows that
\begin{align*}
  -\log \prob{\xi \ge \alpha D}  &\le -\log \prob{\xi = \ceil{\alpha D}} \\
  &\le pn \phi\left( \frac{\ceil{\alpha D}}{pn}-1 \right)  + O(\log \ceil{\alpha D}) \\
  & \sim \alpha D \log \frac{\alpha D}{pn}\\
  &= \frac{\alpha \log n}{\log \frac{\log n}{pn}} \cdot \left( \log \frac{\log n}{pn} - \log \log \frac{\log n}{pn} + \log \alpha\right) \\
  &\sim \alpha \log n\,,
\end{align*}
establishing that~$\prob{\xi \ge \alpha D} \ge n^{-\alpha + o(1)}$.
Combining this lower bound with the upper bound~\eqref{eq:n_alpha_upp}, 
it readily follows (using that $\alpha,\eps>0$ are constants) that 
\begin{equation*}
  \prob{\alpha D \le \xi < (\alpha + \eps)D}  \ge n^{-\alpha + o(1)} - n^{-(\alpha + \eps) + o(1)} = n^{-\alpha + o(1)}\,,
\end{equation*}
completing the proof of inequality~\eqref{eq:n_alpha_low}. 
\end{proof}

The proofs of Propositions~\ref{prop:max_deg_dense} and~\ref{prop:max_deg_asymp} rely on the following maximum degree criterion,
which follows from Theorem~3.2 and Theorem~3.1 parts (i) and (ii) in \cite{Bollobas01}.
\begin{theorem}[\cite{Bollobas01}]
  \label{thm:threshold}
  Assume that~$pqn \ge 1$. For any natural number~$k = k(n)$, 
  \begin{equation*} 
    \prob{\Delta(\Gnp) \ge k} 
		\to \begin{cases}
      0 & \;\text{if $n \cdot \prob{\Bin(n-1,p) \ge k} \to 0$,}\\ 
      1& \;\text{if $n \cdot \prob{\Bin(n-1,p) \ge k} \to \infty$.}
    \end{cases}
  \end{equation*}
\end{theorem}

\begin{proof}[Proof of Proposition~\ref{prop:max_deg_dense}]
It is enough to prove~\eqref{eq:minmaxdegree} for the maximum degree~$\Delta$, since the other limit follows by considering the random variable~${(n - 1) - \delta}$, which happens to be the maximum degree in the complement of $\Gnp$ which is distributed as $\G_{n,q}$ (note that both the condition $pqn \gg \log n$ and the deno\-minator $\sqrt{2pqn \log n}$ are unchanged if we swap the roles of~$p$ and~$q$).

Under the additional assumption $pqn \gg \log^3 n$, the maximum degree limit in~\eqref{eq:minmaxdegree} follows from \citeref{Bollobas01}{Corollary~3.4}. So we can henceforth assume $\log n \ll {pqn = O(\log^3 n)}$.

To complete the proof of Proposition~\ref{prop:max_deg_dense} it remains to show, for an arbitrarily small number~${\eps > 0}$,~that 
  \begin{equation}
    \label{eq:Delta_bounds}
    pn + h_- \le \Delta \le pn + h_+ \quad \text{whp},
  \end{equation}
  where $h_\pm := \sqrt{ {(2 \pm \eps) pqn \log n} }$. For this we apply Theorem~\ref{thm:threshold}. For~${X \sim \Bin(n-1,p)}$ it thus suffices to check that the following two inequalities hold:
  \begin{align}
\label{eq:expected_degs_large}
\prob{X \ge pn + h_+} &\ll 1/n\,,\\
\label{eq:bin_point_more_n}
\prob{X \ge pn + h_-} & \gg 1/n \,.
\end{align}
 
  To prove \eqref{eq:expected_degs_large}, note that the variance $\sigma^2$ of $X$ satisfies $\sigma^2 \sim pqn \gg \log n$ and therefore $h_+ \sim  \sigma \sqrt{(2 + \eps)\log n} \ll \sigma^2$. A standard Bernstein bound (see \citeref{JLRbook}{eq.~(2.14)}) gives 
  \begin{equation*}
    \prob{X - p(n-1) \ge h_+} \le \exp \left( - \frac{h_+^2}{2(\sigma^2 + h_+/3)}  \right) = \e^{ - (1 + \eps/2 + o(1))\log n } \ll \frac{1}{n} \,,
  \end{equation*}
  which implies \eqref{eq:expected_degs_large}.


  To prove \eqref{eq:bin_point_more_n}, we apply the following inequality (see \citeref{Bollobas01}{Theorem 1.5}), which we simplify for ease of application (formally weakening it using the inequalities~${p, q \le 1}$). Namely, for any integer~$k = pn + h < n$ such that $h > 0$ and $pn \ge 1$, we~have
  \begin{equation}
\label{eq:lower_bollobas}
\prob{\Bin(n,p) = k} \ge \frac{1}{\sqrt{2\pi pqn}} \exp \left( - \frac{h^2}{2pqn} \left( 1 + \frac{h}{pqn} + \frac{2h^2 }{3(pqn)^2} + \frac{1}{h} \right) - \beta \right) \,, 
  \end{equation}
	where $\beta := 1/(12k) + 1/(12(n-k))$. 
  We apply inequality~\eqref{eq:lower_bollobas} with $k := \ceil{pn + h_- + 1}$. 
	Since $nq \ge pqn \gg h_-$, it is easy to check that $n-k \to \infty$ and therefore $\beta \to 0$. Moreover $h = k-pn = h_- + O(1) \sim \sqrt{(2 - \eps)pqn \log n}$ whence $1 \ll h \ll pqn$. Now \eqref{eq:lower_bollobas} implies
  \begin{equation*}
    \prob{X \ge pn + h_-} \ge \prob{\Bin(n,p) = k} \ge \frac{1}{\sqrt{2\pi pqn}}\exp\left( - \frac{(2 - \eps)\log n}{2}(1 + o(1)) \right) \,.
  \end{equation*}
  Recalling that we also assume $pqn = O( \log^3 n)$, we therefore obtain 
  \begin{equation*}
    \prob{X \ge pn + h_-} \ge \exp \left( -[1 - \eps/2 + o(1)] \log n + O(\log \log n)\right) \gg 1/n\,, 
  \end{equation*}
  which implies \eqref{eq:bin_point_more_n}, completing the proof of estimate~\eqref{eq:Delta_bounds} and thus of Proposition~\ref{prop:max_deg_dense}. 
\end{proof}

\begin{proof}[Proof of Proposition~\ref{prop:degree_range}]
Note that by the subsubsequence principle we can assume that $pn/(\log n) \to c \in [C_0, \infty]$. 
For $c \in [C_0, \infty)$ the claim follows from \citeref{Bollobas01}{Exercise~3.4}, 
while for $c=\infty$  standard Chernoff bounds and a union bound give 
\[
\prob{\delta \le C_1 pn \text{ or } \Delta \ge C_2 pn} \le 2n \cdot e^{-\Omega(np)} \to 0
\]
since $pn \gg \log n$.
\end{proof}

\begin{proof}[Proof of Proposition~\ref{prop:max_deg_asymp}]
First, the asymptotics~\eqref{eq:alpha_n_sim} follow directly from~\eqref{eq:phi_inv}.
Turning to the remaining proof of~\eqref{eq:LLNmaxdeg}, set~$\eps := 1/(\log \log n)$ and let 
  \[
    k_\pm := \left( 1 + \eta_\pm \right) pn, \quad \text{with} \quad \eta_\pm := \phi^{-1}\left( (1 \pm \eps)\frac{\log n}{pn} \right) \,. 
  \]
    Using asymptotics \eqref{eq:phi_inv} one can check that $1 \ll k_- = O \left(  \log n  \right)$, and therefore 
    \begin{equation*}
      1 \ll \log k_- = O(\log \log n)\,.
    \end{equation*}
    By monotonicity of $\phi^{-1}$ and assumption $1 \ll pn = O(\log n)$ we have $\eta_- = \Omega(1)$ which implies the following two sets of~inequalities:
    \begin{gather}
		\label{eq:alpha_grows}
      \alpha_n \ge \eta_-pn= \Omega(pn) \to \infty\,,\\
\label{eq:kfloorpos}
pn \le  \floor{k_-} \le k_- \,.
  \end{gather}
	Hence Proposition~\ref{prop:lower_poisson} implies that 
    \begin{align*}
      \log \prob{\Bin(n,p) \ge \floor{k_-}} &\ge -pn \phi\left( \frac{\floor{k_-}}{pn}- 1 \right) + O(\log k_-) \\
      \justify{\eqref{eq:kfloorpos}, $\phi$ is increasing on $[0,\infty)$} &\ge - pn \phi (\eta_-) + O(\log k_-) \\
      &= -(1 - \eps)\log n + O(\log \log n)\,, 
    \end{align*}
		from which it readily follows that 
    \begin{equation*}
      n \cdot \prob{\Bin(n-1,p) \ge \floor{k_-} - 1}  \ge n\cdot \prob{\Bin(n,p) \ge \floor{k_-}} \to \infty \,.
    \end{equation*}
    On the other hand, the Chernoff bound~\eqref{eq:Chern_upper} implies
    \begin{equation*}
		\begin{split}
      n \cdot \prob{\Bin(n-1,p) \ge k_+} & \le n \e^{-p(n-1)\phi(\eta_+)} \\
			& = \exp \left( \log n - (1 + \eps) \frac{p(n-1)\log n}{pn} \right) \to 0\,.
		\end{split}
    \end{equation*}
    Consequently, by Theorem~\ref{thm:threshold} it follows that whp~$\floor{k_-} - 1 \le \Delta < k_+$. 
    Note~that 
    \[
      {k_\pm - pn} = pn \phi^{-1}\left( (1 \pm \eps) \frac{\log n}{p(n-1)} \right) \sim pn \phi^{-1}\left( \frac{\log n}{pn} \right) = \alpha_n \,,
    \]
    where $\sim$ can be justified using the subsubsequence principle (which allows us to assume that both arguments of~$\phi^{-1}$ converge to~$c \in (0, \infty]$; then we use continuity of function~$\phi^{-1}$ when~$c < \infty$, and the asymptotics~\eqref{eq:phi_inv} of~$\phi^{-1}$ when~$c = \infty$). 
		In view of~\eqref{eq:alpha_grows} it also follows that~$\floor{k_-} -1 - pn\sim \alpha_n$, completing the proof of~\eqref{eq:LLNmaxdeg}.
\end{proof}

\begin{proof}[Proof of Proposition~\ref{prop:variance}]
Writing $N := (n-1)_{v_T - 1}$, let $T_1, \dots, T_N$ denote all~$T$-extensions of vertex~$v$ in~$K_n$. For convenience we will sometimes treat extensions as subgraphs of $K_n$ (remembering that several extensions correspond to the same subgraph). 
Let~$I_i$  denote the indicator random variable for the event that~${T_i \subseteq \Gnp}$. Since~$\E I_i = p^{e_T}$ and~$X_v = I_1 + \dots + I_N$, using~$e_T=v_T - 1$ it follows that the expectation of~$X_v$~satisfies
\[
  \mu_T=\E X_v = N p^{e_T} = (n-1)_{e_T} p^{e_T} = (pn)^{e_T}(1 + O(1/n)) \,.
\]

We now turn to the variance~$\sigma_T^2=\Var X_v$. 
Writing~$\Cov(I_i, I_j) = \E \left( I_i - p^{e_T} \right) \left( I_j - p^{e_T} \right)$, 
by observing that  $\Cov(I_i, I_j) = 0$ whenever $E(T_i) \cap E(T_j) = \emptyset$ it follows that
  \begin{equation}
    \label{eq:var_sum}
    \sigma_T^2  = \E \Big( \sum_i (I_i - p^{e_T}) \Big)^2 = \sum_{i,j} \Cov(I_i, I_j) = \sum_{i,j : E(T_i) \cap E(T_j) \neq \emptyset} \Cov(I_i, I_j)\,.
  \end{equation}
  We will see that the leading term in \eqref{eq:var_sum} comes from 
  the pairs $(i,j)$ for which~$T_i$ and~$T_j$ overlap in exactly one edge that is incident to the root~$v$. The number of such pairs is~$a^2 n^{2e_T - 1}(1 + o(1))$, since, having chosen $T_i$ (in one of $N \sim n^{e_T}$ ways), there are $a$ choices of the overlapping edge $e$ in $T_i$ and $a$ choices of which edge of $T$ is mapped to $e$ in $T_j$, prescribing exactly one vertex in $T_j$ and leaving $n^{e_T - 1}(1 + o(1))$ choices for the remaining vertices of $T_j$. 
  Given $k \in \{ 2, \dots, e_T \}$, the number of pairs $(i,j)$ with~${e_{T_i \cap T_j} = k}$ is~${O(n^{2e_T - k})}$, since fixing $k$ edges in $T_i$, due to the tree structure fixes at least~$k$ vertices of~$T_j$. Moreover, the number of pairs $(i,j)$ such that $T_i$ and $T_j$ share exactly one edge, but this edge is not incident to the root $v$, is $O(n^{2e_T - 2})$. 
  Since for any pair sharing $k$ edges we~have
  \begin{equation*}
    \Cov\left( I_i, I_j \right) = p^{2e_T - k} (1 - p^{k}) \asymp p^{2e_T - k}q \,,
  \end{equation*}
it follows that the variance of~$X_v$ satisfies 
  \begin{align*}
    \sigma_T^2 &= a^2 n^{2e_T - 1} p^{2e_T - 1} q \left( 1 + o(1) \right) + O\left(n^{2e_T - 2}p^{2e_T-1}q \right) + O\left( \sum_{k=2}^{e_T} (pn)^{2e_T - k} q\right) \\
    &= a^2 n^{2e_T - 1}p^{2e_T - 1}q \left( 1 + O\left(\frac{1}{pn}\right) \right) \sim \frac{a^2 \mu_T^2 q}{pn} \,,
  \end{align*}
  completing the proof of Proposition~\ref{prop:variance}. 
\end{proof}

\ifarxiv
\section{Optimizing \texorpdfstring{$f_{a,b}$}{f(a,b)}}
\label{app:optimizing-fab}

In this appendix we solve the optimization problem for $f_{a,b}$
which occurs in the limiting constant for the spherically symmetric
tree $T_{a,b}$ (see Remark~\ref{rem:optimization}).  Although each fixed-dimensional feasible set is compact,
it is not immediate that the supremum over vectors of arbitrary finite
length is attained, nor what its maximizers look like.  We shall see
that the answer is simple, but that the most natural
first guess for the size of the support is not always correct.

For the reader's convenience, recall from \eqref{eq:Lambda} that
\[
  \Lambda
  :=
  \bigcup_{n\geq 1}
  \Lambda_n,
  \qquad
  \Lambda_n
  :=
  \left\{
    (x_1,\dots,x_n)\in[0,\infty)^n:
    \sum_{i=1}^n x_i\leq 1
  \right\},
\]
and, for integers $a\geq 1$ and $b\geq 2$,
\begin{equation}
  \label{eq:fab-def}
  f_{a,b}(x_1,\dots,x_n)
  :=
  \sum_{\substack{i_1,\dots,i_a\in[n]\\
                  i_1,\dots,i_a\ \text{distinct}}}
  \prod_{j=1}^a x_{i_j}^{\,b}, 
\end{equation}
where the sum is over ordered $a$-tuples of distinct indices.  
For $k\geq a$, let $\oset{k}{a}:=k(k-1)\dots(k-a+1)$ denote the falling factorial, as usual. 

\begin{theorem}
  \label{thm:fab-complete}
  Fix integers $a\geq1$ and $b\geq2$.  Then
  \begin{equation}
    \label{eq:fab-complete}
    \max_{\xx\in\Lambda}f_{a,b}(\xx)
    =
    \max_{\substack{k\in\mathbb{Z}\\ k\geq a}}
       \frac{\oset{k}{a}}{k^{ab}}
    =
    f_{a,b}\left(
       \smash[b]{\underbrace{
          \tfrac1{k_{a,b}},\dots,\tfrac1{k_{a,b}}
       }_{k_{a,b}\ \text{entries}}}
    \right),
  \end{equation}
  where $k_{a,b}$ is the unique integer $k$ maximizing $\frac{(k)_a}{k^{ab}}$. Up to insertion of zero
  coordinates, the vector displayed in \eqref{eq:fab-complete} is the
  unique maximizer.
\end{theorem}

\begin{remark}\label{rem:kab-properties}
	The proof gives some additional properties, including that $k_{a,b}=a$ whenever~$b\geq a$. 
	Moreover, for all~$a \geq 2$, $b \ge 2$ we have that on the integers $k\geq a$ the function  $k\mapsto\oset{k}{a}/k^{ab}$ is strictly unimodal, with maximum at~$k_{a,b}$, which satisfies inequality $a\leq k_{a,b} < a+\frac{a-1}{b-1}$. 
\end{remark}

Note that $k_{a,b}=a$ need not hold when $a>b$: for example,
  \[
       \frac{\oset{7}{6}/7^{12}}{6!/6^{12}}
       =
       7\left(\frac67\right)^{12}>1,
       \qquad
       \frac{\oset{8}{6}/8^{12}}{\oset{7}{6}/7^{12}}
       =
       4\left(\frac78\right)^{12}<1,
  \]
  and hence strict unimodality gives $k_{6,2}=7$.
\smallskip

The proof of \refT{thm:fab-complete} is based on a two-coordinate ``merge-or-balance'' argument: for every
unequal positive pair, at least one of merging its mass into one
coordinate and balancing it equally between the two strictly improves
the objective function. The following lemma encapsulates this idea. 

\begin{lemma}[Two-coordinate improvement]
  \label{lem:two-coordinate}
  Let $a\geq1$, $b\geq2$, and let
  $\xx\in\Lambda_n$ have at least $a$ positive
  coordinates.  Suppose that two of its positive coordinates are
  $u\neq v$, and put $s=u+v$.  Keeping all other coordinates fixed,
  at least one of the replacements
  \[
       (u,v)\longmapsto(s,0)
       \qquad\text{and}\qquad
       (u,v)\longmapsto\left(\frac{s}{2},\frac{s}{2}\right)
  \]
  strictly increases $f_{a,b}$.
\end{lemma}

\begin{proof}
  If $a=1$, then
  $f_{1,b}(\xx)=\sum_i x_i^b$, and the first replacement
  strictly increases its value, since
  \[
       (u+v)^b>u^b+v^b.
  \]
  We may therefore assume that $a\geq2$.

  Let $i_1,i_2$ be the positions occupied by $u,v$, and put
  \[
     A
     :=
     \sum_{\substack{
       S\subseteq[n]\setminus\{i_1,i_2\}\\ |S|=a-1}}
       \prod_{\ell\in S}x_\ell^b,
     \qquad
     B
     :=
     \sum_{\substack{
       S\subseteq[n]\setminus\{i_1,i_2\}\\ |S|=a-2}}
       \prod_{\ell\in S}x_\ell^b,
  \]
  where the sum defining $B$ is $1$ when $a=2$. The terms of $f_{a,b}/a!$ which depend on $u$ and $v$ add up to
  \begin{equation}
    \label{eq:two-coordinate-objective}
       A(u^b+v^b)+B(uv)^b.
  \end{equation}
  Since $\xx$ has at least $a$ positive coordinates,
  the coordinates other than $u,v$ include at least $a-2$ positive
  ones.  Consequently,
  \begin{equation}
    \label{eq:B-positive}
       B>0.
  \end{equation}

  Keep $s=u+v$ fixed and use $r=uv$ as the variable.  Its feasible
  interval is $0\leq r\leq s^2/4$. Denoting 
  $d=(s^2-4r)^{1/2}$, write $u=(s+d)/2$ and $v=(s-d)/2$, so by the binomial formula
  \begin{equation}
    \label{eq:power-sum-expansion}
    u^b+v^b
    =
    2^{1-b}
    \sum_{\ell=0}^{\floor{b/2}}
      \binom{b}{2\ell}
      s^{b-2\ell}(s^2-4r)^\ell.
  \end{equation}
  Every summand on the right is a convex function of $r$ on
  $[0,s^2/4]$.  Hence $u^b+v^b$ is convex as a function of $r$.
  In view of \eqref{eq:B-positive}, the second term in
  \eqref{eq:two-coordinate-objective} is a positive multiple of the
  strictly convex function $r^b$.  The whole expression in
  \eqref{eq:two-coordinate-objective} is therefore strictly convex in
  $r$.

  Because $u,v>0$ and $u\neq v$, we have
  $0<r<s^2/4$.  A strictly convex function at an interior point of an
  interval is strictly smaller than at least one of the two endpoints.
  The endpoints $r=0$ and $r=s^2/4$ correspond respectively to
  $(s,0)$ and $(s/2,s/2)$, which proves the claim.
\end{proof}

\begin{proof}[Proof of \refT{thm:fab-complete}]
  We divide the proof into three steps.  The first, based on
  \refL{lem:two-coordinate}, is the key step: it reduces the
  original optimization over arbitrary vectors to a one-dimensional
  optimization over support sizes.  The second analyzes that
  one-dimensional problem, and the third records the additional
  properties of its maximizer stated in
  \refR{rem:kab-properties}.

  \medskip
  \noindent\emph{Step 1: Reduction to uniform supports.}
  For every fixed $n\geq a$, compactness of $\Lambda_n$ gives a
  maximizer of $f_{a,b}$.  The vector with $a$ coordinates equal
  to $1/a$ and all remaining coordinates equal to zero has value
  $a!/a^{ab}>0$, whereas $f_{a,b}$ vanishes on every vector with
  fewer than $a$ positive coordinates.  Thus every maximizer has at
  least $a$ positive coordinates.  Since $f_{a,b}$ is homogeneous
  of degree $ab$, a maximizer whose coordinates summed to $s<1$
  could be scaled by $1/s$ to produce a feasible vector with larger
  value.  Thus every maximizer has coordinate sum $1$.

  \refL{lem:two-coordinate} shows that no two positive
  coordinates of a maximizer can be unequal.  Thus, if its support has
  size $k$, its positive coordinates all equal $1/k$.  Define
  \begin{equation}
    \label{eq:g-def}
       g(k)
       :=
       \frac{\oset{k}{a}}{k^{ab}} = \frac{k(k-1)\dots(k-a+1)}{k^{ab}},
       \qquad k\in\mathbb{Z},\quad k\geq a.
  \end{equation}
  The value at a uniform vector of support size $k$ is
  $f_{a,b}(1/k,\dots,1/k)=g(k)$.
  It follows in particular that
  \[
      \max_{\xx\in\Lambda_n} f_{a,b}(\xx)
      =
      \max_{a\leq k\leq n}g(k).
  \]
  Moreover,
  \[
       0<g(k)\leq k^{-a(b-1)}\longrightarrow 0.
  \]
  Hence the supremum of $g(k)$ over $k\geq a$ is attained at some
  finite integer.

  \medskip
  \noindent\emph{Step 2: The one-dimensional maximization.}
  If $a=1$, then $g(k)=k^{1-b}$ is strictly decreasing, so
  $k_{1,b}=1$.  
	We may therefore henceforth assume that $a\geq2$. 
	  For this step, we extend $g$ from~\eqref{eq:g-def} to real $t>a-1$ (by the same formula),
  and define
  \begin{align*}
       \Psi(t) & :=\log\Bigl(\frac{g(t+1)}{g(t)}\Bigr) \\
			& = \log\biggl(\frac{t+1}{t-a+1}\left(\frac{t}{t+1}\right)^{ab}\biggr)\\
			& = ab\log\left(1-\frac1{t+1}\right)
       -\log\left(1-\frac a{t+1}\right) .
  \end{align*}
  Direct differentiation gives
  \begin{equation}
    \label{eq:psi-derivative}
       \Psi'(t)
       =
       \frac{
          a\bigl((b-1)t-b(a-1)\bigr)
       }{
          t(t+1)(t-a+1)
       }.
  \end{equation}
  Put
  \[
       t_0:=\frac{b(a-1)}{b-1}.
  \]
  Since the denominator in \eqref{eq:psi-derivative} is positive,
  $\Psi$ is strictly decreasing on $(a-1,t_0)$ and strictly
  increasing on $(t_0,\infty)$.  The endpoint limits are
  \[
       \lim_{t\downarrow a-1}\Psi(t)=+\infty,
       \qquad
       \lim_{t\to\infty}\Psi(t)=0.
  \]
  Since $\Psi$ is strictly increasing on $(t_0,\infty)$ and tends
  to $0$, one has $\Psi(t)<0$ for every $t\geq t_0$.  On the
  other hand, $\Psi(t)\to+\infty$ as $t\downarrow a-1$, and
  $\Psi$ is strictly decreasing on $(a-1,t_0)$.  Hence $\Psi$
  has a unique zero
  \[
       \rho_{a,b}\in(a-1,t_0).
  \]
  Thus $\Psi(t)>0$ for $a-1<t<\rho_{a,b}$, while
  $\Psi(t)<0$ for $t>\rho_{a,b}$.

  It remains to exclude a tie at an integer.  If
  $\Psi(k)=0$ for an integer $k\geq a$, then
  \[
       k^{ab}=(k+1)^{ab-1}(k-a+1).
  \]
  The right-hand side is divisible by $k+1$, whereas the left-hand
  side is coprime to $k+1$, a contradiction.  Hence
  $\rho_{a,b}\notin\mathbb{Z}$.

  Let $K=\ceil{\rho_{a,b}}$.  For every integer $k$ with
  $a\leq k<K$, one has $\Psi(k)>0$, and therefore
  $g(k+1)>g(k)$; for every integer $k\geq K$, one has
  $\Psi(k)<0$, and therefore
  $g(k+1)<g(k)$.  Consequently, $g$ is strictly unimodal on the
  integers $k\geq a$, and its unique maximizer is
  \[
       k_{a,b}=K=\ceil{\rho_{a,b}}.
  \]

  \medskip
  \noindent\emph{Step 3: Additional properties of $k_{a,b}$ and
  completion.}
  For $a\geq2$, the location of $\rho_{a,b}$ gives
  \[
       a\leq k_{a,b}
       =\ceil{\rho_{a,b}}
       <\rho_{a,b}+1
       <t_0+1
       =
       a+\frac{a-1}{b-1},
  \]
  proving the bound in \refR{rem:kab-properties}.
  If $b\geq a\geq2$, then
  \[
       a\leq k_{a,b}
       <
       a+\frac{a-1}{b-1}
       \leq a+1,
  \]
  and hence $k_{a,b}=a$, as also recorded in
  \refR{rem:kab-properties}.

  Finally, Step~1 and the strict unimodality from Step~2 prove
  \eqref{eq:fab-complete}.  They also show that every maximizing
  vector has exactly $k_{a,b}$ equal positive coordinates, which
  proves the asserted uniqueness up to insertion of zero
  coordinates.
\end{proof}
\fi

\end{document}